\documentclass[11pt]{article}
\usepackage{colordvi,epsfig,amsmath,amssymb}
\usepackage{amsthm,mathtools,graphicx,enumitem,hyphenat,float}
\usepackage[table]{xcolor}   
\usepackage[noend]{algpseudocode}
\usepackage{nameref,hyperref,cleveref}
\newcommand*{\tref}[1]{(\nameref{#1})} 
\usepackage{fullpage}
\usepackage{verbatim,tikz}
\usepackage{caption,subcaption}   
\usepackage[giveninits=true, 
maxcitenames=50, maxbibnames=50, 
doi=false,isbn=false,url=false,
backend=bibtex,
style=numeric]{biblatex}
\hypersetup{colorlinks, citecolor=blue, filecolor=black, linkcolor=blue, urlcolor=blue}

\bibliography{refs.bib}

\graphicspath{{figuras/}}  
\newcommand{\R}{\mathbb{R}}
\newcommand{\N}{\mathbb{N}}

\newcommand{\cL}{\mathcal{L}}
\DeclareMathOperator{\avg}{avg}

\newcommand{\eqdef}{\overset{\text{def}}{=}}

\newcommand{\dom}{\mbox{dom}\,}

\newcommand{\br}[1]{\left(#1\right) }
\newcommand{\Exp}{\mathbb{E} }
\newcommand{\EE}[2]{\mathbb{E}_{#1}\left[#2\right] }
\newcommand{\E}[1]{\mathbb{E}\left[#1\right] }

\newcommand{\prox}{\mbox{prox}}
\DeclareMathOperator{\sgrad}{sgrad} 
\newcommand{\norm}[1]{\lVert#1\rVert}
\newcommand{\dotprod}[1]{\left< #1\right>}

\DeclareMathOperator{\argmininn}{argmin} 
 
\newcommand{\argmin}[1]{ \underset{#1}{\argmininn} \;}

\let\oldproofname=\proofname
\renewcommand{\proofname}{\rm\bf{\oldproofname}}

\usepackage{mdframed} 
\usepackage{thmtools}

\definecolor{shadecolor}{gray}{0.95}
\declaretheoremstyle[
headfont=\normalfont\bfseries,
notefont=\mdseries, notebraces={(}{)},
bodyfont=\normalfont,
postheadspace=0.5em,
spaceabove=1pt,
mdframed={
  skipabove=8pt,
  skipbelow=8pt,
  hidealllines=true,
  backgroundcolor={shadecolor},
  innerleftmargin=4pt,
  innerrightmargin=4pt}
]{shaded}

\usepackage[colorinlistoftodos,bordercolor=orange,backgroundcolor=orange!20,linecolor=orange,textsize=scriptsize]{todonotes}

\declaretheorem[style=shaded,within=section]{definition}
\declaretheorem[style=shaded,sibling=definition]{theorem}
\declaretheorem[style=shaded,sibling=definition]{proposition}
\declaretheorem[style=shaded,sibling=definition]{assumption}
\declaretheorem[style=shaded,sibling=definition]{problem}
\declaretheorem[style=shaded,sibling=definition]{corollary}

\declaretheorem[style=shaded,sibling=definition]{lemma}
\declaretheorem[style=shaded,sibling=definition]{example}
\declaretheorem[style=shaded,sibling=definition]{remark}
\declaretheorem[style=shaded,sibling=definition]{algorithm}

\title{Handbook of Convergence Theorems \\ for (Stochastic) Gradient Methods}
\author{Guillaume Garrigos  \\ Université Paris Cité and Sorbonne Université, CNRS\\
       Laboratoire de Probabilités, Statistique et Modélisation\\
       F-75013 Paris, France\\
       \texttt{garrigos@lpsm.paris}
 \and Robert M. Gower \\ Center for Computational Mathematics, 
Flatiron Institute\\
Simons Foundation,  New York \\ \texttt{rgower@flatironinstitute.org} } %

\begin{document}
\maketitle
\begin{abstract} This is a handbook of simple proofs of the convergence of gradient and stochastic gradient descent type methods. We consider functions that are Lipschitz, smooth, convex, strongly convex, and/or Polyak-Łojasiewicz  functions. Our focus is on ``good proofs'' that are also simple.  Each section can be consulted separately. We start with proofs of gradient descent, then on stochastic variants, including minibatching and momentum. Then move on to nonsmooth problems with the subgradient method, the proximal gradient descent and their stochastic variants. 
Our focus is on global convergence rates and complexity rates. 
Some slightly less common proofs found here include that of SGD (Stochastic gradient descent) with a proximal step in~\ref{sec:sgdprox},  with momentum in Section~\ref{sec:mom},  and with mini-batching in Section~\ref{sec:mini}.  
\end{abstract}

\section{Introduction}

Here we collect our favourite convergence proofs for gradient and stochastic gradient based methods. Our focus has been on simple proofs, that are easy to copy and understand, and yet achieve the best convergence rate for the setting.

\noindent {\bf Disclaimer:} Theses notes are not proper review of the literature. Our aim is to have an easy to reference handbook.
Most of these proofs are not our work,  but rather a collection of known proofs. If you find these notes useful,  feel free to cite them,  but we kindly ask that you cite the original sources as well that are given either
 before most theorems or in the bibliographic notes at the end of each section.  
 
\subsection*{How to use these notes}

We recommend searching for the theorem you want in the table of contents, or in the in Table~\ref{tab:theorems} just below, then going directly to the section to see the proof. You can then follow the hyperlinks for the assumptions and properties backwards as needed. For example, if you want to know about the proof of Gradient Descent in the  convex and smooth case you can jump ahead to Section~\ref{sec:gradconvsmooth}. There you will find you need a property of convex function given in Lemma~\ref{L:convexity via hyperplanes}. These notes were not made to be read linearly: it would be impossibly boring.

\section*{Acknowledgements}

The authors would like to thank all the readers who pointed out errors and typos in earlier versions of this document. 
In chronological order: 
Benjamin Grimmer, 
Shuvomoy Das Gupta, 
Heinz Bauschke, 
Konstantin Mischenko, 
Shuang Song. 

\newpage
\tableofcontents

\begin{table}
\centering
\label{tab:algs}
\caption{Where to find the corresponding theorem and complexity for all the algorithms and assumptions.  
\tref{Algo:gradient descent constant stepsize} = Gradient Descent, 
\tref{Algo:Stochastic GD}  = Stochastic Gradient Descent, 
\tref{Algo:SGD minibatch} = SGD with mini-batching, 
\tref{Algo:momentum} = SGD with momentum also known as stochastic heavy ball, 
\tref{Algo:Stochastic Subgradient Descent} = Stochastic Subgradient Descent, 
\tref{Algo:SPS for SSD} = SSD with Stochastic Polyak Stepsize,
\tref{Algo:proximal gradient descent} = Proximal Gradient Descent also known as Forward-Backward, 
\tref{Algo:Proximal Stochastic GD} = Proximal Stochastic Gradient Descent,
\tref{Algo:Stoch-Proximal-Point} = Stochastic Proximal Point.
The X's are settings which are currently not covered in the handbook.}
 \begin{subtable}[h]{0.9\textwidth}
 \centering
\begin{tabular}{|c|c|c|c|} \hline
Methods&   convex   & $\mu$-strongly convex & $\mu$--PL  \\ \hline
\tref{Algo:gradient descent constant stepsize} &  Theorem~\ref{theo:convgrad}  &   Theorem~\ref{theo:gradstrconv} &   Theorem~\ref{T:GD PL smooth}    \\
\tref{Algo:Stochastic GD} &  Theorem~\ref{theo:sgdconvsmooth}  &  Theorem~\ref{theo:strcnvlin} & Theorem~\ref{theo:PLConstant} \\
\tref{Algo:SGD minibatch}  & Theorem \ref{T:SGD minibatch CV convex smooth general stepsize}  &  Theorem~\ref{theo:strcnvlinmini} & X \\
\tref{Algo:momentum} &  Theorem~\ref{theo:momentumconv} &  X  &X\\
\tref{Algo:Stochastic Subgradient Descent}  & Theorem \ref{T:SSD CV convex bounded general stepsize}& Theorem \ref{T:SSD CV strong convex}& X \\
\tref{Algo:SPS for SSD} & Theorem \ref{theo:SPS} & Theorem \ref{T:SPS strongly convex lipschitz} & X \\
\tref{Algo:proximal gradient descent} &   Theorem~\ref{theo:convproxgrad}  & Theorem~\ref{T:CV PGD strongly convex} & X
\\
\tref{Algo:Proximal Stochastic GD} & Theorem~\ref{theo:sgdproxconvex varying stepsizes} &   Theorem~\ref{theo:sgdprox strongconvex constant stepsizes} &X\\ 
\tref{Algo:Stoch-Proximal-Point} & Theorem \ref{T:SPP CV convex lipschitz general stepsize} & X & X \\
\hline
\end{tabular} 
\caption{Main results for each method}
 \label{tab:theorems}
\end{subtable} \\

\bigskip

\begin{subtable}[h]{0.9\textwidth}
\hspace{-4em}
\begin{tabular}{|c|c|c|c|} \hline
Methods &   convex   & $\mu$-strongly convex & $\mu$--PŁ  \\ \hline
\tref{Algo:gradient descent constant stepsize} &   $\frac{D^2L}{\varepsilon}$ &   $\frac{L}{\mu} \log\left(\frac{D^2}{\varepsilon}\right)$ &  $\frac{L}{\mu} \log\left(\frac{\delta_f}{\varepsilon}\right)$   \\
\tref{Algo:Stochastic GD} & $\frac{1}{\varepsilon^2} \left(D^4L_{\max}^2 + D^2\sigma_f^* \right)$  & $\max\left\{ \frac{ \sigma_f^*}{\varepsilon \mu^2}, \; \frac{L_{\max}}{\mu} \right\} \log \left( \frac{D^2}{\varepsilon} \right)$ & $ \max \left\{ \frac{\Delta^*_f}{\varepsilon}, 1 \right\}  \frac{L_{\max} L}{\mu^2} \log \left( \frac{\delta_f}{\varepsilon} \right) $ \\
\tref{Algo:SGD minibatch}  & $\frac{1}{\varepsilon^2} \left(D^4\mathcal{L}_b^2 + D^2\sigma_f^* \right)$ &  $\max\left\{ \frac{ \sigma_f^*}{\varepsilon \mu^2}, \; \frac{\mathcal{L}_b}{\mu} \right\} \log \left( \frac{D^2}{\varepsilon} \right)$ & X \\
\tref{Algo:momentum} &  $\frac{1}{\varepsilon^2} \left(D^4L_{\max}^2 + D^2\sigma_f^* \right)$ & X  &X\\
\tref{Algo:Stochastic Subgradient Descent}& $\frac{D^2G^2}{\varepsilon^2}$& $\max \left\{ \frac{G^2}{\varepsilon \mu}, 1 \right\} \log \left( \frac{D^2}{\varepsilon} \right)$ &X \\
\tref{Algo:SPS for SSD} & $\frac{D^2G^2}{\varepsilon^2}$ & $\frac{G^2}{\varepsilon \mu^2}$ & X \\
\tref{Algo:proximal gradient descent} &    $\frac{D^2L}{\varepsilon}$ &  $\frac{L}{\mu} \log\left(\frac{D^2}{\varepsilon}\right)$  & X 
\\
\tref{Algo:Proximal Stochastic GD} & $\frac{1}{\varepsilon^2}\left( D^2 \sigma_F^* + D^4 L_{\max}^2 + \tfrac{\delta_F\sigma_F^* }{L_{\max}} + \delta_F^2 \right)$  &  $\max\left\{ \frac{ \sigma_F^*}{\varepsilon \mu^2}, \; \frac{L_{\max}}{\mu} \right\} \log \left( \frac{D^2}{\varepsilon} \right)$ &X\\ 
\tref{Algo:Stoch-Proximal-Point} & $\frac{D^2G^2}{\varepsilon^2}$ & X & X \\
\hline
\end{tabular}
\caption{Table of the complexity of each algorithm. In each cell we give the number of iterations required to guarantee $\E{\norm{x^{T}-x^*}^2} \leq \varepsilon$ in the strongly convex setting, or $\E{f(\bar x^T) -\inf f} \leq \varepsilon$ in the convex and PŁ settings, where $\bar x^T$ is some average of the past iterates.
Numerical constants are omitted.
Smoothness constants are noted $L$ and $L_{\max}$, and $G$ refers to the Lipschitz constant of the functions.
Further, $\sigma_f^*$ is the gradient noise (see \cref{D:gradient solution variance}), $\Delta_f^*$ is the function noise (see \cref{D:function noise}),
$D := \norm{x^0-x^*}$, and  $ \delta_f : = f(x^0) -\inf f$. 
For composite functions  $F = f+ g$ we also note $ \delta_F : = F(x^0) -\inf F$, and $\sigma^{*}_F$ is defined in~\eqref{eq:gradient solution variance composite}. 
For the mini-batch-SGD with fixed batch size $b\in \N$, we have $\sigma_b^*$ defined in~\eqref{eq:sigmini} and $\cL_b$ defined in~\eqref{eq:Lbmini}.  }
     \label{tab:complexity}
\end{subtable}
\end{table}

\newpage
\section{Theory : Smooth functions and convexity}

\subsection{Differentiability}

\subsubsection{Notations}

\begin{definition}[Jacobian]
Let $\mathcal{F} : \mathbb{R}^d \to \mathbb{R}^p$ be differentiable, and $x \in \mathbb{R}^d$.
Then we note $D \mathcal{F}(x)$ the \textbf{Jacobian} of $\mathcal{F}$ at $x$, which is the matrix defined by its first partial derivatives:
\begin{equation*}
   \big[ D \mathcal{F}(x)\big]_{ij} = 
   \displaystyle
        \frac{\partial f_i}{\partial x_j}(x)
,
    \quad \mbox{for } i=1,\ldots  p, \ j = 1, \ldots, d,
\end{equation*}
where we write $\mathcal{F}(x) = (f_1(x), \dots, f_p(x))$. Consequently $D \mathcal{F}(x)$ is a matrix with $D \mathcal{F}(x) \in \mathbb{R}^{p \times d}$.
\end{definition}

\begin{remark}[Gradient]
If $f : \mathbb{R}^d \to \mathbb{R}$ is differentiable, then $Df(x) \in \mathbb{R}^{1 \times d}$ is a row vector, whose transpose is called the \textbf{gradient} of $f$ at $x$ : $\nabla f(x) = Df(x)^\top \in \mathbb{R}^{d \times 1}$.
\end{remark}

\begin{definition}[Hessian]
Let $f : \mathbb{R}^d \to \mathbb{R}$ be twice differentiable, and $x \in \mathbb{R}^d$.
Then we note $\nabla^2 f(x)$ the \textbf{Hessian} of $f$ at $x$, which is the matrix defined by its second-order partial derivatives:
\begin{equation*}
  \big[  \nabla^2 f(x) \big]_{i,j} = 
        \frac{\partial^2 f}{\partial x_i \partial x_j}(x), \quad \mbox{for }
i,j =1,\ldots, d.
\end{equation*}
Consequently $\nabla^2 f(x) $ is a $d \times d $ matrix.
\end{definition}

\begin{remark}[Hessian and eigenvalues]
If $f$ is twice differentiable, then its Hessian  is always a symmetric matrix (Schwarz's Theorem).
Therefore, the Hessian matrix $\nabla^2 f(x)$ admits $d$ eigenvalues (Spectral Theorem).
\end{remark}

\subsubsection{Lipschitz functions}

 \begin{definition}\label{D:Lipschitz}
 Let $\mathcal{F} : \mathbb{R}^d \to \mathbb{R}^p$, and $L>0$.
 We say that $\mathcal{F}$ is $L$-\textbf{Lipschitz} if 
 \begin{equation*}
     \text{for all $x,y \in \mathbb{R}^d$}, \quad
    \Vert \mathcal{F}(y) - \mathcal{F}(x) \Vert \leq L \Vert y-x \Vert.
\end{equation*}
\end{definition}

A differentiable function is $L$-Lipschitz if and only if its differential is uniformly bounded by $L$.

\begin{lemma}\label{L:Lipschitz via jacobian}
Let $\mathcal{F} : \mathbb{R}^d \to \mathbb{R}^p$ be differentiable, and $L>0$.
Then $\mathcal{F}$ is $L$-Lipschitz if and only if 
\begin{equation*}
    \text{for all $x \in \mathbb{R}^d$}, 
    \quad
    \Vert D \mathcal{F}(x) \Vert \leq L
\end{equation*}
\end{lemma}

\begin{proof}
$\Rightarrow$ Assume that $\mathcal{F}$ is $L$-Lipschitz.
Let $x \in \mathbb{R}^d$, and let us show that $\Vert D \mathcal{F}(x) \Vert \leq L$.
This is equivalent to show that $\Vert D \mathcal{F}(x) v \Vert \leq L$, for any $v \in \mathbb{R}^d$ such that $\Vert v \Vert=1$.
For a given $v$, the directional derivative is given by
\begin{equation*}
    D \mathcal{F}(x) v = \lim\limits_{t \downarrow 0} \frac{\mathcal{F}(x+tv) - \mathcal{F}(x)}{t}.
\end{equation*}
Taking the norm in this equality, and using our assumption that $\mathcal{F}$ is $L$-Lipschitz, we indeed obtain
\begin{equation*}
    \Vert D \mathcal{F}(x) v \Vert 
    = 
    \lim\limits_{t \downarrow 0} \frac{\Vert \mathcal{F}(x+tv) - \mathcal{F}(x) \Vert}{t}
    \leq
    \lim\limits_{t \downarrow 0} \frac{L \Vert (x+tv) - x \Vert}{t}
    = 
    \lim\limits_{t \downarrow 0} \frac{L t \Vert v \Vert}{t}
    =L.
\end{equation*}
\noindent $\Leftarrow$ Assume now that $\Vert D \mathcal{F}(z) \Vert \leq L$ for every vector $z \in \mathbb{R}^d$, and let us show that $\mathcal{F}$ is $L$-Lipschitz.
For this, fix $x,y \in \mathbb{R}^d$, and use the Mean-Value Inequality (see e.g. \cite[Theorem 17.2.2]{Gar14}) to write 
\begin{equation*}
    \Vert \mathcal{F}(y) - \mathcal{F}(x) \Vert 
    \leq \left(  \sup\limits_{z \in [x,y]} \Vert D \mathcal{F}(z) \Vert \right) \Vert y-x \Vert
    \leq
    L \Vert y-x \Vert.
\end{equation*}
\end{proof}

\subsection{Convexity}

\begin{definition}
We say that $f : \mathbb{R}^d \to \mathbb{R}\cup \{+\infty\}$ is convex if
\begin{equation}\label{eq:convoriginal}
    \text{ for all } x,y \in \mathbb{R}^d, \text{ for all } t \in ]0,1[, 
    \quad
    f(tx+(1-t)y) \leq tf(x) + (1-t)f(y).
\end{equation}
\end{definition}

The next two lemmas characterize the convexity of a function with the help of first and second-order derivatives.
These properties will be heavily used in the proofs.

\begin{lemma}\label{L:convexity via hyperplanes}
If $f : \mathbb{R}^d \to \mathbb{R}$ is convex and differentiable then,
\begin{equation}\label{eq:conv}
    \text{for all $x,y \in \mathbb{R}^d$}, \quad
	f(x) \geq  f(y) + \dotprod{\nabla f(y), x-y}.
\end{equation}
\end{lemma}

\begin{proof}
We can deduce~\eqref{eq:conv} from~\eqref{eq:convoriginal} by dividing by $t$ and re-arranging
\[\frac{f(y+ t(x-y)) - f(y)}{t} \leq f(x) -f(y).\]
Now taking the limit when $t \rightarrow 0$ gives
\[ \dotprod{\nabla f(y), x-y} \leq f(x) -f(y).\]
\end{proof}

\begin{lemma}\label{L:convexity via hessian}
Let $f : \mathbb{R}^d \to \mathbb{R}$ be convex and twice differentiable. 
Then, for all $x \in \mathbb{R}^d$, for every eigenvalue $\lambda$ of $\nabla^2 f(x)$, we have $\lambda \geq 0$.
\end{lemma}

\begin{proof}
Since $f$ is convex we can use \eqref{eq:conv} twice (permuting the roles of $x$ and $y$) and summing the resulting two inequalities, to obtain that
\begin{equation}\label{cvh1}
    \text{for all $x,y \in \mathbb{R}^d$}, \quad 
 \langle \nabla f(y) - \nabla f(x), y - x \rangle \geq 0.
\end{equation}
Now, fix $x,v \in \mathbb{R}^d$, and write
\begin{equation*}
    \langle \nabla^2 f(x)v,v     \rangle 
    =
    \langle \lim\limits_{t\to 0}
    \frac{\nabla f(x+tv) - \nabla f(x)}{t}, v \rangle
    =
    \lim\limits_{t\to 0}
    \frac{1}{t^2}
    \langle \nabla f(x+tv) - \nabla f(x), (x+tv) - x \rangle 
    \geq 0,
\end{equation*}
where the first equality follows because the gradient is a continuous function and the last inequality follows from \eqref{cvh1}.
Now we can conclude : if $\lambda$ is an eigenvalue of $\nabla^2 f(x)$, take any non zero eigenvector $v \in \mathbb{R}^d$ and write
\begin{equation*}
    \lambda \Vert v \Vert^2 = \langle \lambda v,v  \rangle 
    =
    \langle \nabla^2 f(x)v,v \rangle \geq 0.
\end{equation*}
\end{proof}

\begin{example}[Least-squares is convex]\label{Ex:least squares convex}
Let $\Phi \in \R^{n \times d}$ and $y \in \mathbb{R}^n$, and let $f(x)=\frac{1}{2}\Vert \Phi x - y \Vert^2$ be the corresponding least-squares function.
Then $f$ is convex, since $\nabla^2 f(x) \equiv \Phi^\top \Phi$ is positive semi-definite.
\end{example}

\subsection{Strong convexity}

\begin{definition}\label{D:strong convexity}
Let $f : \mathbb{R}^d \to \mathbb{R} \cup \{+\infty\}$, and $\mu >0$.
We say that $f$ is $\mu$-strongly convex if, for every $x,y \in \mathbb{R}^d$, and every $t \in ]0,1[$ we have that
\begin{equation*}
    \mu \frac{t(1-t)}{2}\Vert x-y \Vert^2 + f(tx+(1-t)y) \leq tf(x) + (1-t)f(y).
\end{equation*}
We say that $\mu$ is the strong convexity constant of $f$.
\end{definition}

The lemma below shows that it is easy to craft a strongly convex function : just add a multiple of $\Vert \cdot \Vert^2$ to a convex function.
This happens for instance when using Tikhonov regularization (a.k.a. ridge regularization) in machine learning or inverse problems.

\begin{lemma}\label{L:strong convexity is convex plus norm}
Let $f : \mathbb{R}^d \to \mathbb{R}$, and $\mu>0$.
The function $f$ is $\mu$-strongly convex if and only if there exists a convex function $g : \mathbb{R}^d \to \mathbb{R}$ such that $f(x) = g(x) + \frac{\mu}{2}\Vert x \Vert^2$.
\end{lemma}

\begin{proof}
Given $f$ and $\mu$, define $g(x) := f(x) - \frac{\mu}{2}\Vert x \Vert^2$.
We need to prove that $f$ is $\mu$-strongly convex if and only if $g$ is convex.
We start from Definition \ref{D:strong convexity} and write
(we note $z_t =(1-t)x + t y$): 
\begin{eqnarray*}
  & & \text{$f$ is $\mu$-strongly convex} \\
  \Leftrightarrow  & 
\forall t \ \forall x,y, &  
  f(z_t) + \frac{\mu}{2}t(1-t) \Vert x - y \Vert^2 \leq (1-t) f(x) + t f(y) \\
  \Leftrightarrow &  
\forall t \ \forall x,y, &  
  g(z_t) + \frac{\mu}{2} \Vert z_t \Vert^2 +  \frac{\mu}{2}t(1-t) \Vert x - y \Vert^2 
  \leq (1-t) g(x) + t g(y) + (1-t)\frac{\mu }{2} \Vert x \Vert^2 + t \frac{\mu}{2} \Vert y \Vert^2.
\end{eqnarray*}
Let us now gather all the terms multiplied by $\mu/2$ to find that
\begin{eqnarray*}
& &  \Vert z_t \Vert^2 +  t(1-t) \Vert x - y \Vert^2 - (1-t) \Vert x \Vert^2 - t  \Vert y \Vert^2 \\
= &  &
(1-t)^2\Vert x \Vert^2 + t^2 \Vert y \Vert^2 + 2 t(1-t) \langle  x,y \rangle + t(1-t) \Vert x \Vert^2 + t (1-t) \Vert y \Vert^2 - 2t(1-t) \langle  x , y \rangle \\
&-& (1- t) \Vert x \vert^2 - t \Vert y \Vert^2 \\
= & & \Vert x \Vert^2 \left( (1-t)^2 + t (1- t) - (1-t) \right) + \Vert y \Vert^2 \left( t^2 + t(1-t) - t \right)\\
= & & 0.
\end{eqnarray*}
So we see that all the terms in $\mu$ disappear, and what remains is exactly the definition for $g$ to be convex. $\qed$
\end{proof}

\begin{lemma}\label{L:strong convexity minimizers}
If $f : \mathbb{R}^d \to \mathbb{R}$ is a continuous strongly convex function, 
then $f$ admits a unique minimizer.
\end{lemma}

\begin{proof}
See \cite[Corollary 2.20]{Pey}.
\end{proof}

Now we present some useful variational inequalities satisfied by strongly convex functions.

\begin{lemma}\label{L:strong convexity differentiable hyperplans}
If $f : \mathbb{R}^d \to \mathbb{R}$ is $\mu$-strongly convex and differentiable function  then 
\begin{equation} \label{eq:strconv}
 \text{for all $x,y \in \mathbb{R}^d$}, \quad
 f(y) \geq  f(x) + \dotprod{\nabla f(x), y-x} + \frac{\mu}{2} \norm{y-x}^2.
\end{equation}
\end{lemma}

\begin{proof}
Define $g(x) := f(x) - \frac{\mu}{2}\Vert x \Vert^2$.
According to Lemma \ref{L:strong convexity is convex plus norm}, $g$ is convex.
It is also clearly differentiable by definition.
According to the sum rule, we have $\nabla f(x) = \nabla g(x) + \mu x$.
Therefore we can use the convexity of $g$ with Definition \ref{D:subdifferential convex} to write
\begin{equation*}
    f(y) - f(x) - \langle \nabla f(x), y-x \rangle 
    \geq 
    \frac{\mu}{2}\Vert y \Vert^2 - \frac{\mu}{2}\Vert x \Vert^2 - \langle \mu x, y-x \rangle 
    =
    \frac{\mu}{2}\Vert y - x \Vert^2.
\end{equation*}
\end{proof}

\begin{lemma}\label{L:strong convexity hessian} 
Let $f:\mathbb{R}^d \to \mathbb{R}$ be a twice differentiable $\mu$-strongly convex function.
Then, for all $x \in \mathbb{R}^d$, for every eigenvalue $\lambda$ of $\nabla^2 f(x)$, we have $\lambda \geq \mu$.
\end{lemma}

\begin{proof}
Define $g(x) := f(x) - \frac{\mu}{2}\Vert x \Vert^2$, which is convex according to Lemma \ref{L:strong convexity is convex plus norm}.
It is also twice differentiable, by definition, and we have $\nabla^2 f(x) = \nabla^2 g(x) + \mu Id$.
So the eigenvalues of $\nabla^2 f(x)$ are equal to the ones of $\nabla^2 g(x)$ plus $\mu$.
We can conclude by using Lemma \ref{L:convexity via hessian}.
\end{proof}

\begin{example}[Least-squares and strong convexity]\label{Ex:least squares strongly convex}
Let $f$ be a least-squares function as in~\Cref{Ex:least squares convex}.
Then $f$ is strongly convex if and only if $\Phi$ is injective.
In this case, the strong convexity constant $\mu$ is $\lambda_{\min}(\Phi^\top \Phi)$, the smallest eigenvalue of $\Phi^\top \Phi$.
\end{example}

\subsection{Polyak-Łojasiewicz}

\begin{definition}[Polyak-Łojasiewicz]\label{D:polyak lojasiewicz}
Let $f : \mathbb{R}^d \to \mathbb{R}$ be differentiable, and $\mu >0$.
We say that $f$ is $\mu$-\textbf{Polyak-Łojasiewicz} if it is bounded from below, and if for all $x \in \mathbb{R}^d$
\begin{equation}\label{eq:PL}
f(x)-\inf f \leq \frac{1}{2 \mu} \|\nabla f(x)\|^2. 
\end{equation}
We just say that $f$ is {Polyak-Łojasiewicz} (PŁ for short) if there exists $\mu>0$ such that $f$ is $\mu$-{Polyak-Łojasiewicz}.
\end{definition}

The Polyak-Łojasiewicz property is weaker than strong convexity,  as we see next.

\begin{lemma}\label{L:strong convexity Polyak Lojasiewicz}
Let $f : \mathbb{R}^d \to \mathbb{R}$ be differentiable, and $\mu >0$.
If $f$ is $\mu$-strongly convex, then $f$ is $\mu$-Polyak-Łojasiewicz.
\end{lemma}
\begin{proof}
Let $x^*$ be a minimizer of $f$ (see Lemma \ref{L:strong convexity minimizers}), such that  $f(x^*) = \inf f.$
Multiplying~\eqref{eq:strconv} by minus one and substituting $y = x^*$ as the minimizer, we have that
 \begin{eqnarray*}
  f(x) - f(x^*) &\leq &   \dotprod{\nabla f(x), x-x^*} - \frac{\mu}{2} \norm{x^*-x}^2 \\
  & =& -\frac{1}{2}\norm{\sqrt{\mu}(x-x^*)-\frac{1}{\sqrt{\mu}}\nabla f(x)}^2 + \frac{1}{2\mu}\norm{\nabla f(x)}^2 \\
  &\leq & \frac{1}{2\mu}\norm{\nabla f(x)}^2.
 \end{eqnarray*}
\end{proof}

It is important to note that the Polyak-Łojasiewicz property can hold without strong convexity or even convexity, as illustrated in the next examples.

\begin{example}[Least-squares is PŁ]\label{Ex:least squares PL}
Let $f$ be a least-squares function as in Example \ref{Ex:least squares convex}.
Then it is a simple exercise to show that $f$ is PŁ, and that the PŁ constant $\mu$ is $\lambda_{\min}^*(\Phi^\top \Phi)$, the smallest \emph{nonzero} eigenvalue of $\Phi^\top \Phi$ (see e.g. \cite[Example 3.7]{GarRosVil22}).
\end{example}

\begin{example}[Nonconvex PŁ functions]\label{Ex:PL nonconvex}~~
\begin{itemize}
    \item Let $f(t) = t^2 + 3\sin(t)^2$.
It is an exercise to verify that $f$ is PŁ, while not being convex (see Lemma \ref{L:PL nonconvex example} for more details).
    \item If $\Omega \subset \mathbb{R}^d$ is a closed set and $f(x) = {\rm{dist}}(x; \Omega)^2$ is the squared distance function to this set, then it can be shown that $f$ is PŁ. See Figure \ref{F:PL nonconvex} for an example, and \cite{Gar23} for more details.
\end{itemize}
\end{example}

\begin{figure}[h!]
    \centering
    \includegraphics[width=0.5\linewidth]{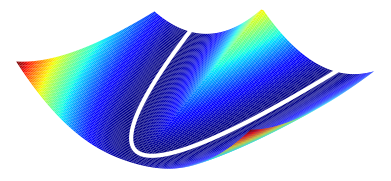}
    \caption{Graph of a PŁ function $f: \mathbb{R}^2 \to \mathbb{R}$. Note that the function is not convex, but that the only critical points are the global minimizers (displayed as a white curve).}
    \label{F:PL nonconvex}
\end{figure}

\begin{example}[PŁ for nonlinear models]\label{Ex:neural network PL}
Let $f(x) = \frac{1}{2}\Vert \Phi(x) - y \Vert^2$, where $\Phi : \mathbb{R}^d \to \mathbb{R}^n$ is differentiable. 
Then $f$ is PŁ if $D\Phi^\top(x)$ is uniformly injective:
\begin{equation}\label{hyp:PL for elliptic NN}
    \text{there exists $\mu>0$ such that for all $x \in \mathbb{R}^d$,} \quad \lambda_{\min}(D\Phi(x) D\Phi(x)^\top) \geq \mu.
\end{equation}
Indeed it suffices to write
\begin{equation*}
    \Vert \nabla f(x) \Vert^2
    =
    \Vert D\Phi(x)^\top(\Phi(x) - y) \Vert^2
    \geq
    \mu \Vert \Phi(x) - y \Vert^2
    = 2 \mu f(x)
    \geq
    2 \mu(f(x) - \inf f).
\end{equation*}
Note that assumption \eqref{hyp:PL for elliptic NN} requires $d \geq n$, which holds if $\Phi$ represents an overparametrized neural network.
For more refined arguments, including less naive assumptions and exploiting the neural network structure of $\Phi$, see \cite{belkin2020}.
\end{example}

One must keep in mind that the PŁ property is rather strong, as it is a \textit{global} property and requires the following to be true, which is typical of convexity.

\begin{lemma}\label{L:PL function global minimizers}
Let $f : \mathbb{R}^d \to \mathbb{R}$ be a differentiable PŁ function. Then $x^* \in {\rm{argmin}}~f$ if and only if $\nabla f(x^*)=0$.
\end{lemma}

\begin{proof}
Immediate from plugging in $x = x^*$ in \eqref{eq:PL}.
\end{proof}

\begin{remark}[Local Łojasiewicz inequalities]
In this document we focus only on the Polyak-Łojasiewicz inequality, for simplicity.
Though there exists a much larger family of Łojasiewicz inequalities, which by and large cover most functions used in practice.
\begin{itemize}
    \item The inequality can be more \emph{local}. For instance by requiring that  \eqref{eq:PL}  holds only on some subset $\Omega \subset \mathbb{R}^d$  instead of the whole $\mathbb{R}^d$.
    For instance, logistic functions typically verify \eqref{eq:PL} on every bounded set, but not on the whole space. 
    The same can be said about the empirical risk associated to wide enough neural networks  \cite{belkin2020}.
    \item While PŁ describes functions that grow like $x \mapsto \frac{\mu}{2}\Vert x \Vert^2$, there are \emph{$p$-Łojasiewicz} inequalities describing functions that grow like $x \mapsto \frac{\mu^{p-1}}{p}\Vert x \Vert^p$ and satisfy $f(x) - \inf f \leq \frac{1}{q \mu}\Vert \nabla f(x) \Vert^q$ on some set $\Omega$, with $\frac{1}{p} + \frac{1}{q} =1$.
    \item The inequality can be even more local, by dropping the property that every critical point is a global minimum. For this we do not look at the growth of $f(x) - \inf f$, but of $f(x) - f(x^*)$ instead, where $x^* \in \mathbb{R}^d$ is a critical point of interest. This can be written as
    \begin{equation}\label{eq:Lojasiewicz}
    \text{ for all } x \in \Omega, \quad f(x) - f(x^*) \leq \frac{1}{q \mu}\Vert \nabla f(x) \Vert^q, \quad \text{ where } \quad \frac{1}{p} + \frac{1}{q} =1.
    \end{equation}
\end{itemize}
A famous result \cite[Corollary 16]{BolDanLewShi07}  shows that any semi-algebraic function (e.g. sums and products of polynomials by part functions) verifies \eqref{eq:Lojasiewicz} at every $x^* \in \mathbb{R}^d$ for some $p \geq 1$, $\mu > 0$, and $\Omega$ being an appropriate neighbourhood of $x^*$.
This framework includes for instance quadratic losses evaluating a Neural Network with ReLU as activations.
\end{remark}

\subsection{Smoothness}
\begin{definition}\label{D:smooth}
Let $f : \mathbb{R}^d \to \mathbb{R}$, and $L>0$.
We say that $f$ is $L$-\textbf{smooth} if it is differentiable and if $\nabla f : \mathbb{R}^d \to \mathbb{R}^d$ is $L$-Lipschitz:
\begin{eqnarray}\label{eq:smoothness}
\text{for all $x,y \in \mathbb{R}^d$,} \quad 
\norm{\nabla f(x)- \nabla f(y)} &\leq & L \norm{x-y}.
 \end{eqnarray}
\end{definition}

\subsubsection{Smoothness and nonconvexity}

As for the convexity (and strong convexity), we give two characterizations of the smoothness by means of first and second order derivatives.

\begin{lemma}\label{L:smooth upper quadratic}
If $f: \R^d \to \R$ is  $L$-smooth then
 \begin{equation} 
\label{eq:smoothnessfunc}
\text{for all $x,y \in \mathbb{R}^d$,} \quad 
 f(y)  \leq  f(x) +\langle \nabla f(x), y-x \rangle +\frac{L}{2} \norm{y-x}^2.
 \end{equation}
 \end{lemma}

\begin{proof}
Let $x,y \in \mathbb{R}^d$ be fixed.
Let $\phi(t) := f(x+t(y-x))$.
Using the Fundamental Theorem of Calculus on $\phi$, we can write that
 \begin{eqnarray*}
f(y) & =& f(x) + \int_{t=0}^1 \dotprod{ \nabla f(x+t(y-x)) , y-x } dt.  \\
 & =& f(x) + \dotprod{\nabla f(x), x-y} + \int_{t=0}^1 \dotprod{ \nabla f(x+t(y-x)) - \nabla f(x), y-x } dt.  \\
 & \leq & f(x) + \dotprod{\nabla f(x), y-x}  +
 \int_{t=0}^1 \norm{\nabla f(x+t(y-x)) - \nabla f(x)} \norm{ y-x } dt \\
 & \overset{\eqref{eq:smoothness}}{ \leq} &
 f(x) + \dotprod{\nabla f(x), y-x} +  \int_{t=0}^1 Lt\norm{y-x}^2 \ dt \\
 & { \leq} &
 f(x) + \dotprod{\nabla f(x), y-x} + \frac{L}{2}  \norm{y-x}^2.
 \end{eqnarray*}
\end{proof}

\begin{lemma}\label{L:smooth via hessian}
Let $f:\mathbb{R}^d \to \mathbb{R}$ be a twice differentiable $L$-smooth function.
Then, for all $x \in \mathbb{R}^d$, for every eigenvalue $\lambda$ of $\nabla^2 f(x)$, we have $\vert \lambda \vert \leq L$.
\end{lemma}

\begin{proof}
Use Lemma \ref{L:Lipschitz via jacobian} with $\mathcal{F}=\nabla f$, together with the fact that $D(\nabla f)(x) = \nabla^2 f(x)$.
We obtain that, for all $x \in \mathbb{R}^d$, we have $\Vert \nabla^2 f(x) \Vert \leq L$.
Therefore, for every eigenvalue $\lambda$ of $\nabla^2 f(x)$, we can write for a nonzero eigenvector $v \in \mathbb{R}^d$ that
\begin{equation*}
    \vert \lambda \vert \Vert v \Vert=
    \Vert \lambda v \Vert=
    \Vert \nabla^2 f(x) v \Vert
    \leq
    \Vert \nabla^2 f(x) \Vert \Vert v \Vert
    \leq
    L \Vert v \Vert.
\end{equation*}
The conclusion follows after dividing by $\Vert v \Vert \neq 0$.
\end{proof}

\begin{remark}\label{R:mu and L}
From Lemmas \ref{L:smooth via hessian} and \ref{L:strong convexity hessian}  we see that if a function is $L$-smooth and $\mu$-strongly convex then $\mu \leq L$.
\end{remark}

Some direct consequences of the smoothness are given in the following lemma.
You can compare \eqref{eq:inversePL} with Lemma \ref{L:strong convexity Polyak Lojasiewicz}.

\begin{lemma}\label{L:smooth descent} If $f$ is $L$--smooth and $\gamma>0$ then
\begin{equation}\label{eq:gradnormupx}
\text{ for all $x,y \in \mathbb{R}^d$}, \quad
f(x - \gamma \nabla f(x))-f(x) \leq  - \gamma \left( 1 - \frac{\gamma L}{2} \right)\norm{\nabla f(x)}^2.
\end{equation}
If moreover $\inf f > - \infty$, then
\begin{equation}\label{eq:inversePL}
\text{ for all $x \in \mathbb{R}^d$},\quad
 \frac{1}{2L} \norm{\nabla f(x)}^2\leq f(x) - \inf f.
\end{equation}
\end{lemma}

\begin{proof}
Inequality~\eqref{eq:gradnormupx} follows by inserting $y = x - \gamma \nabla f(x)$ in \eqref{eq:smoothnessfunc} since
\begin{eqnarray}
 f(x - \gamma \nabla f(x)) &\leq & f(x) -  \gamma \langle \nabla f(x),   \nabla f(x) \rangle + \frac{L}{2} \norm{ \gamma \nabla f(x)}^2\nonumber\\
 & =& f(x) - \gamma \left( 1 - \frac{\gamma L}{2} \right)\norm{\nabla f(x)}^2.\nonumber
 \end{eqnarray}
 Assume now $\inf f > - \infty$.
By using~\eqref{eq:gradnormupx} with $\gamma = 1/L$, we get~\eqref{eq:inversePL} up to a multiplication by $-1$ :
 \begin{equation*}
 \inf f - f(x)  \leq f(x - \tfrac{1}{L} \nabla f(x))-f(x) \leq  -  \frac{1}{2L}\norm{\nabla f(x)}^2.
\end{equation*}
\end{proof}

\subsubsection{Smoothness and Convexity}

There are many problems in optimization where the function is both smooth and convex. 
Such functions enjoy properties which are strictly better than a simple combination of their convex and smooth properties.

\begin{lemma}\label{lem:convandsmooth}
If $f: \mathbb{R}^d \to \mathbb{R}$ is convex and $L$-smooth,
then \text{for all $x,y \in \mathbb{R}^d$} we have that
\begin{eqnarray}
\frac{1}{2L}\norm{\nabla f(y)-\nabla f(x)}^2  & \leq & f(y) - f(x) - \dotprod{\nabla f(x),y-x},
\label{eq:convandsmooth}
\\ \frac{1}{L} \norm{\nabla f(x) - \nabla f(y)}^2 &\leq & \dotprod{\nabla f(y)-\nabla f(x),y-x} \quad (\mbox{Co-coercivity})\label{eq:coco}
\end{eqnarray}
\end{lemma}

\begin{proof}
To prove~\eqref{eq:convandsmooth}, fix $x,y \in \mathbb{R}^d$ and start by using the convexity and the smoothness of $f$ to write, for every $z \in \mathbb{R}^d$,
\begin{eqnarray*}
f(x) -f(y) & = & f(x) -f(z)+f(z) - f(y)\\ &\overset{\eqref{eq:conv}+\eqref{eq:smoothnessfunc} } \leq &
\dotprod{\nabla f(x), x-z} + \dotprod{\nabla f(y), z-y} +\frac{L}{2}\norm{z-y}^2.
\end{eqnarray*}
To get the tightest upper bound on the right hand side, we can minimize the right hand side with respect to $z$, which gives
\begin{equation*}
z = y - \frac{1}{L}(\nabla f(y) -\nabla f(x)).
\end{equation*}
Substituting this $z$ in gives, after reorganizing the terms:
\begin{eqnarray*}
f(x) -f(y) & \leq &
\dotprod{\nabla f(x), x-z} + \dotprod{\nabla f(y), z-y} +\frac{L}{2}\norm{z-y}^2. \\
& =&\dotprod{\nabla f(x), x-y}   - \frac{1}{L}\norm{\nabla f(y)-\nabla f(x)}^2 +\frac{1}{2L}\norm{\nabla f(y) -\nabla f(x)}^2\\
&= & \dotprod{\nabla f(x), x-y}   - \frac{1}{2L}\norm{\nabla f(y)-\nabla f(x)}^2.
\end{eqnarray*}
This proves~\eqref{eq:convandsmooth}.
To obtain~\eqref{eq:coco}, apply~\eqref{eq:convandsmooth} twice by interchanging the roles of $x$ and $y$ 
\begin{eqnarray*}
\frac{1}{2L}\norm{\nabla f(y)-\nabla f(x)}^2  & \leq & f(y) - f(x) - \dotprod{\nabla f(x),y-x}, \\
\frac{1}{2L}\norm{\nabla f(x)-\nabla f(y)}^2  & \leq & f(x) - f(y) - \dotprod{\nabla f(y),x-y},
\end{eqnarray*}
and sum those two inequalities.
\end{proof}

\section{Gradient Descent}

\begin{problem}[Differentiable Function]\label{Pb:differentiable function}
We want to minimize a differentiable function $f : \mathbb{R}^d \to \mathbb{R}$.
We require that the problem is well-posed, in the sense that ${\rm{argmin}}~f \neq \emptyset$.
\end{problem}

\begin{algorithm}[GD]\label{Algo:gradient descent constant stepsize}
Let $x^0 \in \mathbb{R}^d$, and let $\gamma>0$ be a step size.
The \textbf{Gradient Descent (GD)} algorithm defines a sequence $(x^t)_{t \in \mathbb{N}}$ satisfying
\begin{align*}
    x^{t+1} & = x^t - \gamma\nabla f(x^t).
\end{align*}
\end{algorithm}

\begin{remark}[Vocabulary]
Stepsizes are often called \textit{learning rates} in the machine learning community.
\end{remark}

We will now prove that the iterates of~\tref{Algo:gradient descent constant stepsize} converge. In Theorem~\ref{theo:convgrad} we will prove sublinear convergence under the assumption that $f$ is convex. In Theorem~\ref{theo:gradstrconv} we will prove linear convergence (a faster form of convergence) under the stronger assumption that $f$ is $\mu$--strongly convex.

\subsection{Convergence for convex and smooth functions}
\label{sec:gradconvsmooth}

\begin{theorem}\label{theo:convgrad}
Consider the Problem \tref{Pb:differentiable function} and assume that $f$ is convex and $L$-smooth, for some $L>0$.
Let $(x^t)_{t \in \mathbb{N}}$ be the sequence of iterates generated by the \tref{Algo:gradient descent constant stepsize} algorithm, with a stepsize satisfying $0 < \gamma\leq \frac{1}{L}$. Then, for all $x^* \in {\rm{argmin}}~f$, for all $t \in \mathbb{N}$ we have that
\begin{equation*}
f(x^t)-\inf f \leq \frac{\norm{x^0-x^*} ^2}{2 \gamma t}.
\end{equation*}
\end{theorem}
For this theorem we give two proofs.  The first proof uses an energy function, that we will also use later on.  The second proof is a direct proof taken from~\cite{bubeck-book}.

\begin{proof}[\rm\bf Proof of Theorem \ref{theo:convgrad} with Lyapunov arguments]
Let $x^* \in {\rm{argmin}}~f$ be any minmizer of $f$.
First, we will show that $f(x^t)$ is decreasing.
Indeed we know from \eqref{eq:gradnormupx}, and from our assumption $\gamma L \leq 1$, that
\begin{equation}
\label{GD values decreasing}
    f(x^{t+1}) - f(x^t)
    \leq
    - \gamma (1 - \frac{\gamma L}{2}) \Vert \nabla f(x^t) \Vert^2
    \leq 0.
\end{equation}
Second, we will show that $ \Vert x^t - x^* \Vert^2$  is also decreasing.
For this we expand the squares to write
\begin{align}
    \frac{1}{2 \gamma} \Vert x^{t+1} - x^* \Vert^2 - \frac{1}{2 \gamma} \Vert x^t - x^* \Vert^2
    & =
    \frac{-1}{2 \gamma}\Vert x^{t+1} - x^t \Vert^2
    - \langle \nabla f(x^t), x^{t+1} - x^* \rangle 
   \nonumber \\
    &=
    \frac{-1}{2 \gamma}\Vert x^{t+1} - x^t \Vert^2
    - \langle \nabla f(x^t), x^{t+1} - x^t \rangle
    +
    \langle \nabla f(x^t), x^* - x^t \rangle.\label{eq:tempoxzp8ne8onzz4}
\end{align}
Now to bound the right hand side 
we use the convexity of $f$ and \eqref{eq:conv} to write
\begin{equation*}
    \langle \nabla f(x^t), x^* - x^t \rangle \leq f(x^*) - f(x^t) = \inf f - f(x^t).
\end{equation*}
To bound the other inner product we use the smoothness of $L$ and \eqref{eq:smoothnessfunc} which gives
\begin{equation*}
    - \langle \nabla f(x^t), x^{t+1} - x^t \rangle
    \leq
    \frac{L}{2}\Vert x^{t+1} - x^t \Vert^2 + f(x^t) - f(x^{t+1}).
\end{equation*}
By using the two above inequalities in~\eqref{eq:tempoxzp8ne8onzz4} we obtain 
\begin{align}
    \frac{1}{2 \gamma} \Vert x^{t+1} - x^* \Vert^2 - \frac{1}{2 \gamma} \Vert x^t - x^* \Vert^2
    &\leq 
    \frac{\gamma L-1}{2 \gamma}\Vert x^{t+1} - x^t \Vert^2
    - (f(x^{t+1}) - \inf f), \notag\\
    & \leq - (f(x^{t+1}) - \inf f).\label{GD iterates decreasing}
\end{align}

Let us now combine the two positive decreasing quantities $f(x^t)-\inf f$ and $\frac{1}{2 \gamma} \Vert x^t - x^* \Vert^2$, and introduce the following Lyapunov energy, for all $t \in \mathbb{N}$:
\begin{equation*}
    E_t := 
    \frac{1}{2 \gamma} \Vert x^t - x^* \Vert^2 + t(f(x^t) - \inf f).
\end{equation*}
We want to show that it is decreasing with time.
For this we start by writing
\begin{eqnarray}
E_{t+1} - E_t 
& =& \notag
(t+1)(f(x^{t+1}) - \inf f) - t(f(x^t) - \inf f) + \frac{1}{2 \gamma} \Vert x^{t+1} - x^* \Vert^2 - \frac{1}{2 \gamma} \Vert x^t - x^* \Vert^2 \\
& =& \label{GD energy develop 1}
f(x^{t+1}) - \inf f + t(f(x^{t+1}) - f(x^t)) + \frac{1}{2 \gamma} \Vert x^{t+1} - x^* \Vert^2 - \frac{1}{2 \gamma} \Vert x^t - x^* \Vert^2.
\end{eqnarray}

Combining now \eqref{GD energy develop 1}, \eqref{GD values decreasing} and \eqref{GD iterates decreasing}, we finally obtain (after cancelling terms) that
\begin{align*}
    E_{t+1} - E_t 
    & \leq  f(x^{t+1}) - \inf f +  \frac{1}{2 \gamma} \Vert x^{t+1} - x^* \Vert^2 - \frac{1}{2 \gamma} \Vert x^t - x^* \Vert^2  & \mbox{using~\eqref{GD values decreasing}}\\
        & \leq  f(x^{t+1}) - \inf f 
    - (f(x^{t+1}) - \inf f)   & \mbox{using~\eqref{GD iterates decreasing}}\\ 
    & = 0.
\end{align*}
Thus $E_t$ is decreasing.
Therefore we can write that
\begin{equation*}
    t(f(x^t) - \inf f) \leq E_t \leq E_0 = \frac{1}{2\gamma} \Vert x^0 - x^* \Vert^2,
\end{equation*}
and the conclusion follows after dividing by $t$. 
\end{proof}
\begin{proof}[\rm\bf Proof of Theorem \ref{theo:convgrad} with direct arguments]
Let $f$ be convex and $L$--smooth.  This proof will only hold for $\gamma =1/L$. It follows that
\begin{eqnarray}
\norm{x^{t+1}-x^*}^2 &\overset{\gamma =1/L}{=}& \norm{x^t -x^*- \tfrac{1}{L} \nabla f(x^t)}^2 \nonumber \\
& =& \norm{x^{t}-x^*}^2-2\tfrac{1}{L} \dotprod{x^t -x^*, \nabla f(x^t)} + \tfrac{1}{L^2}\norm{\nabla f(x^t)}^2 \nonumber \\
  & \overset{\eqref{eq:coco}}{\leq} & \norm{x^{t}-x^*}^2-\tfrac{1}{L^2}\norm{\nabla f(x^t)}^2.\label{eq:decressse}
\end{eqnarray}
Thus  $\norm{x^{t}-x^*}^2$ is a decreasing sequence in $t$, and thus consequently
\begin{equation}\label{eq:decreasings9ks4}
\norm{x^{t}-x^*} \leq \norm{x^{0}-x^*}.
\end{equation}
 Calling upon~\eqref{eq:gradnormupx} and subtracting $f(x^*)$ from both sides gives
\begin{eqnarray}
f(x^{t+1}) -f(x^*) & \leq &  f(x^t)-f(x^*)-  \frac{1}{2L}\norm{\nabla f(x^t)}^2.\label{eq:ao9822nhf}
\end{eqnarray}
Applying convexity we have that
\begin{eqnarray}
 f(x^t)-f(x^*) & \leq & \dotprod{\nabla f(x^t), x^t-x^*} \nonumber\\
 & \leq & \norm{\nabla f(x^t)} \norm{x^t-x^*} \overset{\eqref{eq:decreasings9ks4}}{\leq} \norm{\nabla f(x^t)} \norm{x^0-x^*} .\label{eq:pmw0u2nj}
\end{eqnarray}
Suppose now that $x^0 \neq x^*,$ otherwise the proof is finished.
Isolating $\norm{\nabla f(x^t)}$ in the above and
inserting in~\eqref{eq:ao9822nhf} gives
\begin{eqnarray}
f(x^{t+1}) -f(x^*) & \overset{\eqref{eq:ao9822nhf}+\eqref{eq:pmw0u2nj}}{\leq} &  f(x^t)-f(x^*)-  \underbrace{\frac{1}{2L}\frac{1}{\norm{x^0-x^*} ^2}}_{\beta} (f(x^t)-f(x^*))^2
\label{eq:mimsj99sdsd}
\end{eqnarray}
Let $\delta_t = f(x^t)-f(x^*).$
Since $\delta_{t+1}\leq \delta_t$, and by manipulating~\eqref{eq:mimsj99sdsd} we have that
\[\delta_{t+1} \leq \delta_t - \beta \delta_t^2 \overset{\times \tfrac{1}{\delta_t \delta_{t+1}}}{\Leftrightarrow }
\beta \frac{\delta_t}{\delta_{t+1}} \leq \frac{1}{\delta_{t+1}} -\frac{1}{\delta_{t}}
\overset{\delta_{t+1} \leq \delta_t}{\Leftrightarrow }
\beta\leq \frac{1}{\delta_{t+1}} -\frac{1}{\delta_{t}}.
 \]
 Summing up both sides over $t= 0, \ldots, T-1$ and using telescopic cancellation we have that
\[T\beta \leq \frac{1}{\delta_T} - \frac{1}{\delta_{0}} \leq \frac{1}{\delta_T}.\]
Re-arranging the above we have that
\[ f(x^T)-f(x^*) = \delta_T \leq \frac{1}{\beta T} = \frac{2L \norm{x^0-x^*}^2 }{T}.\]
\end{proof}

\begin{corollary}[$\mathcal{O}(1/t)$ Complexity]
Under the assumptions of Theorem~\ref{theo:convgrad}, for a given $\varepsilon >0$ and $\gamma = 1/L$ we have that 

\begin{equation*}
t \geq  \frac{L}{\varepsilon}\frac{\norm{x^0-x^*} ^2}{2 } \; \implies \; f(x^t) -\inf f \leq \varepsilon.
\end{equation*}
\end{corollary}

\subsection{Convergence for strongly convex and smooth functions}
Now we prove the convergence of gradient descent for strongly convex and smooth functions.

\begin{theorem}\label{theo:gradstrconv}
Consider the Problem \tref{Pb:differentiable function} and assume that $f$ is $\mu$-strongly convex and $L$-smooth, for some $L\geq \mu >0$.
Let $(x^t)_{t \in \mathbb{N}}$ be the sequence of iterates generated by the \tref{Algo:gradient descent constant stepsize} algorithm, with a stepsize satisfying $0 < \gamma\leq \frac{1}{L}$. Then, for $x^* = {\rm{argmin}}~f$ and for all $t \in \mathbb{N}$:
\begin{equation*}
\norm{x^{t}-x^*}^2 \leq (1-\gamma \mu)^{t} \norm{x^0 -x^*}^2.
\end{equation*}
\end{theorem}

\begin{remark}
Note that with the choice $\gamma = \tfrac{1}{L}$, the iterates enjoy a linear convergence with a rate of $(1-\mu/L).$
\end{remark}

Below we provide two different proofs for this Theorem \ref{theo:gradstrconv}.
The first one makes use of first-order variational inequalities induced by the strong convexity and smoothness of $f$.
The second one (assuming further that $f$ is twice differentiable) exploits the fact that the eigenvalues of the Hessian of $f$ are in between $\mu$ and $L$.

\begin{proof}[\rm\bf Proof of Theorem \ref{theo:gradstrconv} with first-order properties]
From \tref{Algo:gradient descent constant stepsize} we  have that
\begin{eqnarray*}
\norm{x^{t+1}-x^*}^2 & =& \norm{x^{t}-x^*- \gamma \nabla f(x^t)}^2 \nonumber\\
& =& \norm{x^{t}-x^*}^2 - 2\gamma \dotprod{\nabla f(x^t),x^t-x^*} + \gamma^2 \norm{\nabla f(x^t)}^2 \nonumber \\
& \overset{\eqref{eq:strconv}}{\leq} &(1-\gamma \mu)\norm{x^{t}-x^*}^2 - 2\gamma (f(x^t) -\inf f) + \gamma^2 \norm{\nabla f(x^t)}^2 \nonumber \\
& \overset{\eqref{eq:inversePL}}{\leq} & (1-\gamma \mu)\norm{x^{t}-x^*}^2 - 2\gamma (f(x^t) -\inf f) + 2\gamma^2 L (f(x^t) -\inf f)\nonumber \\
& =& (1-\gamma \mu)\norm{x^{t}-x^*}^2 - 2\gamma(1-\gamma L) (f(x^t) -\inf f).
\end{eqnarray*}
Since $\tfrac{1}{L} \geq \gamma$ we have that $- 2\gamma(1-\gamma L) $  is nonpositive, and thus can be safely dropped to give
\[\norm{x^{t+1}-x^*}^2  \leq (1-\gamma \mu)\norm{x^{t}-x^*}^2. \]
It now remains to unroll the recurrence.
\end{proof}

\begin{proof}[\rm\bf Proof of Theorem \ref{theo:gradstrconv} with the Hessian]
Let $T : \mathbb{R}^d \to \mathbb{R}^d$ be defined by $T(x) = x - \gamma \nabla f(x)$, so that we can write an iteration of Gradient Descent   as $x^{t+1} = T(x^t)$.
Note that the minimizer $x^*$ verifies $\nabla f(x^*) =0$, so it is a fixed point of $T$ in the sense that $T(x^*) = x^*$.
This means that $\Vert x^{t+1} - x^* \Vert = \Vert T(x^t) - T(x^*) \Vert$.
Now we want to prove that
\begin{equation}\label{L:GD contraction}
    \Vert T(x^t) - T(x^*) \Vert
    \leq 
    (1 - \gamma \mu) \Vert x^t - x^* \Vert.
\end{equation}
Indeed, unrolling the recurrence from \eqref{L:GD contraction} would provide the desired bound and end the proof.

We see that \eqref{L:GD contraction} is true as long as $T$ is $\theta$-Lipschitz, with $\theta = (1 - \gamma \mu)$.
From Lemma \ref{L:Lipschitz via jacobian}, we know that is equivalent to proving that the norm of the differential of $T$ is bounded by $\theta$.
It is easy to compute 
this differential : $DT(x) = I_d - \gamma \nabla^2 f(x)$.
If we note $v_1(x) \leq \dots \leq v_d(x)$ the eigenvalues of $\nabla^2 f(x)$, we know by Lemmas~\ref{L:strong convexity hessian} and~\ref{L:smooth via hessian} that $\mu \leq v_i(x) \leq L$.
Since we assume $\gamma L \leq 1$, we see that $0 \leq 1 - \gamma v_i(x) \leq 1 - \gamma \mu$.
So we can write
\begin{equation*}
    \text{for all $x \in \mathbb{R}^d$,} \quad
   \Vert DT(x) \Vert
    =
    \max\limits_{i=1, \dots, d} \vert 1 - \gamma v_i(x) \vert
    \leq 1 - \gamma \mu = \theta,
\end{equation*}
which allows us to conclude that \eqref{L:GD contraction} is true.
To conclude the proof of Theorem \ref{theo:gradstrconv}, take the squares in \eqref{L:GD contraction} and use the fact that $\theta \in ]0,1[ \Rightarrow \theta^2 \leq \theta$.
\end{proof}

The linear convergence rate in Theorem~\ref{theo:gradstrconv} can be transformed into a complexity result as we show next. 
\begin{corollary}[$\mathcal{O}\left(\log(1/\varepsilon)\right)$ Complexity]
Under the same assumptions as Theorem~\ref{theo:gradstrconv}, for a given $\varepsilon>0,$ we have that if $\gamma = 1/L$ then
\begin{equation*} 
t\geq \frac{L}{\mu} \log\left(\frac{\norm{x^{0}-x^*}^2}{\varepsilon}\right)  \quad \Rightarrow \quad \norm{x^{t}-x^*}^2 \leq \varepsilon.
\end{equation*}
\end{corollary}

\begin{proof}
   It is a direct consequence of lemma~\ref{lem:itercomplex} in the appendix. 
\end{proof}

\subsection{Convergence for Polyak-Łojasiewicz and smooth functions}

Here we present a convergence result for nonconvex functions satisfying the Polyak-Łojasiewicz condition (see Definition \ref{D:polyak lojasiewicz}).
This is a favorable setting, since all local minima and critical points are also global minima (Lemma \ref{L:PL function global minimizers}), which will guarantee convergence.
Moreover the PŁ property imposes a quadratic growth on the function, so we will recover bounds which are similar to the strongly convex case.

\begin{theorem}
\label{T:GD PL smooth}
Consider the Problem \tref{Pb:differentiable function} and assume that $f$ is $\mu$-Polyak-Łojasiewicz and $L$-smooth, for some $L\geq \mu >0$.
Consider $(x^t)_{t \in \mathbb{N}}$ a sequence generated by the \tref{Algo:gradient descent constant stepsize} algorithm, with a stepsize satisfying $0<\gamma \leq \frac{1}{L}$. Then:
\begin{equation*}\label{eq:GD PL smooth rate}
f(x^{t})-\inf f \leq (1-\gamma \mu)^t (f(x^0)-\inf f).
\end{equation*}
\end{theorem}

\begin{proof}
We can use Lemma \ref{L:smooth upper quadratic}, together with the update rule of \tref{Algo:gradient descent constant stepsize}, to write
\begin{eqnarray*}
f(x^{t+1})& \leq &  f(x^{t})+ \langle \nabla f(x^t), x^{t+1}-x^t \rangle +\frac{L}{2} \| x^{t+1}-x^t\|^2 \notag\\
&=& f(x^{t})-\gamma\Vert \nabla f(x^t) \Vert^2 +\frac{L \gamma^2}{2} \| \nabla f(x^t)\|^2 \\
&=& f(x^{t}) - \frac{\gamma}{2} \left(2 - L \gamma \right)\Vert \nabla f(x^t) \Vert^2 \\
& \leq & f(x^{t}) - \frac{\gamma}{2}\Vert \nabla f(x^t)\Vert^2,
\end{eqnarray*}
where in the last inequality we used our hypothesis on the stepsize that $\gamma L \leq 1$.
We can now use the Polyak-Łojasiewicz property (recall Definition \ref{D:polyak lojasiewicz}) to write:
\begin{equation*}
    f(x^{t+1})
    \leq
    f(x^{t}) - \gamma \mu (f(x^t) - \inf f).
\end{equation*}
The conclusion follows after subtracting $\inf f$ on both sides of this inequality, and using recursion.
\end{proof}

\begin{corollary}[$\log(1/\varepsilon)$ Complexity]
Under the same assumptions as Theorem~\ref{T:GD PL smooth}, for a given $\varepsilon>0,$ we have that if $\gamma = 1/L$ then
\begin{equation*} 
t\geq \frac{L}{\mu} \log\left(\frac{f(x^{0})-\inf f}{\varepsilon}\right)  \quad \Rightarrow \quad f(x^{t})-\inf f  \leq \varepsilon.
\end{equation*}
\end{corollary}

\begin{proof}
    This is a direct consequence of lemma~\ref{lem:itercomplex}.
\end{proof}

\subsection{Bibliographic notes}

Our second proof for convex and smooth in Theorem~\ref{theo:convgrad} is taken from~\cite{bubeck-book}.
Proofs in the convex and strongly convex case can be found in \cite{Nes04}.
Our proof under the Polyak-Łojasiewicz  was taken from~\cite{KarimiNS16}.

\section{Theory : Sum of functions}

\subsection{Definitions}

In the next sections we will assume that our objective function is a sum of functions.

\begin{problem}[Sum of Functions]\label{Pb:sum of functions}
We want to minimize a function $f : \mathbb{R}^d \to \mathbb{R}$ which writes as 
\begin{equation*}
f(x) \eqdef \frac{1}{n} \sum_{i=1}^n f_i(x),
\end{equation*}
where $f_i : \mathbb{R}^d \to \mathbb{R}$.
We require that the problem is well-posed, in the sense that ${\rm{argmin}}~f \neq \emptyset$ and that the $f_i$'s are bounded from below.
\end{problem}

Depending on the applications, we will consider two different sets of assumptions.

\begin{assumption}[Sum of Convex]\label{Ass:SGD sum of convex}
We consider the Problem \tref{Pb:sum of functions} 
where each $f_i : \mathbb{R}^d \to \mathbb{R}$ is assumed to be convex.
\end{assumption}

\begin{assumption}[Sum of $L_{\max}$--Smooth]\label{Ass:SGD sum of smooth}
We consider the Problem \tref{Pb:sum of functions} 
where each $f_i : \mathbb{R}^d \to \mathbb{R}$ is assumed to be $L_i$-smooth.
We will note $L_{\max} \eqdef \max\limits_{1,\dots, n} L_i$, and $L_{\avg} \eqdef \frac{1}{n}\sum_{i=1}^n L_i$.
We will also note $L_f$ the Lipschitz constant of $\nabla f$.
\end{assumption}

Note that, in the above Assumption \tref{Ass:SGD sum of smooth}, the existence of $L_f$ is not an assumption but the consequence of the smoothness of the $f_i$'s. Indeed:

\begin{lemma}\label{L:sum of smooth is smooth}
Consider the Problem \tref{Pb:sum of functions}.
If the $f_i$'s are $L_i$-smooth, then $f$ is $L_{\avg} $-smooth. 
\end{lemma}

\begin{proof}
Using the triangular inequality we have that
\begin{align*}
    \Vert \nabla f(y) - \nabla f(x) \Vert
    &=
    \Vert \frac{1}{n}\sum_{i=1}^n \nabla f_i(y) - \nabla f_i(x) \Vert \;
    \leq \;
    \frac{1}{n}\sum_{i=1}^n \Vert \nabla f_i(y) - \nabla f_i(x) \Vert
    \; \leq \;
    \frac{1}{n}\sum_{i=1}^n L_i \Vert y-x \Vert\\
    & = L_{\avg}  \Vert y-x \Vert.
\end{align*}
This proves that $f(x)$ is  $L_{\avg}$-smooth.  
\end{proof}

\begin{definition}[Notation]\label{D:variance}
Given two random variables $X,Y$ in $\mathbb{R}^d$, we note :
\begin{itemize}
    \item the \textbf{expectation} of $X$ as $\mathbb{E}\left[ X \right]$,
    \item the expectation of $X$ \textbf{conditioned} to $Y$ as $\mathbb{E}\left[ X \ | \ Y \right]$,
    \item the \textbf{variance} of $X$ as $\mathbb{V}[X] \eqdef \mathbb{E}[\Vert X - \mathbb{E}[X] \Vert^2]$.
\end{itemize}
\end{definition}

\begin{lemma}[Variance and expectation]\label{L:variance and expectation}
Let $X$ be a random variable in $\mathbb{R}^d$.
\begin{enumerate}
    \item For all $y \in \mathbb{R}^d$, $\mathbb{V}\left[ X \right] \leq \mathbb{E}\left[ \Vert X - y \Vert^2 \right]$.
    \item $\mathbb{V}\left[ X \right] \leq \mathbb{E}\left[ \Vert X \Vert^2 \right]$.
\end{enumerate}
\end{lemma}

\begin{proof}
Item 2 is a direct consequence of the first with $y=0$. 
To prove item 1,  we use that
\begin{equation*}
    \Vert X-\mathbb{E}\left[ X \right] \Vert^2 = \Vert X-y \Vert^2 + \Vert y - \mathbb{E}\left[ X \right] \Vert^2 + 2 \langle X - y, y - \mathbb{E}\left[ X \right] \rangle,
\end{equation*}
and then take  expectation to conclude
\begin{equation*}
    \mathbb{V}\left[ X \right]
    =
    \mathbb{E}\left[ \Vert X-y \Vert^2 \right] 
    -2 \mathbb{E}\left[ \Vert y - \mathbb{E}\left[ X \right] \Vert^2 \right]
    \leq
    \mathbb{E}\left[ \Vert X-y \Vert^2 \right].
\end{equation*}
\end{proof}

\subsection{Expected smoothness}

Here we focus on the smoothness properties that the functions $f_i$ verify in \textit{expectation}.
The so-called \textit{expected smoothness} property below can be seen as  ``cocoercivity in expectation'' (remember Lemma \ref{lem:convandsmooth}).

\begin{lemma}\label{L:expected smoothness}
If Assumptions \tref{Ass:SGD sum of smooth} and \tref{Ass:SGD sum of convex} hold,
then $f$ is $L_{\max}$-\textbf{smooth in expectation}, in the sense that
\begin{equation*}
\text{for all $x,y \in \mathbb{R}^d$}, \quad
\frac{1}{2L_{\max}}\E{\norm{\nabla f_i(y) - \nabla f_i(x) }^2 } \; \leq \; f(y) - f(x) -\dotprod{\nabla f(x), y-x}.
\end{equation*}
\end{lemma}

\begin{proof}
Using~\eqref{eq:convandsmooth} in Lemma~\ref{lem:convandsmooth} applied to $f_i$, together with  the fact that $L_i \leq L_{\max}$, allows us to write
\begin{eqnarray*}\label{eq:convandsmoothfi}
\frac{1}{2L_{\max}}\norm{\nabla f_i(y)-\nabla f_i(x)}^2 & \leq & f_i(y) - f_i(x) +\dotprod{\nabla f_i(x),y-x}.
\end{eqnarray*}
To conclude, multiply the above inequality by $\frac{1}{n}$, and sum over $i$, using the fact that $\frac{1}{n}\sum_i f_i = f$ and $\frac{1}{n}\sum_i \nabla f_i = \nabla f$.
\end{proof}

As a direct consequence we also have the analog of Lemma~\ref{lem:convandsmooth} in expectation.

\begin{lemma}\label{L:expected smoothness minimizer}
If Assumptions \tref{Ass:SGD sum of smooth} and \tref{Ass:SGD sum of convex} hold,
then,  for every $x \in \mathbb{R}^d$ and every $x^* \in {\rm{argmin}}~f$, we have that
\begin{eqnarray*}
\frac{1}{2L_{\max}} \E{\norm{\nabla f_i(x) - \nabla f_i(x^*) }^2 } \; \leq \;  f(x) - \inf f.
\end{eqnarray*}
\end{lemma}

\begin{proof}
Apply Lemma \ref{L:expected smoothness} with $x = x^*$ and $y=x$, since $f(x^*) = \inf f$ and $\nabla f(x^*) =0.$
\end{proof}

\subsection{Controlling the variance}

Some stochastic problems are easier than others. For instance, a problem where all the $f_i$'s are the same is easy to solve, as it suffices to minimize one $f_i$ to obtain a minimizer of $f$.
We can also imagine that if the $f_i$ are not exactly the same but look similar, the problem will also be easy.
And of course, we expect that the easier the problem, the faster our algorithms will be.
In this section we present one way to quantify this phenomena.

\subsubsection{Interpolation}

\begin{definition}\label{D:interpolation holds}
Consider the Problem \tref{Pb:sum of functions}.
We say that \textbf{interpolation holds} if there exists a common $x^* \in \mathbb{R}^d$ such that $f_i(x^*) = \inf f_i$ for all $i=1,\dots,n$.
In this case, we say that interpolation holds at $x^*$.
\end{definition}

\noindent Even though unspecified, the $x^*$ appearing in Definition \ref{D:interpolation holds} must be a minimizer of $f$.

\begin{lemma}\label{L:interpolation vector is minimizer}
Consider the Problem \tref{Pb:sum of functions}.
If interpolation holds at $x^* \in \mathbb{R}^d$, then $x^* \in {\rm{argmin}}~f$.
\end{lemma}

\begin{proof}
Let interpolation hold at $x^* \in \mathbb{R}^d$.
By Definition \ref{D:interpolation holds}, this means that $x^* \in {\rm{argmin}}~f_i$.
Therefore, for every $x \in \mathbb{R}^d$,
\begin{equation*}
    f(x^*) 
    = 
    \frac{1}{n}\sum\limits_{i=1}^n f_i(x^*)
    = 
    \frac{1}{n}\sum\limits_{i=1}^n \inf f_i
    \leq
    \frac{1}{n}\sum\limits_{i=1}^n f_i(x)
    =
    f(x).
\end{equation*}
This proves that $x^* \in {\rm{argmin}}~f$.
\end{proof}

Interpolation means that there exists some $x^*$ that simultaneously achieves the minimum of all loss functions $f_i$.
In terms of learning problems, this means that the model perfectly fits every data point. 
This is illustrated below with a couple of examples.

\begin{example}[Least-squares and interpolation]\label{Ex:least squares interpolation}
Consider a regression problem with data $(\phi_i,y_i)_{i=1}^n \subset \mathbb{R}^d \times \mathbb{R}$, and let $f(x)=\frac{1}{2m}\Vert \Phi x - y \Vert^2$ be the corresponding least-squares function with $\Phi = (\phi_i)_{i=1}^n$ and $y = (y_i)_{i=1}^n$.
This is a particular case of Problem \tref{Pb:sum of functions}, with $f_i(x) = \frac{1}{2}(\langle \phi_i,x \rangle - y_i)^2$.
We see here that interpolation holds if and only if there exists $x^* \in \mathbb{R}^d$ such that $\langle \phi_i,x^* \rangle = y_i$.
In other words, we can find an hyperplane in $\mathbb{R}^d \times \mathbb{R}$ passing through each data point $(\phi_i, y_i)$.
This is why we talk about \emph{interpolation}.

For this linear model, note that interpolation holds if and only if $y$ is in the range of $\Phi$, which is always true if $\Phi$ is surjective.
This generically holds when $d > n$, which is usually called the \emph{overparametrized} regime.
\end{example}

\begin{example}[Neural Networks and interpolation]\label{Ex:neural network interpolation}
\label{exe:interpolation}
Let $\Phi : \mathbb{R}^d \to \mathbb{R}^n$, $y \in \mathbb{R}^n$, and consider the nonlinear least-squares $f(x) = \frac{1}{2}\norm{\Phi(x) -y}^2$.
As in the linear case, interpolation holds if and only if there exists $x^* \in \mathbb{R}^d$ such that $\Phi(x^*) =y$, or equivalently, if $\inf f =0$.
The interpolation condition has drawn much attention recently because it was empirically observed that many overparametrized deep neural networks (with $d \gg n$)
achieve $\inf f=0$~\cite{ma2018power,belkin2020}.
\end{example}

\subsubsection{Interpolation constants}

Here we introduce different measures of how \emph{far from interpolation} we are.
We start with a first quantity measuring how the infimum of $f$ and the $f_i$'s are related.

\begin{definition}\label{D:function noise}
Consider the Problem \tref{Pb:sum of functions}. 
We define the \textbf{function noise} as
\begin{equation*}
    \Delta^{*}_f \eqdef \inf f - \frac{1}{n}\sum_{i=1}^n \inf f_i.
\end{equation*}
\end{definition}

\begin{example}[Function noise for least-squares]
Let $f$ be a least-squares as in Example \ref{Ex:least squares interpolation}.
It is easy to see that $\inf f_i = 0$, implying that the function noise is exactly $\Delta^*_f = \inf f$.
We see in this case that $\Delta^*_f=0$ if and only if interpolation holds (see also the next Lemma).
If the function noise $\Delta^*_f = \inf f$ is nonzero, it can be seen as a measure of how far we are from interpolation.
\end{example}

\begin{lemma}\label{L:interpolation via function noise}
Consider the Problem \tref{Pb:sum of functions}. We have that
\begin{enumerate}
    \item $\Delta^{*}_f \geq 0$.
    \item Interpolation holds if and only if $\Delta^{*}_f=0$.
\end{enumerate}
\end{lemma}

\begin{proof}~~
\begin{enumerate}
    \item Let $x^* \in {\rm{argmin}}~f$, so that we can write 
    \begin{equation*}
        \Delta^{*}_f 
        = 
        f(x^*) - \frac{1}{n}\sum_{i=1}^n \inf f_i
        \geq 
        f(x^*) - \frac{1}{n}\sum_{i=1}^n f_i(x^*)
        =
        f(x^*) - f(x^*)
        = 0.        
    \end{equation*}
    \item Let interpolation hold at $x^* \in \mathbb{R}^d$.
    According to Definition \ref{D:interpolation holds} we have $x^* \in {\rm{argmin}}~f_i$. According to Lemma \ref{L:interpolation vector is minimizer}, we have $x^* \in {\rm{argmin}}~f$.
    So we indeed have
    \begin{equation*}
        \Delta^{*}_f 
        = 
        \inf f - \frac{1}{n}\sum_{i=1}^n \inf f_i
        = 
        f(x^*) - \frac{1}{n}\sum_{i=1}^n f_i(x^*)
        =
        f(x^*) - f(x^*)
        = 0.        
    \end{equation*}
    If instead we have $\Delta^{*}_f  = 0$, then we can write for some $x^* \in {\rm{argmin}}~f$ that
    \begin{equation*}
        0 
        = 
        \Delta^{*}_f
        =
        f(x^*) - \frac{1}{n}\sum_{i=1}^n \inf f_i
        =
        \frac{1}{n}\sum_{i=1}^n \left( f_i(x^*) - \inf f_i \right).
    \end{equation*}
    Clearly we have $f_i(x^*) - \inf f_i \geq 0$, so this sum being $0$ implies that $f_i(x^*) - \inf f_i =0$ for all $i=1,\dots,n$.
    In other words, interpolation holds.
\end{enumerate}
\end{proof}

We can also measure how close we are to interpolation using  gradients instead of function values.

\begin{definition}\label{D:gradient solution variance}
Let Assumption \tref{Ass:SGD sum of smooth} hold.
We define the \textbf{gradient noise} as
\begin{equation*}
    \sigma^{*}_f \eqdef \inf\limits_{x^* \in {\rm{argmin}}~f} \ \mathbb{V}\left[ \nabla f_i(x^*) \right],
\end{equation*}
where for a random vector $X\in\R^d$ we use
\[\mathbb{V}[X] := \mathbb{E}[\,\Vert X - \mathbb{E}[X]\, \Vert^2]  .\]

\end{definition}

\begin{lemma}\label{L:interpolation via gradient variance}
Let Assumption \tref{Ass:SGD sum of smooth} hold.  It follows that
\begin{enumerate}
    \item $\sigma^{*}_f \geq 0$.
    \item If Assumption \tref{Ass:SGD sum of convex} holds, then $\sigma^{*}_f = \mathbb{V}\left[ \nabla f_i(x^*) \right]$ for every $x^* \in {\rm{argmin}}~f$.
    \item If interpolation holds then $\sigma^{*}_f = 0$. This becomes an equivalence if Assumption \tref{Ass:SGD sum of convex} holds.
\end{enumerate}
\end{lemma}

\begin{proof}~~
\begin{enumerate}
    \item From Definition \ref{D:variance} we have that the variance $\mathbb{V}\left[ \nabla f_i(x^*) \right]$ is nonnegative, which implies $\sigma^{*}_f \geq 0$.
    \item Let $x^*, x' \in {\rm{argmin}}~f$, and let us show that $\mathbb{V}\left[ \nabla f_i(x^*) \right] = \mathbb{V}\left[ \nabla f_i(x') \right]$.
    Since Assumptions \tref{Ass:SGD sum of smooth} and \tref{Ass:SGD sum of convex} hold, we can use the expected smoothness via Lemma \ref{L:expected smoothness minimizer} to obtain
    \begin{equation*}
        \frac{1}{2L_{\max}} \E{\norm{\nabla f_i(x') - \nabla f_i(x^*) }^2 }  
        \leq  
        f(x') - \inf f
        =
        \inf f - \inf f =0.
    \end{equation*}
    This means that $\E{\norm{\nabla f_i(x') - \nabla f_i(x^*) }^2 } =0$, which in turns implies that, for every $i=1,\dots,n$, we have $\norm{\nabla f_i(x') - \nabla f_i(x^*) } = 0$.
    In other words, $\nabla f_i(x') = \nabla f_i(x^*)$, and thus $ \mathbb{V}\left[ \nabla f_i(x^*) \right] =  \mathbb{V}\left[ \nabla f_i(x') \right].$
    \item If interpolation holds, then there exists (see Lemma \ref{L:interpolation vector is minimizer}) $x^* \in {\rm{argmin}}~f$ such that $x^* \in {\rm{argmin}}~f_i$ for every $i=1,\dots,n$.
    From Fermat's theorem, this implies that $\nabla f_i(x^*)=0$ and $\nabla f(x^*)=0$.
   Consequently $\mathbb{V}\left[ \nabla f_i(x^*) \right] = \mathbb{E}\left[ \Vert \nabla f_i(x^*) - \nabla f(x^*)\Vert^2 \right] = 0$.
    This proves that $\sigma^{*}_f =0$.
    Now, if Assumption \tref{Ass:SGD sum of convex} holds and $\sigma^{*}_f =0$, then we can use the previous item to say that for any $x^* \in {\rm{argmin}}~f$ we have $\mathbb{V}\left[ \nabla f_i(x^*) \right] = 0$.
    By definition of the variance and the fact that $\nabla f(x^*) =0$, this implies that for every $i=1,\dots,n$, $\nabla f_i(x^*)=0$.
    Using again the convexity of the $f_i$'s, we deduce that $x^* \in {\rm{argmin}}~f_i$, which means that interpolation holds.
\end{enumerate}
\end{proof}

Both $\sigma^{*}_f$ and $\Delta^{*}_f$ measure how far we are from interpolation.  Furthermore,  these two constants 
are related through the following bounds.

\begin{lemma}\label{L:interpolation noise variance relationship}
Let Assumption \tref{Ass:SGD sum of smooth} hold. 
\begin{enumerate}
    \item We have $\sigma^{*}_f \leq 2 L_{\max}\Delta^{*}_f$.
    \item If moreover each $f_i$ is $\mu$-strongly convex, then  $2 \mu \Delta^{*}_f \leq \sigma^{*}_f$.
\end{enumerate}
\end{lemma}

\begin{proof}
~~
\begin{enumerate}
    \item Let $x^* \in {\rm{argmin}}~f$.
    Using Lemma \ref{L:smooth descent}, we can write $\Vert \nabla f_i(x^*) \Vert^2 \leq 2 L_{\max} (f_i(x^*) - \inf f_i)$ for each $i$.
    The conclusion follows directly after taking the expectation on this inequality, and using the fact that $\mathbb{E}\left[ \Vert \nabla f_i(x^*) \Vert^2 \right] = \mathbb{V}\left[ \nabla f_i(x^*) \right] \geq \sigma^{*}_f$.
    \item This is exactly the same proof, except that we use Lemma \ref{L:strong convexity Polyak Lojasiewicz} instead of \ref{L:smooth descent}. 
\end{enumerate}
\end{proof}

\subsubsection{Variance transfer}

Here we provide two lemmas which allow to exchange variance-like terms like $\mathbb{E}\left[ \Vert \nabla f_i(x) \Vert^2 \right]$ with interpolation constants and function values.
This is important since $\mathbb{E}\left[ \Vert \nabla f_i(x) \Vert^2 \right]$ actually controls the variance of the gradients (see Lemma \ref{L:variance and expectation}).

\begin{lemma}[Variance transfer : function noise]\label{L:variance transfer function noise}
If Assumption \tref{Ass:SGD sum of smooth} holds,
then for all $x \in \mathbb{R}^d$ we have
\begin{equation*}
\E{\norm{\nabla f_i(x) }^2} \leq 2 L_{\max} (f(x) - \inf f) + 2 L_{\max}\Delta^*_f.
\end{equation*}
 \end{lemma}
 
\begin{proof}
Let $x \in \mathbb{R}^d$ and $x^* \in {\rm{argmin}}~f$.
    Using Lemma \ref{L:smooth descent}, we can write 
    \begin{equation}
        \Vert \nabla f_i(x) \Vert^2 \leq 
        2 L_{\max} (f_i(x) - \inf f_i)
        =
        2 L_{\max} (f_i(x) - f_i(x^*))
        +
        2 L_{\max} (f_i(x^*) - \inf f_i),
    \end{equation}
    for each $i$.
    The conclusion follows directly after taking  expectation over  the above inequality.
\end{proof}

\begin{lemma}[Variance transfer : gradient noise]
\label{L:variance transfer gradient variance}
If Assumptions \tref{Ass:SGD sum of smooth} and \tref{Ass:SGD sum of convex} hold, 
then for all $x \in \mathbb{R}^d$ we have that
\begin{align*}
\EE{}{\|\nabla f_i (x)\|^2 } & \leq  4 L_{\max} ( f(x)-\inf f )    +2 \sigma^{*}_f.
\end{align*}
\end{lemma}

\begin{proof}
Let $x^* \in {\rm{argmin}}~f$, so that $\sigma^{*}_f = \mathbb{V}\left[ \norm{\nabla f_i(x^*)}^2 \right]$ according to Lemma \ref{L:interpolation via gradient variance}.
Start by writing
\begin{align*}
\norm{\nabla f_i (x) }^2  & \leq 2\norm{\nabla f_i (x) -\nabla f_i(x^*)}^2  +2\norm{\nabla f_i(x^*)}^2.
\end{align*}
Taking the expectation over the above inequality, then applying Lemma \ref{L:expected smoothness minimizer} gives the result.
\end{proof}

\section{Stochastic Gradient Descent}\label{sec:SGD}

\begin{algorithm}[SGD]\label{Algo:Stochastic GD}
Consider Problem \tref{Pb:sum of functions}.
Let $x^0 \in \mathbb{R}^d$, and let $\gamma_t>0$ be a sequence of step sizes.
The \textbf{Stochastic Gradient Descent (SGD)} algorithm is given by the iterates $(x^t)_{t \in \mathbb{N}}$ where
\begin{align*}
    i_t & \in \{1,\ldots n\} & \mbox{Sampled with probability }\frac{1}{n}\\
    x^{t+1} & = x^t - \gamma_t\nabla f_{i_t}(x^t).
\end{align*}
\end{algorithm}

\begin{remark}[Unbiased estimator of the gradient]
An important feature of the \tref{Algo:Stochastic GD} Algorithm is that at each iteration we follow the direction $-\nabla f_{i_t}(x^t)$, which is an \textit{unbiaised} estimator of $-\nabla f(x^t)$.
Indeed, since
\begin{eqnarray*}
 \E{\nabla f_i(x^t) \ | \ x^t} = \sum_{i=1}^n \frac{1}{n}\nabla f_i(x^t) = \nabla f(x^t). 
\end{eqnarray*}
\end{remark}

\subsection{Convergence for convex and smooth functions}

The behaviour of the \tref{Algo:Stochastic GD} algorithm is very dependant of the choice of the sequence of stepsizes $\gamma_t$. In our next Theorem \ref{theo:sgdconvsmooth} we prove the convergence of SGD with a general sequence of stepsizes $\gamma_t$ which is bouded above by $\frac{1}{4L_{\max}}.$
The particular cases of constant stepsizes and of decreasing stepsizes are dealt with in Theorems \ref{T:SGD convex smooth constant stepsize} and \ref{T:SGD convex smooth vanishing stepsize}, respectively. 

\begin{theorem}\label{theo:sgdconvsmooth}
Let Assumptions \tref{Ass:SGD sum of smooth} and \tref{Ass:SGD sum of convex} hold.
Consider $(x^t)_{t \in \mathbb{N}}$ a sequence generated by the \tref{Algo:Stochastic GD} algorithm, with a sequence of stepsizes satisfying $0<\gamma_t\leq \frac{1}{4L_{\max}}$.
It follows that for every $T \geq 1$, 
\begin{equation*}
    \E{f(\bar{x}^T) - \inf f} \leq \frac{\norm{x^0 - x^*}^2}{\sum_{t=0}^{T-1}\gamma_t}  + 2\sigma_f^* \frac{\sum_{t=0}^{T-1}\gamma_t^2}{\sum_{t=0}^{T-1}\gamma_t},
\end{equation*}
where $\bar{x}^T \; \eqdef \; \tfrac{1}{\sum_{t=0}^{T-1}\gamma_t}\sum_{t=0}^{T-1}\gamma_t x^t$.
\end{theorem}

\begin{proof}
Let $x^* \in {\rm{argmin}}~f$, so we have $\sigma_f^* = \mathbb{V}[\nabla f_i(x^*)]$ (see Lemma \ref{L:interpolation via gradient variance}).
We will note $\EE{t}{\cdot}$ instead of $\EE{t}{\cdot \ | \ x^t}$, for simplicity.
Let us start by analyzing the behaviour of $\norm{x^t - x^*}^2$. By developing the squares, we obtain
\begin{align*}
\norm{x^{t+1} - x^*}^2 &= \norm{x^t - x^*}^2 - 2\gamma_t\langle \nabla f_i(x^t), x^t - x^*\rangle + \gamma_t^2\norm{\nabla f_i(x^t)}^2
\end{align*}
Hence, after taking the expectation conditioned on $x^t$, we can use the convexity of $f$ (see \Cref{L:convexity via hyperplanes}) and a variance transfer lemma (see \Cref{L:variance transfer gradient variance}) to write
\begin{eqnarray*}
\EE{t}{\norm{x^{t+1} - x^*}^2} &=& \norm{x^t - x^*}^2 + 2\gamma_t\langle \nabla f(x^t), x^* - x^t\rangle + \gamma_t^2\EE{t}{\norm{\nabla f_i(x^t)}^2}\\
&{\leq}& \norm{x^t - x^*}^2 + 2\gamma_t(2\gamma_t L_{\max} - 1)(f(x^t) - \inf f)) + 2\gamma_t^2\sigma_f^* \\
&{\leq}& \norm{x^t - x^*}^2  -\gamma_t(f(x^t) - \inf f)) + 2\gamma_t^2\sigma_f^*,
\end{eqnarray*}
where in the last inequality we have used our assumption that $\gamma_t L_{\max} \leq \tfrac14$.
Rearranging and taking expectation, we have
\begin{align*}
\gamma_t \E{f(x^t) - \inf f} \leq \E{\norm{x^t - x^*}^2} - \E{\norm{x^{t+1} - x^*}^2} + 2\gamma_t^2\sigma_f^*.
\end{align*}
Summing over $t =0,\ldots, T-1$ for $T \geq 1$, and using telescopic cancellation gives
\begin{align*}
\sum_{t=0}^{T-1}\gamma_t\E{f(x^t) - \inf f} \leq \norm{x^0 - x^*}^2 - \E{\norm{x^{T} - x^*}^2} + 2\sigma_f^*\sum_{t=0}^{T-1}\gamma_t^2.
\end{align*}
Since $\E{\norm{x^{T} - x^*}^2} \geq 0$, dividing both sides by $\sum_{t=0}^{T-1}\gamma_t$ gives:
\begin{align*}
\E{\sum_{t=0}^{T-1}\frac{\gamma_t}{\sum_{t=0}^{T-1}\gamma_t}(f(x^t) - \inf f)} 
\leq 
\frac{\norm{x^0 - x^*}^2}{\sum_{t=0}^{T-1}\gamma_t} + \frac{2\sigma_f^*\sum_{t=0}^{T-1}\gamma_t^2}{\sum_{t=0}^{T-1}\gamma_t}.
\end{align*}
Finally, using that $f$ is convex together with Jensen's inequality gives
\begin{align*}
\E{f(\bar{x}^T) - \inf f}  &\leq 
\E{\sum_{t=0}^{T-1}\frac{\gamma_t}{\sum_{t=0}^{T-1}\gamma_t}(f(x^t) - \inf f)}  \\
&\leq \frac{\norm{x^0 - x^*}^2}{\sum_{t=0}^{T-1}\gamma_t} + \frac{2\sigma_f^*\sum_{t=0}^{T-1}\gamma_t^2}{\sum_{t=0}^{T-1}\gamma_t}.
\end{align*}
\end{proof}

\begin{remark}[On the choice of stepsizes for \tref{Algo:Stochastic GD}] \label{rem:sgdstep}
Looking at the bound obtained in Theorem \ref{theo:sgdconvsmooth},  we see that the first thing we want is $\sum_{s=0}^{\infty} \gamma_s = + \infty$ so that the first term (a.k.a the bias term) vanishes.
This can be achieved with constant stepsizes, or with stepsizes of the form $\frac{1}{t^\alpha}$ with $\alpha <1$ (see Theorems \ref{T:SGD convex smooth constant stepsize} and \ref{T:SGD convex smooth vanishing stepsize} below).
The second term (a.k.a the variance term) is less trivial to analyse.
\begin{itemize}
    \item If interpolation holds (see Definition \ref{D:interpolation holds}), then the variance term $\sigma_f^*$ is zero.
    This means that the expected values converge to zero at a rate of the order $\frac{1}{\sum_{s=0}^{t-1} \gamma_s}$.
    For constant stepsizes this gives a $O \left( \frac{1}{t} \right)$ rate.  For decreasing stepsizes $\gamma_t = \frac{1}{t^\alpha}$ this gives a $O \left( \frac{1}{t^{1 - \alpha}} \right)$ rate.
    We see that the best among those rates is obtained when $\alpha =0$ and the decay in the stepsize is slower. In other words when the stepsize  is  constant.
    Thus when interpolation holds the problem is so easy that the stochastic algorithm behaves like the deterministic one and enjoys a $1/t$ rate with constant stepsize,  as in Theorem \ref{theo:convgrad}.
    \item If interpolation does not hold the expected values will be asymptotically controlled by
    \begin{equation*}
        \frac{\sum_{s=0}^{t-1} \gamma_s^2}{\sum_{s=0}^{t-1} \gamma_s}.
    \end{equation*}
    We see that we want $\gamma_s$ to decrease as slowly as possible (so that the denominator is big) but at the same time that $\gamma_s$ vanishes as fast as possible (so that the numerator is small).
    So a trade-off must be found.
    For constant stepsizes, this term becomes a constant $O(1)$,  and thus  \tref{Algo:Proximal Stochastic GD} does not converge for constant stepsizes.
    For decreasing stepsizes $\gamma_t = \frac{1}{t^\alpha}$,  this term becomes (omitting logarithmic terms) $O \left( \frac{1}{t^\alpha} \right)$ if $0<\alpha \leq \frac{1}{2}$, and $O \left( \frac{1}{t^{1-\alpha}} \right)$ if $\frac{1}{2} \leq \alpha <1$.
    So the best compromise for this bound is to take $\alpha = 1/2$. This case is detailed in the next Theorem.
\end{itemize}
\end{remark}

\begin{theorem} \label{T:SGD convex smooth constant stepsize}
Let Assumptions \tref{Ass:SGD sum of smooth} and \tref{Ass:SGD sum of convex} hold.
Consider $(x^t)_{t \in \mathbb{N}}$ a sequence generated by the \tref{Algo:Stochastic GD} algorithm with a constant stepsize $\gamma_t \equiv \gamma \leq \tfrac{1}{4L_{\max}}$. 
Then, for every $T \geq 1$,
\begin{equation*}
    \E{f(\bar{x}^T) - \inf f} \leq \frac{\norm{x^0 - x^*}^2}{\gamma T} + 2 \gamma \sigma_f^*,
\end{equation*}
where $\bar{x}^T \; \eqdef \; \frac{1}{T}\sum_{t=0}^{T-1} x^t$.
In particular, if for a fixed horizon $T \geq 1$ we set $\gamma = \frac{\gamma_0}{\sqrt{T}}$ for some $\gamma_0 \leq \tfrac{1}{4L_{\max}}$, then
\begin{equation*}
    \E{f(\bar{x}^T) - \inf f} \leq \frac{\norm{x^0 - x^*}^2}{\gamma_0 \sqrt{T}} + \frac{2 \gamma_0 \sigma_f^*}{\sqrt{T}} = \mathcal{O}\left( \frac{1}{\sqrt{T}} \right).
\end{equation*}
\end{theorem}

\begin{proof}
    This is a direct consequence of Theorem \ref{theo:sgdconvsmooth}, since $\sum_{t=0}^{T-1} \gamma_t = \gamma T$ and $\sum_{t=0}^{T-1} \gamma_t^2 = \gamma^2 T$.
\end{proof}

\begin{corollary}[$\mathcal{O}(1/\varepsilon^2)$ Complexity]\label{T:SGD complexity convex smooth}
Consider the setting of Theorem~\ref{T:SGD convex smooth constant stepsize}. For every $\varepsilon > 0$, we can guarantee that $\E{f(\bar{x}^T) - \inf f} \leq \varepsilon$ provided that 
	\begin{equation*}
	\gamma = \frac{\gamma_0}{\sqrt{T}}, \, \gamma_0 = \min \left\{ \frac{1}{4L_{\max}}, \frac{\norm{x^0 - x^*}}{\sqrt{2\sigma_f^*}} \right\}
	\ \text{ and } \ 
	T \geq \left(  \norm{x^0 - x^*} \sqrt{\sigma_f^*} +  \norm{x^0 - x^*}^2 L_{\max}  \right)^2 \frac{16}{\varepsilon^2}.
	\end{equation*}
\end{corollary}

\begin{proof}
This is a direct consequence of Theorem \ref{T:SGD convex smooth constant stepsize} and Lemma \ref{L:complexity meta convex} with $A = \norm{x^0 - x^*}^2$, $B = 2  \sigma_f^*$ and $C = 4L_{\max}$. We also simplify numerical constants like $\sqrt{2} \leq 2$.
\end{proof}

\begin{theorem}\label{T:SGD convex smooth vanishing stepsize}
Let Assumptions \tref{Ass:SGD sum of smooth} and \tref{Ass:SGD sum of convex} hold.
Consider $(x^t)_{t \in \mathbb{N}}$ a sequence generated by the \tref{Algo:Stochastic GD} algorithm with a vanishing stepsize $\gamma_t = \tfrac{\gamma_0}{\sqrt{t+1}}$ where $\gamma_0 \leq \tfrac{1}{4L_{\max}}$.
Then for every $T \geq 1$,
\begin{equation*}
    \E{f(\bar{x}^T) - \inf f} \leq \frac{5\norm{x^0 - x^*}^2}{4\gamma_0 \sqrt{T}}  + \sigma_f^* \frac{5 \gamma_0 \log(T+1) }{ \sqrt{T}} = {\mathcal{O}} \left( \frac{\log(T+1)}{\sqrt{T}} \right),
\end{equation*}
where $\bar{x}^T \; \eqdef \; \tfrac{1}{\sum_{t=0}^{T-1}\gamma_t}\sum_{t=0}^{T-1}\gamma_t x^t$.
\end{theorem}

\begin{proof}
Since our choice of stepsize is decreasing, and because we suppose that $\gamma_0 \leq \tfrac{1}{4L_{\max}}$, we deduce that $\gamma_t \leq \tfrac{1}{4L_{\max}}$, which means that the result of Theorem \ref{theo:sgdconvsmooth} apply: for $T \geq 1$,
\begin{equation*}
    \E{f(\bar{x}^T) - \inf f} \leq \frac{\norm{x^0 - x^*}^2}{\sum_{t=0}^{T-1}\gamma_t}  + 2\sigma_f^* \frac{\sum_{t=0}^{T-1}\gamma_t^2}{\sum_{t=0}^{T-1}\gamma_t}.
\end{equation*}
We will use some estimates on the sum (of squares) of the stepsizes (see Lemma \ref{L:sum integral bounds} for details on how to compute those sums):
\begin{equation*}
\sum_{t=0}^{T-1}\gamma_t^2  = \gamma_0^2 \sum_{t=1}^{T} \frac{1}{t} \; \leq \; 2 \gamma_0^2 \log(T+1)
\quad \text{ and } \quad 
\sum_{t=0}^{T-1}\gamma_t
=
\gamma_0 \sum_{t=1}^{T} \frac{1}{\sqrt{t}}
\geq \frac{4 \gamma_0}{5} \sqrt{T}.
\end{equation*}
Now combine the above inequalities to conclude that
\begin{equation*}
    \E{f(\bar{x}^T) - \inf f} \leq \frac{5\norm{x^0 - x^*}^2}{4\gamma_0 \sqrt{T}}  + \sigma_f^* \frac{5 \gamma_0 \log(T+1) }{ \sqrt{T}}.
\end{equation*}
\end{proof}

\subsection{Convergence for strongly convex and smooth functions}

\begin{theorem}\label{theo:strcnvlin}
Let Assumptions \tref{Ass:SGD sum of smooth} and \tref{Ass:SGD sum of convex} hold, and assume further that $f$ is $\mu$-strongly convex.
Consider $(x^t)_{t \in \mathbb{N}}$ the  sequence generated by the \tref{Algo:Stochastic GD} algorithm with a constant stepsize satisfying $0<\gamma<\frac{1}{2L_{\max}}$.
It follows that for $t \geq 0$, 
\begin{equation*}
\mathbb{E} \| x^t - x^* \|^2 \leq \left( 1 - \gamma \mu \right)^t \| x^0 - x^* \|^2 + \frac{2 \gamma}{\mu} \sigma_f^*. 
\end{equation*}
\end{theorem}
\begin{proof}
Let $x^* \in {\rm{argmin}}~f$, so that $\sigma_f^* = \mathbb{V}[\nabla f_i(x^*)]$ (see Lemma \ref{L:interpolation via gradient variance}).
We will note $\EE{k}{\cdot}$ instead of $\E{\cdot \ | \ x^k}$, for simplicity.
Using the definition of \tref{Algo:Stochastic GD} and expanding the squares we have
\begin{align*}
\| x^{t+1} -x^*  \|^2 & = \; \|  x^k -x^* -\gamma \nabla f_{i}(x^k) \|^2\notag\\
&\;= \; \; \|  x^k -x^* \|^2 - 2\gamma \langle  x^k -x^*, \nabla f_{i}(x^k) \rangle + \gamma^2 \|\nabla f_{i} (x^k)\|^2. \notag
\end{align*}
Taking expectation conditioned on $x^k$ we obtain
\begin{eqnarray*}
\EE{k}{\|x^{t+1} -x^* \|^2}
&{ = } &  \| x^k -x^* \|^2 - 2\gamma \langle x^k -x^*, \nabla f(x^k) \rangle + \gamma^2 \EE{k}{\|\nabla f_{i} (x^k)\|^2}\\
&\overset{Lem.~\ref{L:strong convexity differentiable hyperplans}}{ \leq } & (1- \gamma \mu) \| x^k -x^* \|^2 - 2\gamma [f(x^k)-f(x^*)]   + \gamma^2 \EE{k}{\|\nabla f_{i} (x^k)\|^2}.
\end{eqnarray*}
Taking expectations again and using a variance transfer (see \cref{L:variance transfer gradient variance}) gives
\begin{align*}
\E{\|x^{t+1}-x^*\|^2}
{ \leq } & \; (1- \gamma \mu) \Exp \| x^k -x^* \|^2 + 2 \gamma^2 \sigma_f^*   + 2\gamma (2\gamma L_{\max}- 1)  \Exp [f(x^k)-f(x^*)] \\
\leq &\; (1-\gamma \mu) \E{ \| x^k -x^* \|^2} + 2\gamma^2\sigma_f^*,
\end{align*}
where we used in the last inequality that $2\gamma L_{\max}\leq  1$ since $\gamma \leq \frac{1}{2L_{\max}}.$
Recursively applying the above and summing up the resulting geometric series gives
\begin{eqnarray*}
\mathbb{E} \|x^k -x^*\|^2 &\leq &  \left( 1 - \gamma \mu \right)^k \|x^0-x^*\|^2 + 2\sum_{j=0}^{k-1} \left( 1 - \gamma \mu \right)^j \gamma^2\sigma_f^* \nonumber\\
&\leq &   \left( 1 - \gamma \mu \right)^k \|x^0-x^*\|^2 + \frac{2\gamma \sigma_f^*}{\mu}.
\end{eqnarray*}
\end{proof}

\begin{corollary}[$\mathcal{\tilde{O}}(1/\varepsilon)$ Complexity]
Consider the setting of Theorem~\ref{theo:strcnvlin}. 
For every $\varepsilon >0$, we can guarantee that $\mathbb{E}\norm{x^T-x^*}^2 \leq \varepsilon$ provided that 
\begin{equation*}
    \gamma =  \min\left\{{\varepsilon} \frac{\mu}{4 \sigma_f^*}, \frac{1}{2L_{\max}} \right\}
    \quad \text{ and } \quad 
    T\geq \max\left\{ \frac{1}{\varepsilon}\frac{4 \sigma_f^*}{\mu^2}, \; \frac{2L_{\max}}{\mu} \right\}\log\left(\frac{2\norm{x^0-x^*}^2}{\varepsilon}\right).
\end{equation*} 
\end{corollary}

\begin{proof} 
It is a direct consequence of Lemma~\ref{lem:linear_pls_const} with $A = \frac{2 \sigma_f^*}{\mu}$, $C = 2L_{\max}$ and $\alpha_0 = \norm{x^0-x^*}^2$. 
\end{proof}

\subsection{Convergence for Polyak-Łojasiewicz and smooth functions}

\begin{theorem}
\label{theo:PLConstant}\label{T:SGD polyak-lojasiewicz}
Let Assumption \tref{Ass:SGD sum of smooth} hold, 
and assume that $f$ is $\mu$-Polyak-Łojasiewicz for some $\mu>0$. 
Consider $(x^t)_{t \in \mathbb{N}}$ a sequence generated by the \tref{Algo:Stochastic GD} algorithm, with a constant stepsize satisfying $0<\gamma \leq \frac{\mu}{L_{f}L_{\max}}$.
It follows that
\begin{equation*}\label{eq:functionTheoremExpResidual}
\Exp[f(x^{t})-\inf f] \leq (1-\gamma \mu)^t (f(x^0)-\inf f) +  \frac{\gamma L_{f}L_{\max}}{\mu}\Delta^*_f.
\end{equation*}
\end{theorem}

\begin{proof}
Remember from Assumption \tref{Ass:SGD sum of smooth} that $f$ is $L_{f}$-smooth, so we can use Lemma \ref{L:smooth upper quadratic}, together with the update rule of SGD, to obtain:
\begin{eqnarray*}
f(x^{t+1})& \leq &  f(x^{t})+ \langle \nabla f(x^t), x^{t+1}-x^t \rangle +\frac{L_{f}}{2} \| x^{t+1}-x^t\|^2 \notag\\
&=& f(x^{t})-\gamma\langle \nabla f(x^t), \nabla f_{i}(x^t) \rangle +\frac{L_{f} \gamma^2}{2} \| \nabla f_{i}(x^t)\|^2.
\end{eqnarray*}
After taking expectation conditioned on $x^t$, we can use a variance transfer lemma together with the Polyak-Łojasiewicz property to write
\begin{eqnarray*}
\E{f(x^{t+1})\; | \; x^t} & \leq & f(x^{t}) - \gamma \norm{\nabla f(x^t)}^2 + \frac{L_{f}\gamma^2}{2} \E{\norm{\nabla f_{i}(x^t)}^2\; | \; x^t}  \notag\\
& \overset{\rm Lemma~\ref{L:variance transfer function noise}}{\leq} &  f(x^{t}) - \gamma \norm{\nabla f(x^t)}^2 + \gamma^2 L_{f}L_{\max} ( f(x^t)-\inf f )    + \gamma^2 L_{f}L_{\max}   \Delta^*_f \nonumber\\
&\overset{\rm Definition~\ref{D:polyak lojasiewicz}}{ \leq }& f(x^{t})+ \gamma( \gamma L_{f}L_{\max}-2\mu)(f(x^t) -\inf f)  +\gamma^2 L_{f}L_{\max}   \Delta^*_f\nonumber \\
& \leq & 
f(x^{t}) - \gamma\mu(f(x^t) -\inf f)  +\gamma^2 L_{f}L_{\max}   \Delta^*_f,
\end{eqnarray*} 
where in the last inequality we used our assumption on the stepsize to write $\gamma L_{f}L_{\max}-2\mu \leq - \mu$.
Note that $\mu \gamma \leq 1$ because of our assumption on the stepsize, and the fact that $\mu \leq L_{f} \leq L_{\max}$ (see Remark \ref{R:mu and L}).
Subtracting $\inf f$ from both sides in the last inequality, and taking expectation, we obtain
\begin{eqnarray*}
\E{f(x^{t+1}) -\inf f} & \leq & \big(1-\mu \gamma\big)\E{f(x^t) -\inf f}  + \gamma^2 L_{f}L_{\max}    \Delta^*_f.\label{eq:tempsmo4js9j423}
\end{eqnarray*} 
Recursively applying the above and summing up the resulting geometric series gives:
\begin{eqnarray*}
\E{f(x^{t}) -\inf f} & \leq & (1-\mu \gamma)^{t}\E{f(x^0) -\inf f}
+ \gamma^2 L_{f}L_{\max}    \Delta^*_f \sum_{j=0}^{t-1} (1- \gamma \mu )^j.
\end{eqnarray*}
Using $\sum_{i=0}^{t-1} (1-\mu \gamma)^i = \frac{1-(1-\mu \gamma)^{t}}{1-1+\mu \gamma} \leq \frac{1}{\mu \gamma},$
in the above gives~\eqref{eq:functionTheoremExpResidual}.
\end{proof}

\begin{corollary}[$\mathcal{\tilde{O}}(1/\varepsilon)$ Complexity]\label{T:SGD complexity PL smooth}
Consider the setting of Theorem~\ref{T:SGD polyak-lojasiewicz}. Let $\varepsilon\geq 0$ be given. If we set 
\begin{eqnarray*}
    \gamma &= & \frac{\mu}{L_{f}L_{\max}}\min\left\{\frac{\varepsilon}{2\Delta^*_f} ,\; 1\right\}
\end{eqnarray*}
then
\begin{equation*}
    T\geq   \frac{L_{f}L_{\max}}{\mu^2  } \max\left\{ \frac{2\Delta^*_f}{\varepsilon}, \;1 \right\}\log\left(\frac{2(f(x^0)-\inf f)}{\varepsilon}\right) \quad \implies \quad  f(x^T)-\inf f \leq \varepsilon.
\end{equation*} 
\end{corollary}

\begin{proof} 
Apply lemma~\ref{lem:linear_pls_const} with $A = \frac{L_{f}L_{\max}}{\mu }\Delta^*_f$ and $\alpha_0 = f(x^0) -\inf f$. 
\end{proof}

\subsection{Convergence for smooth functions}

Here we focus on a more general setting, only assuming that the stochastic gradients are Lipschitz~\tref{Ass:SGD sum of smooth}. 
Note that in this nonconvex setting, without the Polyak-Łojasiewicz assumption we cannot prove global optimality results (remember that PŁ functions are invex, as stated in Lemma \ref{L:PL function global minimizers}). 
Nevertheless, we can still obtain bounds on the stationarity of the algorithm.

\begin{theorem}\label{T:SGD smooth nonconvex}
Let Assumption \tref{Ass:SGD sum of smooth} hold.
Consider $(x^t)_{t \in \mathbb{N}}$ a sequence generated by the \tref{Algo:Stochastic GD} algorithm, with a constant stepsize $\gamma = \sqrt{\tfrac{2}{L_f L_{\max}T}}$.
It follows that for every $T \geq 1$
    \begin{eqnarray*}
\min_{t=0, \ldots, T-1} \mathbb{E}\norm{\nabla f(x^t)}^2   
   &\leq &  
   \frac{\sqrt{2L_f L_{\max}}\left( 2(f(x^0)- \inf f) + \Delta^*_f \right)}{\sqrt{T}}.
   \end{eqnarray*}
   Consequently for a given $\varepsilon>0$ we have that
   \[ T = \mathcal{O}(\varepsilon^{-2}) \quad \implies \quad \min_{t=0, \ldots, T-1} \mathbb{E}\norm{\nabla f(x^t)}^2 = \mathcal{O}(\varepsilon) \]
\end{theorem}

\begin{proof}

From Assumption~\tref{Ass:SGD sum of smooth} we know that $f$ is $L_f$-smooth, so we can call the classical Descent Lemma  (Lemma~\ref{L:smooth upper quadratic}) together with the definition of the algorithm to write
\begin{align*}
    f(x^{t+1}) & \leq f(x^t) +\dotprod{\nabla f(x^t), x^{t+1}-x^t} +\frac{L_f}{2} \norm{x^{t+1}-x^t}^2  \\
    & =f(x^t) - \gamma_t \dotprod{\nabla f(x^t), \nabla f_{i_t}(x^t)} +\frac{\gamma_t^2 L_f}{2} \norm{ \nabla f_{i_t}(x^t)}^2.
\end{align*}
Take the expectation conditioned on $x^t$ (we will note it $\mathbb{E}_t$), and use a variance transfer (see lemma \ref{L:variance transfer function noise}) to obtain
\begin{eqnarray}
    \EE{t}{ f(x^{t+1}) }
    & \leq & f(x^t) - \gamma_t \norm{\nabla f(x^t)}^2 +\frac{\gamma_t^2 L_f}{2} \EE{t}{\norm{ \nabla f_{i_t}(x^t)}^2} \nonumber \\
    &{ \leq} &
    f(x^t) - \gamma_t \norm{\nabla f(x^t)}^2 +\gamma_t^2 L_f L_{\max} ( (f(x^t) - \inf f) + \Delta^*_f). \label{eq:tenslo8ehtz4t}
\end{eqnarray}
Subtracting $\inf f$ from both sides of~\eqref{eq:tenslo8ehtz4t}, taking the expectation and re-arranging terms gives
\begin{equation}
   \gamma_t \E{\norm{\nabla f(x^t)}^2}
     \leq 
   (1+ \gamma_t^2 L_f L_{\max})\E{f(x^t)- \inf f} -  \E{f(x^{t+1}) - \inf f} +\gamma_t^2 L_f L_{\max}  \Delta^*_f. \label{eq:tenslo8ehtz43434t}
\end{equation}
We would now like to sum both sides and telescope the suboptimality terms. But unfortunately $(f(x^t)- \inf f)$ and $(f(x^{t+1})- \inf f)$ will not telescopically cancel because of the $ (1+ \gamma_t^2 L_f L_{\max})$ multiplying $(f(x^t)- \inf f)$. Fortunately we can force these terms to telescope by using an artificial weighting scheme introduced by~\cite{Stich2019sgd}.

Let $\gamma_t$ be constant $\gamma_t \equiv \gamma$ from now on.
The idea is to multiply both sides by $\alpha_{t}$ and choose $\alpha_{t}$ such that $\alpha_{t} (1+ \gamma^2 L_f L_{\max}) = \alpha_{t-1}.$ For instance this hold for $\alpha_{-1} =1$ and $\alpha_{t} = (1+ \gamma^2 L_f L_{\max})^{-(t+1)}.$ 
Multiplying both sides of~\eqref{eq:tenslo8ehtz43434t} by $\alpha_{t}$ we have that
\begin{equation*}
   \alpha_{t}\gamma \E{\norm{\nabla f(x^t)}^2}
     \leq 
   \alpha_{t-1}\E{f(x^t)- \inf f} -  \alpha_{t}\E{f(x^{t+1}) - \inf f} +\alpha_{t}\gamma^2 L_f L_{\max}  \Delta^*_f. \label{eq:tenslo8ehtz43434sdt}
\end{equation*}
Summing from $t=0, \ldots, T-1$ on both sides and using telescopic cancellation we have that
\begin{eqnarray}
  \sum_{t=0}^{T-1} \alpha_{t}\gamma \E{\norm{\nabla f(x^t)}^2 } 
    & \leq &
   \alpha_{-1}(f(x^0)- \inf f) -  \alpha_{T-1}\E{f(x^{T}) - \inf f} +\sum_{t=0}^{T-1}\alpha_{t}\gamma^2 L_f L_{\max}  \Delta^*_f \nonumber \\
   &\leq & (f(x^0)- \inf f)  +\sum_{t=0}^{T-1}\alpha_{t}\gamma^2 L_f L_{\max}  \Delta^*_f, \nonumber\label{eq:tenslo8ehtz43434aasdt}
\end{eqnarray}
where we dropped the negative term $-  \alpha_{T}\E{f(x^{T+1}) - \inf f} $ and used that $\alpha_{-1}=1.$
Taking the minimum over the expected squared norm of the gradients and then dividing through by $ \sum_{t=0}^{T-1}\alpha_{t}\gamma $ gives\begin{eqnarray}
\min_{t=0, \ldots, T-1} \E{\norm{\nabla f(x^t)}^2}   & \leq &
  \frac{1}{\sum_{t=0}^{T-1} \alpha_{t}\gamma }\sum_{t=0}^{T-1} \alpha_{t}\gamma \E{\norm{\nabla f(x^t)}^2} \notag  \\
   &\leq &  \frac{1}{\gamma}\frac{f(x^0)- \inf f}{\sum_{t=0}^{T-1}\alpha_{t} }  +\gamma L_f L_{\max}  \Delta^*_f .\label{eq:tenslo8ehtz43434aasdrtt}
\end{eqnarray}
Next we need to find a lower bound on $\sum_{t=0}^{T-1}\alpha_{t}$ so that we can simplify the above. For this note that
\begin{align}
    \sum_{t=0}^{T-1} \alpha_{t} &= \sum_{t=0}^{T-1} \left( \frac{1}{1+ \gamma^2 L_f L_{\max}}\right)^{t+1} =  \frac{1}{1+ \gamma^2 L_f L_{\max}}\frac{1-\left( \frac{1}{1+ \gamma^2 L_f L_{\max}}\right)^T}{1-\left( \frac{1}{1+ \gamma^2 L_f L_{\max}}\right)} \nonumber \\
    &=
\frac{1}{\gamma^2 L_f L_{\max}} \left(1-\left( \frac{1}{1+ \gamma^2 L_f L_{\max}}\right)^T\right).\label{eq:tenlos8hne8rsd}
\end{align} 
For lower bounding $\sum \alpha_t$, we only need to upper bound $\left( \frac{1}{1+ \gamma^2 L_f L_{\max}}\right)^T$. To this end we choose $\gamma$ and $T$ so that
\begin{equation}\label{eq:zzerzerezr}
     \left( \frac{1}{1+ \gamma^2 L_f L_{\max}}\right)^T  \leq \frac{1}{2}
\quad \Leftrightarrow \quad  \frac{\log(2)}{\log(1+\gamma^2 L_f L_{\max})} \leq  T.
\end{equation}
Because $\log(2) \leq 1$, it is sufficient to guarantee that $\tfrac{1}{\log(1+\gamma^2 L_f L_{\max})} \leq  T$.
To simplify this expression, we use the fact (see Lemma \ref{L:log inequality 1/2}) that 
\begin{equation*}
    \frac{1}{\log(1+x)}\leq \frac{1}{x} + \frac{1}{2} , \quad \mbox{ for }x > 0 ,
\end{equation*}
thus it is sufficient to require that
\begin{equation*}
    \frac{1}{\gamma^2 L_f L_{\max}} + \frac{1}{2} \leq T
    \Leftrightarrow
    \gamma^2 \geq \frac{1}{ L_f L_{\max}} \frac{2}{2T - 1}.
\end{equation*}
Under this condition on $\gamma$ we have that~\eqref{eq:zzerzerezr} holds, and consequently from~\eqref{eq:tenlos8hne8rsd} we have
\[\sum_{t=0}^{T-1} \alpha_{t} = \frac{1}{\gamma^2 L_f L_{\max}} \left(1-\left( \frac{1}{1+ \gamma^2 L_f L_{\max}}\right)^T\right) \geq \frac{1}{2} \frac{1}{\gamma^2 L_f L_{\max}}.\]
Plugging this back into~\eqref{eq:tenslo8ehtz43434aasdrtt} means that if $\gamma \geq \sqrt{\tfrac{2}{L_f L_{\max}(2T-1)}}$ then
\begin{eqnarray*}
\min_{t=0, \ldots, T-1} \E{\norm{\nabla f(x^t)}^2 }  
   &\leq &  \gamma L_f L_{\max} \left[ 2(f(x^0)- \inf f) + \Delta^*_f \right].
\end{eqnarray*}
Because $2T -1 \geq T$ whenever $T \geq 1$, we see that we can choose $\gamma = \sqrt{\tfrac{2}{L_f L_{\max}T}}$ which finally gives  
\begin{eqnarray*}
    \min_{t=0, \ldots, T-1} \E{\norm{\nabla f(x^t)}^2 } 
    & \leq &
    \sqrt{\frac{2L_f L_{\max}}{T}}\left[ 2(f(x^0)- \inf f) + \Delta^*_f \right].
\end{eqnarray*}
\end{proof}


\subsection{Bibliographic notes}

The early and foundational works on SGD include \cite{robbins1951stochastic, NemYudin1978, NemYudin1983book, Pegasos, Nemirovski-Juditsky-Lan-Shapiro-2009, HardtRechtSinger-stability_of_SGD}, though these references are either for the non-smooth setting for Lipschitz losses,  or are asymptotic.
The first non-asymptotic analyses of SGD the smooth and convex setting that we are aware of is in \cite{moulines2011non}, closely followed by~\cite{Schmidt-and-roux-2013} under a different growth assumption. These results were later improved in \cite{needell2014stochastic}, where the authors removed the quadratic dependency on the smoothness constant and considered importance sampling. 
The proof of Theorem \ref{theo:sgdconvsmooth} is a simplified version of~\cite[Theorem D.6]{SGDstruct}.
The proof of Theorem \ref{theo:strcnvlin} is a simplified version of~\cite[Theorem 3.1]{gower2019sgd}.
The proof of Theorem \ref{T:SGD polyak-lojasiewicz} has been adapted from the proof of \cite[Theorem 4.6]{SGDstruct}.

For a general convergence theory for SGD in the smooth and non-convex setting we recommend~\cite{Khaled-nonconvex-2020}. 
Also, the definition of function noise that we use here was taken from~\cite{Khaled-nonconvex-2020}. The first time we saw Lemma~\ref{L:variance transfer function noise} was also in~\cite{Khaled-nonconvex-2020}.  
Theorem~\ref{T:SGD polyak-lojasiewicz}, which relies on the Polyak-Łojasiewicz condition, is based on the proof in~\cite{SGDstruct}, with the only different being that we use function noise as opposed to gradient noise. This  Theorem~\ref{T:SGD polyak-lojasiewicz} is also very similar to Theorem 3 in ~\cite{Khaled-nonconvex-2020}, with the difference being that  Theorem 3 in ~\cite{Khaled-nonconvex-2020} is more general (uses weaker assumptions), but also has a more involved proof and a different step size.
The proof of Theorem \ref{T:SGD smooth nonconvex}  is essentially\footnote{We make a minor modification in the proof when bounding the weighting sequence $\alpha_t$  because in our setting $B=0$.} a special case of the one in~\cite[Theorem 2]{Khaled-nonconvex-2020}, which we combined with a weighted telescoping technique from~\cite{Stich2019sgd}.

An excellent reference for proof techniques for SGD  focused on the online setting is the recent book~\cite{orabona2019modern}, which contains proofs for adaptive step sizes such a Adagrad and coin tossing based step sizes.

\begin{remark}[From finite sum to expectation]
The theorems we prove here 
can easily be extended to the case when the objective is a true expectation
of the form
$$ f(x)  = \EE{\xi\sim \mathcal{D}}{f_\xi(x)} .$$
To adapt the results, we need to define the $L_{\max}$ smoothness as the largest smoothness constant of every $f_\xi(\cdot)$ for every $\xi$. The gradient noise $\sigma_f^*$ is would now be given by
\[\sigma^*_f \;\eqdef\;  \inf_{x^* \in {\rm{argmin}}~f} \EE{\xi}{ \norm{\nabla f_\xi(x^*)}^2}. \]
The function noise would now be given by
\[ \Delta^*_f \; \eqdef\; \inf f - \EE{\xi}{\inf f_\xi} \]
With these extended definitions we have that Theorems~\ref{T:SGD convex smooth constant stepsize},~\ref{T:SGD convex smooth vanishing stepsize},~\ref{theo:strcnvlin} and  \ref{theo:PLConstant}  hold verbatim.
In Section~\ref{sec:stochsubgrad} we study this problem in detail by considering more general nonsmooth functions.
\end{remark}

\section{Minibatch SGD}
\label{sec:mini}

\subsection{Definitions}

When solving \tref{Pb:sum of functions} in  practice, an estimator of the gradient is often computed using a small batch of functions, instead of a single one as in \tref{Algo:Stochastic GD}. 
More precisely, given a subset $B \subset\{1,\ldots, n\}$, we want to make use of
\begin{equation*}
    \nabla f_B(x^t) \eqdef \frac{1}{|B|} \sum_{i \in B} \nabla f_i(x^t).
\end{equation*}
This leads to the \emph{minibatching} SGD algorithm:
\begin{algorithm}[MiniSGD]\label{Algo:SGD minibatch}
    Let $x^0 \in \mathbb{R}^d$, let a batch size $b \in \{1,\dots,n\}$, and let $\gamma_t>0$ be a sequence of step sizes.
The \textbf{Minibatching Stochastic Gradient Descent (MiniSGD)} algorithm is given by the iterates $(x^t)_{t \in \mathbb{N}}$ where
\begin{align*}
    B_t & \subset \{1,\ldots n\} & \mbox{Sampled uniformly among sets of size $b$}\nonumber \\
    x^{t+1} & = x^t - \gamma_t \nabla f_{B_t}(x^t).
\end{align*}
\end{algorithm}

\begin{remark}[Mini-batch distribution]\label{R:minibatch random batches}
    We impose in this section that the batches $B$ are sampled uniformly among all subsets of size $b$ in $\{1,\dots, n\}$. 
    This means that each batch is sampled with probability
    \begin{eqnarray*}
        \frac{1}{\binom{n}{b}} = \frac{(n-b)!b!}{n!},
    \end{eqnarray*}
    and that we will compute expectation and variance with respect to this uniform law.
    For instance the expectation of the minibatched gradient writes as
    \begin{equation*}
        \mathbb{E} \left[ \nabla f_B(x) \right] = \frac{1}{\binom{n}{b}} \sum\limits_{\substack{B \subset \{1,\dots, n\}\\ \vert B \vert = b}} \nabla f_B(x),
    \end{equation*}
    and it is an exercise to verify that this is exactly equal to $\nabla f(x)$.
\end{remark}

Mini-batching makes better use of parallel computational resources and it can also improve the complexity of \tref{Algo:Stochastic GD}, as we show next. 
To do so, we will need the same central tools than for \tref{Algo:Stochastic GD}, that is the notions of gradient noise, of expected smoothness, and a variance transfer lemma.

\begin{definition}\label{D:gradient solution variance minibatch}
Let Assumption \tref{Ass:SGD sum of smooth} hold, and let $b \in \{1,\dots,n\}$.
We define the \textbf{minbatch gradient noise} as
\begin{equation*}
    \sigma^{*}_b \eqdef \inf\limits_{x^* \in {\rm{argmin}}~f} \ \mathbb{V}\left[ \nabla f_B(x^*) \right],
\end{equation*}
where $B$ is sampled according to Definition~\ref{R:minibatch random batches}.
\end{definition}

\begin{definition}\label{D:expected smoothness minibatch}
Let Assumption \tref{Ass:SGD sum of smooth} hold, and let $b \in \{1,\dots,n\}$.
We say that $f$ is $\mathcal{L}_{b}$-\textbf{smooth in expectation} if
\begin{equation*}\label{eq:expsmooth minibatch}
\text{for all $x,y \in \mathbb{R}^d$}, \quad
\frac{1}{2\mathcal{L}_b}\EE{}{\norm{\nabla f_B(y) - \nabla f_B(x) }^2 } \; \leq \; f(y) - f(x) -\dotprod{\nabla f(x), y-x},
\end{equation*}
where $B$ is sampled according to Definition~\ref{R:minibatch random batches}.
\end{definition}

\begin{lemma}[From single batch to minibatch]\label{L:minibatch constants formula}
    Let Assumptions \tref{Ass:SGD sum of smooth} and \tref{Ass:SGD sum of convex} hold.
    Then $f$ is $\mathcal{L}_b$-smooth in expectation with 
    \begin{equation} \label{eq:Lbmini}
        \cL_b = \frac{n(b-1)}{b(n-1)}L + \frac{n-b}{b(n-1)}L_{\max},
    \end{equation}
    and the minibatch gradient noise can be computed via
    \begin{equation} \label{eq:sigmini}
        \sigma_b^* = \frac{n-b}{b(n-1)} \sigma_f^*. 
    \end{equation}
\end{lemma}

\begin{proof}
See Proposition 3.8 in~\cite{gower2019sgd}.
\end{proof}

\begin{remark}[Minibatch interpolates between single and full batches]
    It is intersting to look at variations of the expected smoothness constant $\mathcal{L}_b$ and minibatch gradient noise $\sigma_b^*$ when $b$ varies from $1$ to $n$.
    For $b=1$, where \tref{Algo:SGD minibatch} reduces to \tref{Algo:Stochastic GD}, we have that $\cL_b= L_{\max}$ and $\sigma_b^* = \sigma_f^*$, which are the constants governing the complexity of \tref{Algo:Stochastic GD} as can be seen in Section \ref{sec:SGD}.
    On the other extreme, when $b=n$ \tref{Algo:SGD minibatch} reduces to \tref{Algo:gradient descent constant stepsize}, we see that $\cL_b = L$ and $\sigma_b^* =0$. We recover the fact that the behavior of \tref{Algo:gradient descent constant stepsize} is controlled by the Lipschitz constant $L$, and has no variance.
\end{remark}
We end this presentation with a variance transfer lemma, analog to Lemma~\ref{L:variance transfer gradient variance}  (resp. Lemma~\ref{lem:convandsmooth}) in the single  batch (resp. full batch).

\begin{lemma}
\label{L:variance transfer minibatch}
Let Assumptions \tref{Ass:SGD sum of smooth} and \tref{Ass:SGD sum of convex} hold.
It follows that
\begin{equation*}
\EE{}{\|\nabla f_B (x)\|^2 } \leq  4 \cL_b ( f(x)-\inf f )    +2 \sigma_b^*.
\end{equation*}
\end{lemma}

\begin{proof}
This is the very same proof as Lemma~\ref{L:variance transfer gradient variance}, see also Proposition 3.10 in~\cite{gower2019sgd}.
\end{proof}

\subsection{Convergence for convex and smooth functions}

\begin{theorem}\label{T:SGD minibatch CV convex smooth general stepsize}
Let Assumptions \tref{Ass:SGD sum of smooth} and \tref{Ass:SGD sum of convex} hold.
Consider $(x^t)_{t \in \mathbb{N}}$ a sequence generated by the \tref{Algo:SGD minibatch} algorithm, with a sequence of stepsizes satisfying $0<\gamma_t\leq\frac{1}{4\cL_b}$.
It follows that for every $T \geq 1$, $x^* \in {\rm{argmin}}~f$ and $\bar x^T \eqdef \frac{1}{\sum_{t=0}^{T-1}\gamma_t} \sum_{t=0}^{T-1}\gamma_t x^t$,
\begin{eqnarray*}
\E{f(\bar{x}^T) - \inf f} \leq 
\frac{\norm{x^0 - x^*}^2}{\sum_{t=0}^{T-1}\gamma_t} + \frac{2\sigma_b^*\sum_{t=0}^{T-1}\gamma_t^2}{\sum_{t=0}^{T-1}\gamma_t}.
\end{eqnarray*}
\end{theorem}

\begin{proof}
Let $x^* \in {\rm{argmin}}~f$, so we have $\sigma_b^* = \mathbb{V}[\nabla f_B(x^*)]$.
Let us start by analyzing the behaviour of $\norm{x^t - x^*}^2$. By developing the squares, we obtain
\begin{align*}
\norm{x^{t+1} - x^*}^2 &= \norm{x^t - x^*}^2 - 2\gamma_t\langle \nabla f_{B_t}(x^t), x^t - x^*\rangle + \gamma_t^2\norm{\nabla f_{B_t}(x^t)}^2
\end{align*}
Hence, after taking the expectation conditioned on $x^t$, we can use the convexity of $f$ (recall Lemma~\ref{L:convexity via hyperplanes}) and a variance transfer lemma (see Lemma~\ref{L:variance transfer minibatch}) to write  
\begin{eqnarray*}
\EE{}{\norm{x^{t+1} - x^*}^2 \ | \ x^t} &=& \norm{x^t - x^*}^2 + 2\gamma_t\langle \nabla f(x^t), x^* - x^t\rangle + \gamma_t^2\EE{}{\norm{\nabla f_{B_t}(x^t)}^2 \ | \ x^t}\\
& \leq &  \norm{x^t - x^*}^2 - 2\gamma_t (f(x^t) - \inf f) + \gamma_t^2\EE{}{\norm{\nabla f_{B_t}(x^t)}^2 \ | \ x^t} \\
&{\leq}& \norm{x^t - x^*}^2 + 2\gamma_t(2\gamma_t\mathcal{L}_b - 1)(f(x^t) - \inf f)) + 2\gamma_t^2\sigma_b^* \\
& \leq & 
\norm{x^t - x^*}^2 - \gamma_t (f(x^t) - \inf f)) + 2\gamma_t^2\sigma_b^*
\end{eqnarray*}
where in the last inequality we used the fact that $\gamma_t \leq \tfrac{1}{4 \mathcal{L}_b}$.
Rearranging and taking expectation, we have
\begin{align*}
\gamma_t \EE{}{f(x^t) - \inf f} \leq \EE{}{\norm{x^t - x^*}^2} - \EE{}{\norm{x^{t+1} - x^*}^2} + 2\gamma_t^2\sigma_b^*.
\end{align*}
Summing over $t =0,\ldots, T-1$ and using telescopic cancellation gives
\begin{align*}
\sum_{t=0}^{T-1}\gamma_t \EE{}{f(x^t) - \inf f} 
\leq  
\norm{x^0 - x^*}^2 - \EE{b}{\norm{x^{T} - x^*}^2} + 2\sigma_b^*\sum_{t=0}^{T-1}\gamma_t^2.
\end{align*}
Since $\EE{}{\norm{x^{T} - x^*}^2} \geq 0$, dividing both sides by $\sum_{t=0}^{T-1}\gamma_t$ gives:
\begin{align*}
\frac{1}{\sum_{t=0}^{T-1}\gamma_t}\sum_{t=0}^{T-1} \gamma_t\EE{}{f(x^t) - \inf f} 
\leq 
\frac{\norm{x^0 - x^*}^2}{\sum_{t=0}^{T-1}\gamma_t} + \frac{2\sigma_b^*\sum_{t=0}^{T-1}\gamma_t^2}{\sum_{t=0}^{T-1}\gamma_t}.
\end{align*}
Finally, define $\bar x^T \eqdef \frac{1}{\sum_{t=0}^{T-1}\gamma_t} \sum_{t=0}^{T-1}\gamma_t x^t$ and 
use that $f$ is convex together with Jensen's inequality to conclude
\begin{align*}
\EE{}{f(\bar{x}^T) - \inf f}  
&\leq 
\EE{}{\frac{1}{\sum_{t=0}^{T-1}\gamma_t}\sum_{t=0}^{T-1}\gamma_t (f(x^t) - \inf f)}  \\
&\leq 
\frac{\norm{x^0 - x^*}^2}{\sum_{t=0}^{T-1}\gamma_t} + \frac{2\sigma_b^*\sum_{t=0}^{T-1}\gamma_t^2}{\sum_{t=0}^{T-1}\gamma_t}.
\end{align*}
\end{proof}

\begin{theorem}\label{T:SGD minibatch CV convex smooth constant stepsize}
Let Assumptions \tref{Ass:SGD sum of smooth} and \tref{Ass:SGD sum of convex} hold.
Consider $(x^t)_{t \in \mathbb{N}}$ a sequence generated by the \tref{Algo:SGD minibatch} algorithm, with a sequence of constant stepsizes  $\gamma_t \equiv \gamma \leq\frac{1}{4\cL_b}$.
It follows that for every $T \geq 1$, $x^* \in {\rm{argmin}}~f$ and $\bar x^T \eqdef \frac{1}{T} \sum_{t=0}^{T-1} x^t$,
\begin{equation*}
\E{f(\bar{x}^T) - \inf f} \leq 
\frac{\norm{x^0 - x^*}^2}{\gamma T} + {2\gamma \sigma_b^*}.
\end{equation*}
In particular, if for a fixed horizon $T \geq 1$ we set $\gamma = \tfrac{\gamma_0}{\sqrt{T}}$ for some $\gamma_0 \leq \tfrac{1}{4 \mathcal{L}_b}$, then
\begin{equation*}
\E{f(\bar{x}^T) - \inf f} \leq 
\frac{\norm{x^0 - x^*}^2}{\gamma_0 \sqrt{T}} + \frac{2\gamma_0 \sigma_b^*}{\sqrt{T}} = \mathcal{O}\left( \frac{1}{\sqrt{T}} \right).
\end{equation*}
\end{theorem}

\begin{proof}
This is a direct consequence of Theorem \ref{T:SGD minibatch CV convex smooth general stepsize}, since $\sum_{t=0}^{T-1} \gamma_t = \gamma T$ and $\sum_{t=0}^{T-1} \gamma_t^2 = \gamma^2 T$.
\end{proof}

\begin{corollary}[$\mathcal{O}(1/\varepsilon^2)$ Complexity]\label{T:SGD minibatch complexity convex smooth}
Consider the setting of Theorem~\ref{T:SGD minibatch CV convex smooth constant stepsize}.
For every $\varepsilon > 0 $, we can guarantee that $\E{f(\bar{x}^T) - \inf f} \leq \varepsilon$ provided that
	\begin{equation*}
	\gamma = \frac{\gamma_0}{\sqrt{T}}, \
	\gamma_0 = \min \left\{ \frac{1}{4\mathcal{L}_b}, \frac{\norm{x^0 - x^*}}{\sqrt{2\sigma_b^*}} \right\},
	\ \text{ and } \ 
	T \geq \left(  \norm{x^0 - x^*} \sqrt{\sigma_b^*} + \norm{x^0 - x^*}^2 \mathcal{L}_b  \right)^2 \frac{16}{\varepsilon^2}.
	\end{equation*}
\end{corollary}

\begin{proof}
This is a direct consequence of Theorem \ref{T:SGD minibatch CV convex smooth constant stepsize} and Lemma \ref{L:complexity meta convex} with $A = \Vert x^0 - x^* \Vert^2$, $B = 2 \sigma_b^*$ and $C = 4 \mathcal{L}_b$.
\end{proof}

\begin{theorem}\label{T:SGD minibatch convex smooth vanishing stepsize}
Let Assumptions \tref{Ass:SGD sum of smooth} and \tref{Ass:SGD sum of convex} hold.
Consider $(x^t)_{t \in \mathbb{N}}$ a sequence generated by the \tref{Algo:SGD minibatch} algorithm with a vanishing stepsize $\gamma_t = \tfrac{\gamma_0}{\sqrt{t+1}}$ where $\gamma_0 \leq \tfrac{1}{4\mathcal{L}_b}$.
Then for every $T \geq 1$,
\begin{equation*}
    \E{f(\bar{x}^T) - \inf f} \leq \frac{5\norm{x^0 - x^*}^2}{4\gamma_0 \sqrt{T}}  + \sigma_b^* \frac{5 \gamma_0 \log(T+1) }{ \sqrt{T}} = {\mathcal{O}} \left( \frac{\log(T+1)}{\sqrt{T}} \right),
\end{equation*}
where $\bar{x}^T \; \eqdef \; \tfrac{1}{\sum_{t=0}^{T-1}\gamma_t}\sum_{t=0}^{T-1}\gamma_t x^t$.
\end{theorem}

\begin{proof}
Since our choice of stepsize is decreasing, and because we suppose that $\gamma_0 \leq \tfrac{1}{4L_{\max}}$, we deduce that $\gamma_t \leq \tfrac{1}{4L_{\max}}$, which means that the result of Theorem \ref{T:SGD minibatch CV convex smooth general stepsize} apply: for $T \geq 1$,
\begin{equation*}
    \E{f(\bar{x}^T) - \inf f} \leq \frac{\norm{x^0 - x^*}^2}{\sum_{t=0}^{T-1}\gamma_t}  + 2\sigma_b^* \frac{\sum_{t=0}^{T-1}\gamma_t^2}{\sum_{t=0}^{T-1}\gamma_t}.
\end{equation*}
We will use some estimates on the sum (of squares) of the stepsizes (see Lemma \ref{L:sum integral bounds} for details on how to compute those sums):
\begin{equation*}
\sum_{t=0}^{T-1}\gamma_t^2  = \gamma_0^2 \sum_{t=1}^{T} \frac{1}{t} \; \leq \; 2 \gamma_0^2 \log(T+1)
\quad \text{ and } \quad 
\sum_{t=0}^{T-1}\gamma_t
=
\gamma_0 \sum_{t=1}^{T} \frac{1}{\sqrt{t}}
\geq \frac{4 \gamma_0}{5} \sqrt{T}.
\end{equation*}
Now combine the above inequalities to conclude that
\begin{equation*}
    \E{f(\bar{x}^T) - \inf f} \leq \frac{5\norm{x^0 - x^*}^2}{4\gamma_0 \sqrt{T}}  + \sigma_b^* \frac{5 \gamma_0 \log(T+1) }{ \sqrt{T}}.
\end{equation*}
\end{proof}

\subsection{Rates for strongly convex and smooth functions}

\begin{theorem}\label{theo:strcnvlinmini}
Let Assumptions \tref{Ass:SGD sum of smooth} and \tref{Ass:SGD sum of convex} hold, and assume further that $f$ is $\mu$-strongly convex. 
Consider $(x^t)_{t \in \mathbb{N}}$ a sequence generated by the \tref{Algo:SGD minibatch} algorithm, with a constant sequence of stepsizes $\gamma_t \equiv \gamma \in ]0,\frac{1}{2\cL_b}]$. 
Then
\begin{equation*}
\mathbb{E}_b\left[  \| x^t - x^* \|^2 \right] \leq \left( 1 - \gamma \mu \right)^t \| x^0 - x^* \|^2 + \frac{2 \gamma \sigma_b^*}{\mu}. 
\end{equation*}

\end{theorem}

\begin{proof}
Let $x^* \in {\rm{argmin}}~f$, so that $\sigma_b^* = \mathbb{V}_b[\nabla f_B(x^*)]$.
Expanding the squares we have
\begin{eqnarray*}
\| x^{t+1} -x^*  \|^2 &\overset{\tref{Algo:SGD minibatch}}{ =} & \|  x^t -x^* -\gamma \nabla f_{B_t}(x^t) \|^2\notag\\
&= & \|  x^t -x^* \|^2 - 2\gamma \langle  x^t -x^*, \nabla f_{B_t}(x^t) \rangle + \gamma^2 \|\nabla f_{B_t} (x^t)\|^2. \notag
\end{eqnarray*}
Taking expectation conditioned on $x^t$ and using $\mathbb{E}_b \left[ \nabla f_B(x) \right] = \nabla f(x)$ (see Remark \ref{R:minibatch random batches}), we obtain 
\begin{eqnarray*}
\EE{b}{\|x^{t+1} -x^* \|^2 \ | \ x^t}
&{ = } &  \| x^t -x^* \|^2 - 2\gamma \langle x^t -x^*, \nabla f(x^t) \rangle + \gamma^2 \EE{b}{\|\nabla f_{B_t} (x^t)\|^2 \ | \ x^t}\\
&\overset{Lem. \ref{L:strong convexity differentiable hyperplans}}{ \leq } &  (1- \gamma \mu) \| x^t -x^* \|^2 - 2\gamma [f(x^t)-\inf f]   + \gamma^2 \EE{b}{\|\nabla f_{B_t} (x^t)\|^2 \ | \ x^t}.
\end{eqnarray*}
Taking expectations again and using Lemma \ref{L:variance transfer minibatch} gives
\begin{eqnarray*}
\EE{b}{\|x^{t+1}-x^*\|^2}
& { \leq } &  (1- \gamma \mu) \EE{b}{ \| x^t -x^* \|^2} + 2 \gamma^2 \sigma_b^*   + 2\gamma (2\gamma \mathcal{L}_b- 1)  \Exp [f(x^t)-\inf f] \\
& \leq &  (1-\gamma \mu) \EE{b}{ \| x^t -x^* \|^2} + 2\gamma^2\sigma_b^*,
\end{eqnarray*}
where we used in the last inequality that $2\gamma \mathcal{L}_b\leq  1$ since $\gamma \leq \frac{1}{2\mathcal{L}_b}.$
Recursively applying the above and summing up the resulting geometric series gives
\begin{eqnarray*}
\EE{b}{ \|x^t -x^*\|^2} &\leq &  \left( 1 - \gamma \mu \right)^t \|x^0-x^*\|^2 + 2\sum_{k=0}^{t-1} \left( 1 - \gamma \mu \right)^k \gamma^2\sigma_b^* \nonumber\\
&\leq &   \left( 1 - \gamma \mu \right)^t \|x^0-x^*\|^2 + \frac{2\gamma \sigma_b^*}{\mu}.
\end{eqnarray*}
\end{proof}

\begin{corollary}[$\mathcal{\tilde{O}}(1/\varepsilon)$ Complexity]
Consider the setting of Theorem~\ref{theo:strcnvlinmini}. Let $\varepsilon > 0$ be given.
Hence, given any $\varepsilon>0$, choosing  stepsize
\begin{equation*}
\gamma  = \min \left\{ \frac{1}{2\cL_b},\; \frac{\varepsilon\mu}{4 \sigma_b^*}\right\},
\end{equation*}
and 
\begin{equation*}
t\geq  \max \left\{ \frac{2\cL_b}{\mu },\; \frac{4 \sigma_b^*}{\varepsilon\mu^2}\right\} \log\left(\frac{ 2 \|x^0 - x^*\|^2 }{  \varepsilon }\right) \quad \implies  \mathbb{E} \| x^t - x^* \|^2  \leq \varepsilon.
\end{equation*}
\end{corollary}

\begin{proof} 
Apply Lemma~\ref{lem:linear_pls_const} with $A = \frac{2 \sigma_b^*}{\mu}$, $C = 2\cL_b$ and $\alpha_0 = \norm{x^0-x^*}^2$. 
\end{proof}

\subsection{Bibliographic Notes}

The SGD analysis in~\cite{needell2014stochastic} was later extended to a mini-batch analysis~\cite{batchSGDNW16}, but restricted to mini-batches that are disjoint partitions of the data. 
Our results on mini-batching in Section~\ref{sec:mini}  are instead taken from~\cite{gower2019sgd}.  We choose to 
adapt the proofs from~\cite{gower2019sgd} since these proofs allow for sampling with replacement. 
The smoothness constant in~\eqref{eq:Lbmini}  was introduced in~\cite{GowerRichBach2018} and this particular formula was conjectured in~\cite{SAGAminib}.

\section{Stochastic Momentum}
\label{sec:mom}

For most, if not all, machine learning applications SGD is used with \emph{momentum}. 
In the machine learning community, the \emph{momentum method} is often  written as follows
\begin{algorithm}[Momentum]\label{Algo:momentum}
    Let Assumption \tref{Ass:SGD sum of smooth} hold.
    Let $x^0 \in \mathbb{R}^d$ and $m^{-1}=0$, let $(\gamma_t)_{t\in \mathbb{N}} \subset ]0,+\infty[$ be a sequence of stepsizes, and let $(\beta_t)_{t \in \mathbb{N}} \subset [0,1]$ be a sequence of momentum parameters.
    The \textbf{Momentum} algorithm defines a sequence $(x^t)_{t \in \mathbb{N}}$ satisfying for every $t \in \mathbb{N}$
    \begin{align*}
    m^t & = \beta_t m^{t-1} + \nabla f_{i_t}(x^t),  \nonumber\\
    x^{t+1} & = x^t - \gamma_t m^t. 
\end{align*}
\end{algorithm}

At the end of this section we will see in  Corollary~\ref{theo:momentumconv} that in the convex setting, the sequence $x^t$ generated by the \tref{Algo:momentum} algorithm has a complexity rate of $\mathcal{O}(1/\varepsilon^2)$.  
This is an improvement with respect to \tref{Algo:Stochastic GD}, for which we only know complexity results about the  \emph{average of the iterates}, see Corollary \ref{T:SGD complexity convex smooth}.

\subsection{The many ways of writing momentum}

In the optimization community the momentum method is often written in the \emph{heavy ball} format which is
\begin{equation}
 x^{t+1} = x^t -\hat \gamma_t \,\nabla f_{i_t}(x^t)  + \hat{\beta}_t (x^t -x^{t-1}),\label{eq:heavyball}
\end{equation}
where $\hat{\beta}_t \in [0,\;1] $ is another momentum parameter, $i_t \in \{1,\ldots, n\}$ is sampled uniformly and i.i.d at each iteration.
These two ways of writing down momentum in~\tref{Algo:momentum} and~\eqref{eq:heavyball} are equivalent, as we show next.

\begin{lemma}\label{L:momentum is heavy ball}
The algorithms~\tref{Algo:momentum} and Heavy Ball (given by~\eqref{eq:heavyball}) are the equivalent.
More precisely, if $(x^t)_{t\in \mathbb{N}}$ is generated by \tref{Algo:momentum} from parameters $\gamma_t,\beta_t$, then it verifies \eqref{eq:heavyball} by taking $\hat \gamma_t = \gamma_t$ and $\hat{\beta}_t = \frac{\gamma_t \beta_t}{\gamma_{t-1}}$, assuming $\gamma_{-1}=1$ and $x^{-1} = x^0$.
\end{lemma}

\begin{proof}
Let $t \geq 1$.
Starting from~\tref{Algo:momentum} we have that
\begin{eqnarray*}
x^{t+1} &=& x^t - \gamma_t m^t \\
&{=} &x^t - \gamma_t\beta_t m^{t-1} - \gamma_t \nabla f_{i_t}(x^t).
\end{eqnarray*}
Using~\tref{Algo:momentum} at time $t-1$ we have that $m^{t-1} = \frac{x^{t-1} -x^t}{\gamma_{t-1}}$ which when inserted in the above gives
\begin{eqnarray*}
x^{t+1} &=& x^t -  \frac{\gamma_t\beta_t}{\gamma_{t-1}}(x^{t-1} -x^t) - \gamma_t \nabla f_{i_t}(x^t).
\end{eqnarray*}
The conclusion follows by taking $\hat \gamma_t = \gamma_t$ and $\hat{\beta}_t = \frac{\gamma_t \beta_t}{\gamma_{t-1}}$.
In the particular case $t =0$, we see from \tref{Algo:momentum} and the assumption $m^{-1}=0$ that $x^1=x^0 - \gamma_0 \nabla f_{i_0}(x^0)$.
So it is enough to take $x^{-1} = x^0$ and $\hat \beta_0$ can be of any value.
\end{proof}

There is yet a third equivalent way of writing down the momentum method that will be useful in establishing convergence.

\begin{lemma}
\label{L:momentum is IMA}
The algorithm~\tref{Algo:momentum} is equivalent to the following \emph{iterate-moving-average} (IMA) algorithm : start from $z^{-1} = x^0$ and iterate for $t \in \mathbb{N}$
\begin{eqnarray}
z^{t} & = & z^{t-1}-\eta_{t}\nabla f_{i_t}(x^{t}),\label{eq:zup}\\
x^{t+1} & = & \frac{\lambda_{t+1}}{\lambda_{t+1}+1}x^{t}+\frac{1}{\lambda_{t+1}+1}z^{t}. \label{eq:SHB_IMA}
\end{eqnarray}
More precisely, if $(x^t)_{t\in \mathbb{N}}$ is generated by \tref{Algo:momentum} from parameters $\gamma_t,\beta_t$, then it verifies (IMA) by chosing any parameters $(\eta_t, \lambda_t)$ and a vector $z^t$ satisfying
\begin{equation*}
    \beta_t \lambda_{t+1} = \frac{\gamma_{t-1} \lambda_t}{\gamma_t} - \beta_t , \quad
    \eta_t=(1+ \lambda_{t+1}) \gamma_t,
    \quad \text{ and } \quad 
    z^{t} = x^{t+1} + \lambda_{t+1}(x^{t+1} - x^{t}).
\end{equation*}
\end{lemma}

\begin{proof}
Let $(x^t)_{t\in \mathbb{N}}$ be generated by \tref{Algo:momentum} from parameters $\gamma_t,\beta_t$.
Let $z^{-1}=x^0$ and for every $t \geq0$ define $z^{t} := x^{t+1} + \lambda_{t+1}(x^{t+1} - x^{t})$.
By definition, we have
\begin{equation}\label{eq:temso8hzo4}
    z^t = (1 + \lambda_{t+1})x^{t+1} - \lambda_{t+1}x^t,
\end{equation}
which after dividing by $(1 + \lambda_{t+1})$ directly gives us \eqref{eq:SHB_IMA}.
Now use Lemma \ref{L:momentum is heavy ball} to write that \begin{equation*}
     x^{t+1} = x^t - \gamma_t \,\nabla f_{i_t}(x^t)  + \hat{\beta}_t (x^t -x^{t-1})
\end{equation*} 
where $\hat{\beta}_t = \frac{\gamma_t \beta_t}{\gamma_{t-1}}$.
Going back to the definition of $z^t$, we can write
\begin{eqnarray*}
    z^t & =&
    (1 + \lambda_{t+1})x^{t+1} - \lambda_{t+1}x^t \\
    &=&
    (1 + \lambda_{t+1})(x^t - \gamma_t \,\nabla f_{i_t}(x^t)  + \hat{\beta}_t (x^t -x^{t-1})) - \lambda_{t+1}x^t  \\
    &=&
    x^t - (1 + \lambda_{t+1})\gamma_t \,\nabla f_{i_t}(x^t)  + (1 + \lambda_{t+1})\hat{\beta}_t (x^t -x^{t-1}) \\
    &=& (1+ \lambda_t)x^t- \lambda_t x^{t-1} -\eta_{t}\nabla f_{i_t}(x^{t})\\
    &\overset{ \eqref{eq:temso8hzo4}}{=} & z^{t-1}-\eta_{t}\nabla f_{i_t}(x^{t}), 
\end{eqnarray*}
where in the last but one equality we used the fact that
\begin{equation*}
    (1 + \lambda_{t+1})\gamma_t = \eta_t
    \quad \text{ and } \quad 
    (1 + \lambda_{t+1})\hat{\beta}_t
    =
    (1 + \lambda_{t+1})\frac{\gamma_t \beta_t}{\gamma_{t-1}}=\lambda_t.
\end{equation*}

\end{proof}

\subsection{Convergence for convex and smooth functions}

\begin{theorem}\label{theo:momentumconv}
Let Assumptions \tref{Ass:SGD sum of smooth} and \tref{Ass:SGD sum of convex} hold.
Consider $(x^t)_{t \in \mathbb{N}}$ the iterates generated by the~\tref{Algo:momentum} algorithm with stepsize and momentum parameters taken according to
\begin{align*}
    \gamma_t =  \frac{2\eta}{t+3}, \quad
    \beta_t = \frac{t}{t+2},
    \quad \text{ with } \quad 
    \eta \leq \frac{1}{4 L_{\max}}.
\end{align*}
Then the iterates converge according to
 \begin{equation*}
 \E{f(x^{t}) - \inf f} \leq \frac{\norm{x_0 - x^*}^2}{\eta \br{t+1}} + 2\eta \sigma_f^*. 
\end{equation*}
\end{theorem}

\begin{proof}
For the proof, we rely on the iterate-moving-average (IMA) viewpoint of momentum given in Lemma \ref{L:momentum is IMA}.
It is easy to verify that the parameters
\begin{equation*}
    \eta_t = \eta, \quad \lambda_t = \frac{t}{2}
    \quad \text{ and } \quad 
    z^{t-1} = x^{t} + \lambda_{t}(x^{t} - x^{t-1})
\end{equation*}
verify the conditions of Lemma \ref{L:momentum is IMA}. 
Let us then consider the iterates $(x^t, z^t)$ of (IMA), and we start by studing the variations of $\norm{z^{t} - x^*}^2$.
Expanding squares we have for $t \in \mathbb{N}$ that
\begin{eqnarray*}
\norm{z^{t} - x^*}^2 &\overset{~}{=}& \norm{z^{t-1} - x^* - \eta \nabla f_{i_t}(x^{t})}^2 \nonumber \\
&\overset{~}{=} & 
\norm{z^{t-1} - x^*}^2 
+ 2 \eta \langle \nabla f_{i_t}(x^{t}), x^* - z^{t-1} \rangle   
+ \eta^2 \norm{\nabla f_{i_t}(x^{t}) }^2\nonumber \\
&\overset{\eqref{eq:SHB_IMA}}{=}& \norm{z^{t-1} - x^*}^2 
+ 2 \eta \langle \nabla f_{i_t}(x^{t}), x^* - x^{t} \rangle 
+ 2 \eta\lambda_t \langle \nabla f_{i_t}(x^{t}), x^{t-1} - x^{t}  \rangle 
+  \eta^2 \norm{\nabla f_{i_t}(x^{t}) }^2.\nonumber
\end{eqnarray*}
In the last equality we made appear $x^{t-1}$ which , for $t=0$, can be taken equal to zero.
Then taking conditional expectation, using the convexity of $f$ (via Lemma \ref{L:convexity via hyperplanes}) and a variance transfer lemma (Lemma \ref{L:variance transfer gradient variance}), we have
\begin{align*}
    \E{\norm{z^{t} - x^*}^2\; | \; x^t} 
    & = 
    \norm{z^{t-1} - x^*}^2  + 2\eta \langle \nabla f(x^{t}) , x^* - x^{t}  \rangle 
    + 2\eta\lambda_t  \langle \nabla f(x^{t}) , x^{t-1} - x^{t} \rangle + \eta^2 \EE{t}{\norm{\nabla f_{i_t}(x^{t})}^2\; | \; x^t} , \\
    & \leq 
    \norm{z^{t-1} - x^*}^2 + (4\eta^2L_{\max}- 2\eta) \br{f(x^{t}) - \inf f} 
     + 2 \eta\lambda_t \br{f(x^{t-1})-f(x^{t})}+ 2 \eta^2 \sigma_f^* \\
     & = 
     \norm{z^{t-1} - x^*}^2 - 2\eta\br{1 + \lambda_t - 2\eta  L_{\max}}\br{f(x^{t}) - \inf f}
     + 2\eta\lambda_t\br{f(x^{t-1}) - \inf f} + 2\eta^2\sigma_f^* \\
     & \leq 
     \norm{z^{t-1} - x^*}^2 - 2\eta\lambda_{t+1}\br{f(x^{t}) - \inf f}
     + 2\eta\lambda_t\br{f(x^{t-1}) - \inf f} + 2\eta^2\sigma_f^*.
\end{align*}
where we used the facts that $\eta \leq \frac{1}{4 L_{\max}}$ and $ \lambda_t + \frac{1}{2} = \lambda_{t+1}$ in the last inequality.
Taking now expectation and summing over $t=0,\dots,T$ , we have after telescoping and cancelling terms
\begin{equation*}
\E{\norm{z^{T} - x^*}^2} 
    \leq 
     \norm{z^{-1} - x^*}^2 - 2\eta\lambda_{T+1}\E{f(x^{T}) - \inf f}
     + 2\eta\lambda_0\br{f(x^{-1}) - \inf f} + 2\eta^2\sigma_f^*(T+1).
\end{equation*}
Now, the fact that $\lambda_0=0$ cancels one term, and also implies that $z^{-1} = x^0 + \lambda_0(x^0-x^{-1})=x^0$.
After dropping the positive term $\E{\norm{z^{T} - x^*}^2}$, we obtain
\begin{eqnarray}
 2\eta \lambda_{T+1} \E{f(x^{T}) - \inf f}
\leq \norm{x_0 - x^*}^2  + 2(T+1)\sigma_f^* \eta^2.\nonumber
\end{eqnarray}
Dividing through by $2\eta\lambda_{T+1}$,  where our assumption on the parameters gives $2\lambda_{T+1} = T+1$, we finally conclude that for all $T \in \mathbb{N}$
\begin{eqnarray*}
\E{f(x^{T}) - \inf f}
\leq \frac{\norm{x_0 - x^*}^2}{\eta (T+1)} + {2\sigma_f^* \eta}. \label{eq:temploz8h4o8hz4}
\end{eqnarray*}
\end{proof}

\begin{corollary}[$\mathcal{O}(1/\varepsilon^2)$ Complexity]
Consider the setting of Theorem~\ref{theo:momentumconv}. 
We can guarantee that $\E{f({x}^T) - \inf f} \leq \varepsilon$ provided that we take
	\begin{equation*}
	\gamma = \frac{\gamma_0}{\sqrt{T}}, \, \gamma_0 = \min \left\{ \frac{1}{4L_{\max}}, \frac{\norm{x^0 - x^*}}{\sqrt{2\sigma_f^*}} \right\}
	\ \text{ and } \ 
	T \geq \left(  \norm{x^0 - x^*} \sqrt{\sigma_f^*} +  \norm{x^0 - x^*}^2 L_{\max}  \right)^2 \frac{16}{\varepsilon^2}.
	\end{equation*}
\end{corollary}

\begin{proof}
    This is a direct consequence of Theorem \ref{theo:momentumconv} and Lemma \ref{L:complexity meta convex} with $A = \Vert x^0 - x^* \Vert^2$, $B = 2 \sigma_f^*$ and $C = 4 L_{\max}$.
    We used the fact that for $T \geq 1$ we can write $\tfrac{1}{T+1} \leq \tfrac{1}{T}$ in Theorem \ref{theo:momentumconv}.
\end{proof}

\subsection{Bibliographic notes}
This section is based on~\cite{Sebbouh2020}.
The deterministic momentum method was designed for
strongly convex functions \cite{Polyak64}.  The authors in \cite{Ghadimi2014} 
 showed that the deterministic momentum method converged globally and sublinearly for smooth  and convex functions.
Theorem~\ref{theo:momentumconv} is from~\cite{Sebbouh2020},  which in turn is an extension of the results in  \cite{Ghadimi2014}.
For convergence proofs for momentum in the non-smooth setting see~\cite{Defaziofact2020}.

\section{Theory : Nonsmooth functions}

In this section we present the tools needed to handle nonsmooth functions.
``Nonsmoothness'' arise typically in two ways.
\begin{enumerate}
    \item Continuous functions  having points of nondifferentiability. For instance:
    \begin{itemize}
        \item the L1 norm $\Vert x \Vert_1 = \sum_{i=1}^d \vert x_i \vert$. 
        It is often used as a regularizer that promotes sparse minimizers.
        \item the ReLU $\sigma(t) = 0$ if $t \leq 0$, $t$ if $t \geq 0$.
        It is often used as the activation function for neural networks, making the associated loss nondifferentiable.
    \end{itemize}
    \item Differentiable functions not being defined on the entire space. 
    An other way to say it is that they take the value $+ \infty$ outside of their domain.
    This can be seen as nonsmoothness, as the behaviour of the function at the boundary of the domain can be degenerate.
    \begin{itemize}
        \item The most typical example is the indicator function of some constraint $C \subset \mathbb{R}^d$, and which is defined as $\delta_C(x) = 0$ if $x \in C$, $+\infty$ if $x \notin C$.
    Such function is useful because it allows to say that minimizing a function $f$ over the constraint $C$ is the same as minimizing the sum $f + \delta_C$.
    \end{itemize}
\end{enumerate}

\subsection{Real-extended valued functions}

\begin{definition}\label{D:domain proper}
    Let $f : \mathbb{R}^d \to \mathbb{R} \cup \{+\infty\}$.
    \begin{enumerate}
        \item The \textbf{domain} of $f$ is defined by $\dom f\eqdef\{x \in \mathbb{R}^d \ | \ f(x) < + \infty \}$.
        \item We say that $f$ is \textbf{proper} if $\dom f \neq \emptyset$.
    \end{enumerate}
\end{definition}

\begin{definition}\label{D:lsc}
    Let $f : \mathbb{R}^d \to \mathbb{R} \cup \{+\infty\}$, and $\bar x \in \mathbb{R}^d$.
    We say that $f$ is \textbf{lower semi-continuous} at $\bar x$ if 
    \begin{equation*}
        f(\bar x) \leq  \underset{x \to \bar x}{\liminf}~f(x).
    \end{equation*}
    We say that $f$ is lower semi-continuous (\textbf{l.s.c.} for short) if $f$ is lower semi-continuous at every $\bar x \in \mathbb{R}^d$.
\end{definition}

\begin{example}[Most functions are proper l.s.c.]~
\begin{itemize}
    \item If $f : \mathbb{R}^d \to \mathbb{R}$ is continuous, then it is proper and l.s.c.
    \item If $C \subset \mathbb{R}^d$ is closed and nonempty, then its indicator function $\delta_C$ is proper and l.s.c.
    \item A finite sum of proper l.s.c functions is proper l.s.c.
\end{itemize}
\end{example}

As hinted by the above example, it is safe to say that most functions used in practice are proper and l.s.c..
It is a minimal technical assumption which is nevertheless needed for what follows (see e.g. Lemmas \ref{L:subdifferential sum rule} and \ref{L:strong convexity minimizers nonsmooth}).

\subsection{Subdifferential of nonsmooth convex functions}

We have seen in Lemma \ref{L:convexity via hyperplanes} that for \textit{differentiable} convex functions,  $\nabla f(x)$ verifies inequality \eqref{eq:conv}. 
For non-differentiable (convex) functions $f$, this fact is used as the basis to define a more general notion : \emph{subgradients}.

\begin{definition}\label{D:subdifferential convex}
Let $f : \mathbb{R}^d \to \mathbb{R} \cup \{+\infty\}$, and $x \in \mathbb{R}^d$.
We say that $\eta \in  \mathbb{R}^d$ is a \textbf{subgradient} of $f$ at $x\in\R^d$ if 
\begin{equation}\label{eq:defsubgrad}
    \mbox{for every }y \in \mathbb{R}^d, \quad
    f(y) - f(x) - \langle \eta, y-x \rangle \geq 0.
\end{equation}
We denote by $\partial f(x)$  the set of all subgradients at $x$,
that is :
\begin{equation*} 
    \partial f(x) \eqdef \{ \eta \in \mathbb{R}^d \ | \ \text{ for all $y \in \mathbb{R}^d$}, \ f(y) - f(x) - \langle \eta, y-x \rangle \geq 0 \} \subset \mathbb{R}^d.
\end{equation*}
We also call $\partial f(x)$ the \textbf{subdifferential} of $f$.
Finally, define $\dom \partial f \eqdef \{ x \in \mathbb{R}^d \ | \ \partial f(x) \neq \emptyset \}$.
\end{definition}

Subgradients are guaranteed to exist whenever $f$ is convex and continuous.
\begin{lemma}\label{L:subgradient exist continuous}
Let $f : \mathbb{R}^d \to \mathbb{R} \cup \{+\infty\}$ be a convex function.
If it is continuous at $x \in \mathbb{R}^d$,
then $\partial f(x) \neq \emptyset$.
This is always true if $\dom f = \mathbb{R}^d$.
\end{lemma}

\begin{proof}
See {\cite[Proposition 3.25]{Pey}} and \cite[Corollary 8.40]{BauCom}.
\end{proof}

If $f$ is differentiable, then $\nabla f(x)$ is the unique subgradient at $x$, as we see next.
This means that the subdifferential is a faithful generalization of the gradient.
\begin{lemma}\label{L:subdifferential differentiable}
If $f : \mathbb{R}^d \to \mathbb{R} \cup \{+\infty\}$ is a convex function that is is differentiable at $x \in \mathbb{R}^d$,
then $\partial f(x) = \{\nabla f(x) \}$.
\end{lemma}

\begin{proof}
(Proof adapted from \cite[Proposition 3.20]{Pey}).  From  Lemma \ref{L:convexity via hyperplanes} we have that $\nabla f(x) \in \partial f(x)$.
Suppose now that $\eta \in \partial f(x)$, and let us show that $\eta = \nabla f(x)$.
For this, take any $v \in \mathbb{R}^d$ and $t>0$, and Definition \ref{D:subdifferential convex} to write 
\begin{equation*}
 f(x+tv) - f(x) - \langle \eta, (x+tv) - x \rangle \geq 0 
 \quad
 \Leftrightarrow 
 \quad
 \frac{f(x+tv) - f(x)}{t} \geq \langle \eta, v \rangle.
\end{equation*}
Taking the limit when $t \downarrow 0$, we obtain that 
\begin{equation*}
    \text{ for all $v \in \mathbb{R}^d$,} \quad 
    \langle \nabla f(x) , v \rangle \geq \langle \eta, v \rangle.
\end{equation*}
By choosing $v = \eta -\nabla f(x) $, we obtain that $\Vert \nabla f(x) - \eta \Vert^2 \leq 0$ which in turn allows us to conclude that $\nabla f(x) = \eta$.
\end{proof}

\begin{remark}
As hinted by the previous results and comments, this definition of subdifferential is tailored for nonsmooth \textit{convex} functions. 
There exists other notions of subdifferential which are better suited for nonsmooth nonconvex functions. 
But we will not discuss it in this monograph, for the sake of simplicity.
The reader interested in this topic can consult \cite{Cla90,RocWet09}.
\end{remark}

In \Cref{L:Lipschitz via jacobian} we saw that if $f$ is differentiable, then it is $G$-Lipschitz continuous if and only if the norm of its gradients is bounded by $G$.
In the next Lemma we see that this is still true when $f$ is not differentiable.

\begin{lemma}\label{L:lipschitz equiv bouded subgradients}
    Let $f : \mathbb{R}^d \to \mathbb{R}$ be convex, and $G \geq 0$.
    Then $f$ is $G$-Lipschitz if and only if $f$ has uniformly $G$-bounded subgradients:
    \begin{equation*}
        \text{for all $x \in \mathbb{R}^d$, for all $\eta \in \partial f(x)$, \ $\Vert \eta \Vert \leq G$}.
    \end{equation*}
\end{lemma}

\begin{proof}
    This proof is essentially taken from \cite[Proposition 16.20]{BauCom}.
    Note that in this proof subgradients always exist because we assume $f$ to be finite (see Lemma \ref{L:subgradient exist continuous}).
    For the first implication, assume that $f$ is $G$-Lipschitz, take $x \in \mathbb{R}^d$, $\eta \in \partial f(x)$, and show that $\Vert \eta \Vert \leq G$.
    Introduce $y = x + \eta$, and use the definition of subgradient to write
    \begin{equation*}
        \Vert \eta \Vert^2 
        = 
        \langle \eta, y-x \rangle
        \leq 
        f(y) - f(x)
        \leq 
        G \Vert y - x \Vert 
        = G \Vert \eta \Vert,
    \end{equation*}
    so the conclusion follows.
    For the second implication, assume that the subgradients of $f$ are uniformly $G$-bounded.
    Take $x,y \in \mathbb{R}^d$, and without loss of generality suppose that $f(y) \geq f(x)$.
    Then we can take $\eta \in \partial f(y)$ to write
    \begin{equation*}
        \vert f(y) - f(x) \vert=
        f(y) - f(x)
        \leq \langle \eta, y-x \rangle 
        \leq 
        \Vert \eta \Vert \Vert y-x \Vert \leq G \Vert y-x \Vert,
    \end{equation*}
    which proves the claim.
\end{proof}

\begin{proposition}[Fermat's Theorem]\label{P:Fermat convex nonsmooth}
Let $f : \mathbb{R}^d \to \mathbb{R} \cup \{+\infty\}$, and $\bar x \in \mathbb{R}^d$.
Then $\bar x$ is a minimizer of $f$ if and only if $0\in \partial f(\bar x)$.
\end{proposition}

\begin{proof}
From the Definition \ref{D:subdifferential convex}, we see that
\begin{eqnarray*}
    & & 
    \bar x \text{ is a minimizer of } f \\
    & \Leftrightarrow &
    \text{for all } y \in \mathbb{R}^d, f(y) - f(\bar x) \geq 0 \\
    & \Leftrightarrow &
    \text{for all } y \in \mathbb{R}^d, f(y) - f(\bar x) - \langle 0, y-x \rangle \geq 0 \\
    &\overset{\eqref{eq:defsubgrad}}{ \Leftrightarrow} &
    0 \in \partial f(\bar x). 
\end{eqnarray*}

\end{proof}

\begin{lemma}[Sum rule]\label{L:subdifferential sum rule}
Let $f : \mathbb{R}^d \to \mathbb{R}$ be convex and differentiable.
Let $g : \mathbb{R}^d \to \mathbb{R} \cup \{+\infty\}$ be proper l.s.c. convex.
Then, for all $x \in \mathbb{R}^d$, $\partial (f+g)(x) = \{\nabla f(x) \} + \partial g(x)$.
\end{lemma}

\begin{proof}
See \cite[Theorem 3.30]{Pey}.
\end{proof}

\begin{lemma}[Positive homogeneity]
Let $f : \mathbb{R}^d \to \mathbb{R}$ be proper l.s.c. convex.
Let $x \in \mathbb{R}^d$, and $\gamma \geq 0$.
Then $\partial (\gamma f)(x) = \gamma \partial f(x)$.
\end{lemma}

\begin{proof}
It is an immediate consequence of Definition \ref{D:subdifferential convex}.
\end{proof}

\subsection{Nonsmooth strongly convex functions}

In this context Lemma \ref{L:strong convexity minimizers} remains true: Strongly convex functions do not need to be continuous to have a unique minimizer:

\begin{lemma}\label{L:strong convexity minimizers nonsmooth}
If $f : \mathbb{R}^d \to \mathbb{R} \cup \{+\infty\}$ is a proper l.s.c. $\mu$-strongly convex function, 
then $f$ admits a unique minimizer.
\end{lemma}

\begin{proof}
See \cite[Corollary 2.20]{Pey}.
\end{proof}

We also have an obvious analogue to Lemma \ref{L:strong convexity differentiable hyperplans}:

\begin{lemma}\label{L:strong convexity hyperplans nonsmooth}
If $f : \mathbb{R}^d \to \mathbb{R} \cup \{+\infty\}$ is a proper l.s.c and $\mu$-strongly convex function, 
then for every $x,y \in \mathbb{R}^d$, and for every $\eta \in \partial f(x)$ we have that
\begin{equation} \label{eq:strconv nonsmooth}
 f(y) - f(x) - \dotprod{\eta, y-x} \geq \frac{\mu}{2} \norm{y-x}^2.
\end{equation}
\end{lemma}

\begin{proof}
Define $g(x) := f(x) - \frac{\mu}{2}\Vert x \Vert^2$.
According to Lemma \ref{L:strong convexity is convex plus norm}, $g$ is convex.
It is also clearly l.s.c. and proper, by definition.
According to the sum rule in Lemma \ref{L:subdifferential sum rule}, we have $\partial f(x) = \partial g(x) + \mu x$.
Therefore we can use the convexity of $g$ with Definition \ref{D:subdifferential convex} to write
\begin{equation*}
    f(y) - f(x) - \langle \eta, y-x \rangle 
    \geq 
    \frac{\mu}{2}\Vert y \Vert^2 - \frac{\mu}{2}\Vert x \Vert^2 - \langle \mu x, y-x \rangle 
    =
    \frac{\mu}{2}\Vert y - x \Vert^2.
\end{equation*}
\end{proof}

\subsection{Proximal operator}\label{S:nonsmooth theory:proximal operator}

In this section we study a key tool used in some algorithms for minimizing nonsmooth functions.

\begin{definition}\label{D:proximal operator} 
Let $g:\R^d \rightarrow \mathbb{R} \cup \{+\infty\}$ be a proper l.s.c convex function.
We define the \textbf{proximal operator} of $g$ as the function $\prox_g : \mathbb{R}^d \to \mathbb{R}^d$ defined by
\begin{equation*}
\prox_{g}(x) : = \underset{x' \in \mathbb{R}^d}{\rm{argmin}}~ g(x') +  \frac{1}{2}\Vert x' - x \Vert^2 
\end{equation*}
\end{definition}
The proximal operator is well defined because, since $g(x')$ is convex the sum $ g(x') +  \frac{1}{2} \Vert x' - x \Vert^2$ is strongly convex in $x'$. 
Thus there exists only one minimizer (recall Lemma \ref{L:strong convexity minimizers nonsmooth}).

\begin{example}[Projection is a proximal operator]\label{Ex:projection is prox}
Let $C \subset \mathbb{R}^d$ be a nonempty closed convex set, and let $\delta_C$ be its indicator function.
Then the proximal operator of $\delta_C$ is exactly the projection operator onto $C$:
\begin{equation*}
    \prox_{\delta_C}(x) = {\rm proj}_C(x) \eqdef \underset{c \in C}{\rm{argmin}}~\Vert c - x \Vert^2.
\end{equation*}
\end{example}

The proximal operator can be characterized with the subdifferential :
\begin{lemma}\label{L:prox characterization subdifferential}
Let $g:\R^d \rightarrow \mathbb{R} \cup \{+\infty\}$ be a proper l.s.c convex function, let $\gamma >0$ and let $x, p \in \mathbb{R}^d$.
Then $p = \prox_{\gamma g}(x)$ if and only if 
\begin{equation*}
\frac{x-p}{\gamma} \in \partial g(p). 
\end{equation*}
\end{lemma}

\begin{proof}
From Definition \ref{D:proximal operator} we know that $p = \prox_{\gamma g}(x)$ if and only if $p$ is the minimizer of $\phi(u):=g(u) + \frac{1}{2 \gamma}\Vert u-x \Vert^2$.
From our hypotheses on $g$, it is clear that $\phi$ is proper l.s.c convex.
So we can use Proposition \ref{P:Fermat convex nonsmooth} to say that it is equivalent to $0 \in \partial \phi(p)$.
Moreover, we can use the sum rule from Lemma  \ref{L:subdifferential sum rule} to write that $\partial \phi(p) = \partial g(p) + \{\frac{p-x}{\gamma}\}$.
So we have proved that $p = \prox_{\gamma g}(x)$ if and only if $0 \in \partial g(p) + \{\frac{p-x}{\gamma}\}$, which is what we wanted to prove, after rearranging the terms.
\end{proof}

We show that, like the projection, the proximal operator is $1$-Lipschitz (we also say that it is \emph{non-expansive}).
This property will be very interesting for some proofs since it will allow us to ``get rid'' of the proximal terms.

\begin{lemma}[Non-expansiveness]\label{L:prox nonexpansive}
Let $g:\R^d \rightarrow \mathbb{R} \cup \{+\infty\}$ be a proper l.s.c convex function.
Then $\prox_g : \mathbb{R}^d \to \mathbb{R}^d$ is $1$-Lipschitz : 
\begin{equation}\label{eq:nonexpanweak}
\text{for all $x,y \in \mathbb{R}^d$,} \quad
\norm{\prox_{g}(y) - \prox_{g}(x)} \leq\norm{y-x}.
\end{equation}
\end{lemma}

\begin{proof}
Let $p_y \eqdef \prox_{g}(y)$  and $p_x \eqdef \prox_{g}(x)$.
From $p_x = \prox_{g}(x)$ we have $x - p_x \in\partial g(p_x)$ (see Lemma \ref{L:prox characterization subdifferential}), so from the definition of the subdifferential (Definition \ref{D:subdifferential convex}), we obtain
\begin{equation*}
    g(p_y) - g(p_x) - \dotprod{x-p_x, p_y- p_x} \geq 0.
\end{equation*}
Similarly, from $p_y = \prox_{g}(y)$ we also obtain 
\begin{equation*}
    g(p_x) - g(p_y) - \dotprod{y-p_y, p_x- p_y} \geq 0.
\end{equation*}
Adding together the above two inequalities gives
\[\dotprod{y-x-p_y+p_x, p_x- p_y}\leq 0.  \]
Expanding the left argument of the inner product, and using the Cauchy-Schwartz inequality gives
\[
\norm{p_x -p_y}^2 \leq \dotprod{x-y, p_x- p_y} \leq \norm{x-y}\norm{p_x-p_y}.
\]
Dividing through by $\norm{p_x -p_y}$ (assuming this is non-zero otherwise~\eqref{eq:nonexpanweak} holds trivially) we have~\eqref{eq:nonexpanweak}.
\end{proof}

We end this section with an important property of the proximal operator  : it can help to characterize the minimizers of composite functions as fixed points.

\begin{lemma}\label{L:prox fixed point composite}
Let $f : \mathbb{R}^d \to \mathbb{R}$ be convex differentiable, let $g : \mathbb{R}^d \to \mathbb{R} \cup \{+\infty\}$ be proper l.s.c. convex.
If $x^* \in {\rm{argmin}}(f+g)$, then 
\begin{equation*}
    \text{ for all } \gamma >0, \quad 
    \prox_{\gamma g}(x^* - \gamma \nabla f(x^*)) = x^*.
\end{equation*}
\end{lemma}

\begin{proof}
Since $x^* \in {\rm{argmin}}(f+g)$ we have  that $0 \in \partial (f+g)(x^*) = \nabla f(x^*) + \partial g(x^*)$ (Proposition \ref{P:Fermat convex nonsmooth} and Lemma \ref{L:subdifferential sum rule}).
By multiplying both sides by $\gamma$ then by adding $x^*$ to both sides gives
\begin{equation*}
    (x^* - \gamma \nabla f(x^*)) - x^* \in \partial (\gamma g)(x^*).
\end{equation*}
According to Lemma \ref{L:prox characterization subdifferential}, this means that $\prox_{\gamma g}(x^* - \gamma \nabla f(x^*)) = x^*$.
\end{proof}

\subsection{Controlling the variance}

\begin{definition}\label{D:divergence bregman}
Let $f : \mathbb{R}^d \to \mathbb{R}$ be a differentiable function. We define the (Bregman) \textbf{divergence} of $f$ between $y$ and $x$ as
\begin{eqnarray*}
D_f(y;x) \eqdef f(y) -f(x) - \dotprod{\nabla f(x), y-x}.
\end{eqnarray*}
\end{definition}

Note that the divergence $D_f(y;x)$ is always nonnegative when $f$ is convex due to Lemma~\ref{L:convexity via hyperplanes}. 
Moreover,  the divergence is also upper bounded by suboptimality.

\begin{lemma}\label{L:bregman divergence composite}
Let $f : \mathbb{R}^d \to \mathbb{R}$ be convex differentiable, and $g : \mathbb{R}^d \to \mathbb{R} \cup \{+\infty\}$ be proper l.s.c. convex, and $F = g + f$.
Then, for all $x^* \in {\rm{argmin}}~F$, for all $x \in \mathbb{R}^d$,
\begin{equation*}
    0 \leq  D_f(x;x^*) \; \leq \;  F(x) - \inf F.
\end{equation*}
\end{lemma}

\begin{proof}
Since $x^* \in {\rm{argmin}}~F$, we can use the Fermat Theorem (Proposition \ref{P:Fermat convex nonsmooth}) and the sum rule (Lemma \ref{L:subdifferential sum rule}) to obtain the existence of some $\eta^* \in \partial g(x^*)$ such that $\nabla f(x^*) + \eta^* = 0$.
Use now the definition of the Bregman divergence, and the convexity of $g$ (via Lemma \ref{L:convexity via hyperplanes}) to write
\begin{eqnarray*}
    D_f(x;x^*)
    &=&
    f(x) - f(x^*) - \langle \nabla f(x^*), x-x^* \rangle
    =
    f(x) - f(x^*) + \langle  \eta^*, x-x^* \rangle \\
    &\leq &
    f(x) - f(x^*) + g(x) - g(x^*) \\
    & = & F(x) - F(x^*).
\end{eqnarray*}
\end{proof}

Next we provide a variance transfer lemma,  generalizing Lemma  \ref{L:variance transfer gradient variance}, which will prove to be useful when dealing with nonsmooth sum of functions in Section \ref{sec:sgdprox}.

\begin{lemma}[Variance transfer - General convex case]\label{L:PSGD variance transfer convex}
Let $f$ verify Assumptions \tref{Ass:SGD sum of smooth} and \tref{Ass:SGD sum of convex}. 
For every $x,y \in \mathbb{R}^d$, we have
\begin{equation*}
    \mathbb{V}\left[ \nabla f_i(x) \right] \leq 4 L_{\max}D_f(x;y) + 2 \mathbb{V}\left[ \nabla f_i(y) \right],
\end{equation*}
where 
\[\mathbb{V}[X] := \mathbb{E}[\,\Vert X - \mathbb{E}[X]\, \Vert^2]  .\]
\end{lemma}

\begin{proof}
Simply use successively Lemma \ref{L:variance and expectation}, the inequality $\Vert a+b\Vert^2 \leq 2\Vert a \Vert^2 + 2 \Vert b \Vert^2$, and the expected smoothness (via Lemma \ref{L:expected smoothness}) to write: 
\begin{eqnarray*}
    \mathbb{V}\left[ \nabla f_i(x) \right]
    & \leq &
    \mathbb{E}\left[ \Vert \nabla f_i(x) - \nabla f(y) \Vert^2 \right] \\
    & \leq &
    2 \mathbb{E}\left[ \Vert \nabla f_i(x) - \nabla f_i(y) \Vert^2 \right] + 2 \mathbb{E}\left[ \Vert \nabla f_i(y) - \nabla f(y) \Vert^2 \right] \\
    & = & 
    2 \mathbb{E}\left[ \Vert \nabla f_i(x) - \nabla f_i(y) \Vert^2 \right] + 2 \mathbb{V}\left[ \nabla f_i(y) \right] \\
    & \leq &
    4 L_{\max} D_f(x;y) + 2 \mathbb{V}\left[ \nabla f_i(y) \right].
\end{eqnarray*}
\end{proof}

\begin{definition}[Composite Gradient Noise]\label{D:gradient solution variance composite}
Let $f : \mathbb{R}^d \to \mathbb{R}$ verify Assumption \tref{Ass:SGD sum of smooth}.
Let $g : \mathbb{R}^d \to \mathbb{R} \cup \{+\infty\}$ be proper l.s.c convex. Let $F = g+f$ be such that ${\rm{argmin}}~F\neq \emptyset$.
We define the \textbf{composite gradient noise}  as follows
\begin{equation}\label{eq:gradient solution variance composite}
    \sigma^{*}_F \eqdef \inf\limits_{x^* \in {\rm{argmin}}~F} \ \mathbb{V}\left[ \nabla f_i(x^*) \right].
\end{equation}
\end{definition}

Note the difference between $\sigma_f^*$ introduced in Definition \ref{D:gradient solution variance} and $\sigma^{*}_F$ introduced here
is that the variance of gradients taken at the minimizers of the composite sum $F$,  as opposed to $f$.

\begin{lemma}\label{L:interpolation via gradient variance composite}
Let $f : \mathbb{R}^d \to \mathbb{R}$ verify Assumptions \tref{Ass:SGD sum of smooth} and \tref{Ass:SGD sum of convex}.
Let $g : \mathbb{R}^d \to \mathbb{R} \cup \{+\infty\}$ be proper l.s.c convex. Let $F = g+f$ be such that ${\rm{argmin}}~F\neq \emptyset$.
\begin{enumerate}
    \item $\sigma^{*}_F \geq 0$.
    \item $\sigma^{*}_F = \mathbb{V}\left[ \nabla f_i(x^*) \right]$ for every $x^* \in {\rm{argmin}}~F$.
    \item If $\sigma^{*}_F = 0$ then there exists $x^* \in {\rm{argmin}}~F$ such that $x^* \in {\rm{argmin}}~(g+f_i)$ for all $i=1,\dots, n$.
    The converse implication is also true if $g$ is differentiable at $x^*$.
    \item $\sigma_F^* \leq 4 L_{\max}\left( f(x^*) - \inf f) \right) + 2 \sigma_f^*$, for every $x^* \in {\rm{argmin}}~F$.
\end{enumerate}
\end{lemma}

\begin{proof}
Item 1 is trivial.
For item 2, consider two minimizers $x^*,x' \in {\rm{argmin}}~F$, and use the expected smoothness of $f$ (via Lemma \ref{L:expected smoothness}) together with Lemma \ref{L:bregman divergence composite} to write
\begin{equation*}
\frac{1}{2L_{\max}}\mathbb{E}\left[ \Vert \nabla f_i(x^*) - \nabla f_i(x') \Vert^2 \right]
\leq D_f(x^*;x') \leq F(x^*) - \inf F = 0.
\end{equation*}
In other words, we have $\nabla f_i(x^*) = \nabla f_i(x')$ for all $i=1,\dots, n$, which means that indeed $\mathbb{V}\left[ \nabla f_i(x^*) \right]=\mathbb{V}\left[ \nabla f_i(x') \right]$.
Now we turn to item 3, and start by assuming that $\sigma_F^*=0$.
Let $x^* \in {\rm{argmin}}~F$, and we know from the previous item that $\mathbb{V}\left[ \nabla f_i(x^*) \right]=0$.
This is equivalent to say that, for every $i$, $\nabla f_i(x^*) = \nabla f(x^*)$.
But $x^*$ being a minimizer implies that $-\nabla f(x^*) \in \partial g(x^*)$ (use Proposition \ref{P:Fermat convex nonsmooth} and Lemma \ref{L:subdifferential sum rule}).
So we have that $-\nabla f_i(x^*) \in \partial g(x^*)$, from which we conclude by the same arguments that $x^* \in {\rm{argmin}}~(g+f_i)$.
Now let us prove the converse implication, by assuming further that $g$ is differentiable at $x^*$.
From the assumption $x^* \in {\rm{argmin}}~(g+f_i)$, we deduce that $-\nabla f_i(x^*) \in \partial g(x^*) = \nabla g(x^*)$ (see Lemma \ref{L:subdifferential differentiable}).
Taking the expectation on this inequality also gives us that $-\nabla f(x^*) = \nabla g(x^*)$.
In other words,  $\nabla f_i(x^*) = \nabla f(x^*)$ for every $i$.
We can then conclude that $\mathbb{V}\left[ \nabla f_i(x^*) \right]=0$.
We finally turn to item 4, which is a direct consequence of Lemma \ref{L:PSGD variance transfer convex} (with $x=x^* \in {\rm{argmin}}~F$ and $y=x^*_f \in {\rm{argmin}}~f$) :
\begin{eqnarray*}
\sigma_F^* & = & \mathbb{V}\left[ \nabla f_i(x^*) \right] \\
&\leq & 4L_{\max} D_f(x^* ; x^*_f) + 2 \mathbb{V}\left[ \nabla f_i(x^*_f) \right] \\
& = & 4 L_{\max}\left( f(x^*) - \inf f) \right) + 2 \sigma_f^*.
\end{eqnarray*}
\end{proof}

\section{Stochastic Subgradient Descent}
\label{sec:stochsubgrad}
\begin{problem}[Stochastic Function]\label{Pb:stochastic functions}
We want to minimize a function $f : \mathbb{R}^d \to \mathbb{R}$ which writes as 
\begin{equation*}
f(x) \eqdef \mathbb{E}_\mathcal{D}\left[ f_{\xi}(x) \right],
\end{equation*}
where $\mathcal{D}$ is some distribution over $\mathbb{R}^q$, $\xi \in \mathbb{R}^q$ is sampled from $\mathcal{D}$, and  $f_{\xi} : \mathbb{R}^d \to \mathbb{R}$.
We require that the problem is well-posed, in the sense that ${\rm{argmin}}~f \neq \emptyset$.
\end{problem}

In this section we will assume that the  functions $f_\xi$ are convex and have bounded subgradients. 

\begin{assumption}[Expectation of Convex]\label{Ass:expectation of convex}
    Considering the problem \tref{Pb:stochastic functions}, we assume for every $\xi \in \mathbb{R}^q$ that $f_\xi$ is convex.
    Moreover, we assume that we have access to a measurable subgradient oracle: for every $x \in \mathbb{R}^d$ and $\xi \in \mathbb{R}^q$ there exists $\sgrad_\xi(x) \in \partial f_\xi(x)$ such that $\xi \mapsto \sgrad_\xi(x)$ is measurable.
    
\end{assumption}

\begin{assumption}[Expectation of $G$-Lipschitz]\label{Ass:expectation of lipschitz}
    Considering the problem \tref{Pb:stochastic functions}, we assume for every $\xi \in \mathbb{R}^q$ that $f_\xi$ is $G$-Lipschitz continuous.
\end{assumption}

Note that assuming the $f_\xi$ to be Lipschitz is equivalent to assume that their subgradients are bounded (recall Lemma \ref{L:lipschitz equiv bouded subgradients}).
Observe also that this assumption implies that the expected function $f$ is $G$-Lipschitz continuous.

We now define the Stochastic Subgradient Descent algorithm, which is an extension of \tref{Algo:Stochastic GD}.
Instead of considering the gradient of a function $f_i$, we consider here some subgradient of $f_\xi$. 
\begin{algorithm}[SSD]\label{Algo:Stochastic Subgradient Descent}
Consider Problem \tref{Pb:stochastic functions} and let Assumption \tref{Ass:expectation of convex} hold.
Let $x^0 \in \mathbb{R}^d$, and let $\gamma_t>0$ be a sequence of stepsizes.
The \textbf{Stochastic Subgradient Descent (SSD)} algorithm is given by the iterates $(x^t)_{t \in \mathbb{N}}$ where
\begin{align*}
    \xi_t & \in \mathbb{R}^q & \mbox{Sampled i.i.d. $\xi_t \sim \mathcal{D}$}\\
    x^{t+1} & = x^t - \gamma_t \sgrad_{\xi_t}(x^t), & \text{ with } \sgrad_{\xi_t}(x^t) \in \partial f_{\xi_t}(x^t).
\end{align*}
\end{algorithm}
In \tref{Algo:Stochastic Subgradient Descent}, the sampled subgradient $\sgrad_{\xi_t}(x^t)$ is an unbiaised estimator of a subgradient of $f$ at $x^t$.

\begin{lemma}\label{L:expected subgradient formula}
    If Assumption \tref{Ass:expectation of convex} holds, then $f$ is convex, and for all $x \in \mathbb{R}^d$ we have that $\mathbb{E}_\mathcal{D}\left[ \sgrad_{\xi}(x) \right] \in \mathbb{R}^d$ exists and is a subgradient of $f$ at $x$.
\end{lemma}

\begin{proof}
The fact that $f$ is convex is a trivial consequence of Assumption \tref{Ass:expectation of convex}, indeed we can use  the convexity of the $f_\xi$ together with the fact that $f$ takes finite values to write
\begin{equation*}
    f(tx + (1-t) y)
    =
    \mathbb{E}_\mathcal{D}\left[ f_\xi(tx + (1-t) y) \right]
    \leq
    t \mathbb{E}_\mathcal{D}\left[ f_\xi(x) \right]
    + (1-t) \mathbb{E}_\mathcal{D}\left[ f_\xi(y) \right]
    =
    t f(x) + (1-t)f(y).
\end{equation*}
Now, the fact that $f$ takes finite values and that $\xi \mapsto \sgrad_{\xi}(x)$ is measurable imply that $\mathbb{E}_\mathcal{D}\left[ \sgrad_{\xi}(x) \right]$ is well-defined (see \cite[p. 223]{Ber73a}).
Now we can use the fact that $\sgrad_{\xi}(x) \in \partial f_{\xi}(x)$ to write
\begin{equation*}
    f_\xi(y) - f_\xi(x) - \langle \sgrad_{\xi}(x), y-x \rangle \geq 0,
\end{equation*}
which after taking expectation leads to
\begin{equation*}
    f(y) - f(x) - \langle \mathbb{E}_\mathcal{D}\left[ \sgrad_{\xi}(x) \right], y-x \rangle \geq 0,
\end{equation*}
which proves that $\mathbb{E}_\mathcal{D}\left[ \sgrad_{\xi}(x) \right] \in \partial f(x)$.
\end{proof}

\subsection{Convergence for convex Lipschitz functions}

In the next Theorem \ref{T:SSD CV convex bounded general stepsize} we get a bound for \eqref{Algo:Stochastic Subgradient Descent} for general stepsizes.
In the Theorem \ref{T:SSD CV convex bounded constant stepsize} we specialize our estimate for constant stepsizes which leads to a finite-horizon rate of $\mathcal{O}\left( \frac{1}{\sqrt{T}} \right)$.
This will traduce in a $\mathcal{O}\left(\frac{1}{\varepsilon^2} \right)$ complexity in \Cref{T:SSD complexity convex bounded}.
By considering a suitably decreasing sequence of stepsizes, we finally obtain a convergence rate $\mathcal{O}\left( \frac{\log(T+1)}{\sqrt{T}} \right)$ in \Cref{T:SSD CV convex bounded special stepsizes}.

\begin{theorem}\label{theo:stochgradB}\label{T:SSD CV convex bounded general stepsize}

Let Assumptions~\tref{Ass:expectation of convex} and \tref{Ass:expectation of lipschitz} hold. 
Consider $(x^t)_{t \in \mathbb{N}}$ a sequence generated by the \tref{Algo:Stochastic Subgradient Descent} algorithm, with a sequence of stepsizes $\gamma_t >0$.
Then for every $T \geq 1$ and  $\bar{x}^T \eqdef \frac{1}{\sum_{k=0}^{T-1}\gamma_t}\sum_{t=0}^{T-1} \gamma_t x_t$ we have
\begin{eqnarray*}
\EE{}{f(\bar{x}^{T})- \inf f}
    \leq 
    \frac{\norm{x^{0}-x^*}^2 }{2\sum_{t=0}^{T-1}\gamma_t} +\frac{\sum_{t=0}^{T-1}\gamma_t^2 G^2}{2\sum_{t=0}^{T-1}\gamma_t}.
\end{eqnarray*}
\end{theorem}

 \begin{proof}
 Expanding the squares we have that
\begin{eqnarray*}
\norm{x^{t+1}-x^*}^2 & \overset{
{\tref{Algo:Stochastic Subgradient Descent} }}{=}& \norm{x^{t}-x^*- \gamma_t \sgrad_{\xi_t}(x_t)}^2 \\
& =& \norm{x^{t}-x^*}^2 - 2\gamma_t \dotprod{\sgrad_{\xi_t}(x_t),x^t-x^*} + \gamma_t^2 \norm{\sgrad_{\xi_t}(x_t)}^2. \nonumber
\end{eqnarray*}
We will use the fact that our subgradients are bounded from Assumption \tref{Ass:expectation of lipschitz}, and that  $\mathbb{E} \left[ \sgrad_{\xi_t}(x_t) \ | \ x^t \right] \in \partial f(x^t)$  (see Lemma \ref{L:expected subgradient formula}).
Taking expectation conditioned on $x^t$ we have that
\begin{eqnarray*}
\EE{}{\norm{x^{t+1}-x^*}^2 \, | \, x^t} & = & \norm{x^{t}-x^*}^2 - 2\gamma_t \dotprod{\mathbb{E}\left[ \sgrad_{\xi_t}(x_t) \ | \ x^t \right],x^t-x^*} + \gamma_t^2 \EE{}{\norm{\sgrad_{\xi_t}(x_t)}^2 \ | \ x^t} \nonumber\\
& \leq &\norm{x^{t}-x^*}^2 - 2\gamma_t \dotprod{\mathbb{E}\left[ \sgrad_{\xi_t}(x_t) \ | \ x^t \right],x^t-x^*} + \gamma_t^2 G^2 \nonumber \\
& \overset{\eqref{eq:defsubgrad}}{ \leq} &\norm{x^{t}-x^*}^2  -2\gamma_t ( f(x^t) - \inf f) + \gamma_t^2 G^2.
\end{eqnarray*}
Re-arranging, taking expectation and summing up from $t= 0 , \ldots, T-1$ gives
\begin{align*}
    2\sum_{t=0}^{T-1} \gamma_t \EE{}{f(x^t) -\inf f} & \leq \sum_{t=0}^{T-1}\left( \EE{}{\norm{x^{t}-x^*}^2 } - \EE{}{\norm{x^{t+1}-x^*}^2 }\right)+\sum_{t=0}^{T-1}\gamma_t^2 G^2 \nonumber \\
    & = \EE{}{\norm{x^{0}-x^*}^2 } -\EE{}{\norm{x^{T}-x^*}^2 }+\sum_{t=0}^{T-1}\gamma_t^2 G^2 \nonumber \\
    & \leq \norm{x^{0}-x^*}^2  +\sum_{t=0}^{T-1}\gamma_t^2 G^2.
\end{align*}
Let $\bar{x}^T = \frac{1}{\sum_{t=0}^{T-1}\gamma_t}\sum_{t=0}^{T-1} \gamma_t x_t$. Dividing through by $2\sum_{t=0}^{T-1}\gamma_t$ and using Jensen's inequality we have
\begin{eqnarray*}
\EE{}{f(\bar{x}^T)-\inf f} 
&{\leq} &
\frac{1}{\sum_{t=0}^{T-1}\gamma_t}\sum_{t=0}^{T-1} \gamma_t \EE{}{f(x^t) -\inf f} \nonumber \\
 & \leq & \frac{\norm{x^{0}-x^*}^2 }{2\sum_{t=0}^{T-1}\gamma_t} +\frac{\sum_{t=0}^{T-1}\gamma_t^2 G^2}{2\sum_{t=0}^{T-1}\gamma_t}.
\end{eqnarray*}
\end{proof}

\begin{theorem}\label{T:SSD CV convex bounded constant stepsize}
Let Assumptions~\tref{Ass:expectation of convex} and \tref{Ass:expectation of lipschitz} hold. 
Consider $(x^t)_{t \in \mathbb{N}}$ a sequence generated by the \tref{Algo:Stochastic Subgradient Descent} algorithm, with a constant  stepsize $\gamma_t \equiv \gamma >0$.
Then for every $T \geq 1$ and  $\bar{x}^T \eqdef \frac{1}{T}\sum_{t=0}^{T-1}  x_t$ we have
\begin{eqnarray*}
\EE{}{f(\bar{x}^{T})- \inf f}
    \leq 
    \frac{\norm{x^{0}-x^*}^2 }{2 \gamma T} +\frac{\gamma  G^2}{2}.
\end{eqnarray*}
In particular, for a fixed horizon $T \geq 1$ and $\gamma = \frac{1}{\sqrt{T}}$, we see that
\begin{eqnarray*}
\EE{}{f(\bar{x}^{T})- \inf f}
    \leq 
    \frac{\norm{x^{0}-x^*}^2 + G^2 }{2 \sqrt{T}} = \mathcal{O}\left( \frac{1}{\sqrt{T}} \right).
\end{eqnarray*}
\end{theorem}

\begin{proof}
    Apply Theorem \ref{T:SSD CV convex bounded general stepsize} with $\sum_{t=0}^{T-1} \gamma_t = \gamma T$ and   $\sum_{t=0}^{T-1} \gamma_t^2 = \gamma^2 T$.
\end{proof}

\begin{corollary}\label{T:SSD complexity convex bounded}
Consider the setting of Theorem \ref{T:SSD CV convex bounded constant stepsize}.
For every $\varepsilon >0$ we can guarantee that $\EE{}{f(\bar{x}^{T})- \inf f} \leq \varepsilon$ provided that
\begin{equation*}
	\gamma = \frac{\norm{x^{0}-x^*}}{G\sqrt{T}}
	\quad \text{ and } \quad 
    T \geq \frac{\norm{x^{0}-x^*}^2 G^2}{\varepsilon^2}.
\end{equation*}
\end{corollary}

\begin{proof}
This is a direct consequence of Theorem \ref{T:SSD CV convex bounded constant stepsize} and Lemma \ref{L:complexity meta convex}, with $A = \tfrac12 \Vert x^0 - x^* \Vert^2$, $B = \tfrac12 G^2$ and $C=0$.
\end{proof}

\begin{theorem}\label{T:SSD CV convex bounded special stepsizes}
    Let Assumptions~\tref{Ass:expectation of convex} and \tref{Ass:expectation of lipschitz} hold. 
Consider $(x^t)_{t \in \mathbb{N}}$ a sequence generated by the \tref{Algo:Stochastic Subgradient Descent} algorithm, with a sequence of stepsizes $\gamma_t \eqdef \frac{\gamma_0}{\sqrt{t+1}}$ for some $\gamma_0 >0$. 
    We have for every $T \geq 1$ and for $\bar{x}^T \eqdef \frac{1}{\sum_{t=0}^{T-1}\gamma_t}\sum_{t=0}^{T-1} \gamma_t x_t$ that
 \begin{equation*}
\EE{}{f(\bar{x}^{T}) -\inf f }
\leq 
\frac{\norm{x^{0}-x^*}^2}{\gamma_0 \sqrt{T}}
+
\frac{2\gamma_0 G^2 \log(T+1)}{\sqrt{T}}
=
\mathcal{O}\left( \frac{\log(T+1)}{\sqrt{T}} \right).
\end{equation*}
\end{theorem}

\begin{proof}
Start considering $\gamma_t = \frac{\gamma}{\sqrt{t+1}}$, and use integral bounds (see Lemma \ref{L:sum integral bounds}) to write 
\begin{equation*}
    \sum_{t=0}^{T-1} \gamma_t
    = \gamma_0 \sum_{t=1}^T \frac{1}{\sqrt{t}}
    \geq
    \frac{\gamma_0}{2} \sqrt{T}
    \quad \text{ and } \quad 
    \sum_{t=0}^{T-1} \gamma_t^2
    =
    \gamma_0^2 \sum_{t=1}^T \frac{1}{t}
    \leq 
    2\gamma_0^2 \log(T+1).
\end{equation*}
Injecting those bounds in the bound of \Cref{T:SSD CV convex bounded general stepsize}, we obtain
\begin{equation*}
    \EE{}{f(\bar{x}^{T})- \inf f}
    \leq 
    \frac{\norm{x^{0}-x^*}^2 + \sum_{t=0}^{T-1}\gamma_t^2 G^2 }{2\sum_{t=0}^{T-1}\gamma_t}
    \leq
    \frac{\norm{x^{0}-x^*}^2 + 2\gamma_0^2 G^2 \log(T+1)}{\gamma_0 \sqrt{T}}.
\end{equation*}
\end{proof}

\subsection{Better convergence rates for convex functions with bounded solution}\label{S:SSD:convex bounded solution}

In the previous section, we saw that \tref{Algo:Stochastic Subgradient Descent} has a $\mathcal{O}\left( \tfrac{\log(T+1)}{\sqrt{T}} \right)$ convergence rate, but enjoys a $\mathcal{O}\left( \frac{1}{\varepsilon^2} \right)$ complexity rate. 
The latter suggests that it is possible to get rid of the logarithmic term and achieve a $\mathcal{O}\left( \frac{1}{\sqrt{T}} \right)$ convergence rate.
In this section, we see that this can be done, by making a localization assumption on the solution of the problem, and by making a slight modification to the  \tref{Algo:Stochastic Subgradient Descent} algorithm.

\begin{assumption}[$B$--Bounded Solution]\label{ass:boundsol}\label{Ass:SSD bounded solution}
There exists $B>0$ and a solution $x^* \in {\rm{argmin}}~ f$ such that $\norm{x^*} \leq B.$ 
\end{assumption}

We will exploit this assumption by modifying the \tref{Algo:Stochastic Subgradient Descent} algorithm, adding a projection step onto the closed ball $\mathbb{B}(0,B)$ where we know that the solution belongs.  In this case the projection onto the ball is given by
\begin{eqnarray*}
    {\rm{proj}}_{\mathbb{B}(0,B)}(x) &:= & 
    \begin{cases}
      \quad  x & \quad \mbox{if } \norm{x} \leq B, \\
       \displaystyle \quad \frac{B}{\norm{x}} x & \quad \mbox{if } \norm{x} > B.
    \end{cases}
\end{eqnarray*}
See Example \ref{Ex:projection is prox} for the definition of the projection onto a closed convex set.

\begin{algorithm}[PSSD]\label{Algo:Stochastic Subgradient Descent Projection}
Consider Problem \tref{Pb:stochastic functions} and let Assumptions \tref{Ass:expectation of convex} and \tref{Ass:SSD bounded solution} hold.
Let $x^0 \in \mathbb{B}(0,B)$, and let $\gamma_t>0$ be a sequence of stepsizes.
The \textbf{Projected Stochastic Subgradient Descent (PSSD)} algorithm is given by the iterates $(x^t)_{t \in \mathbb{N}}$ where
\begin{align*}
    \xi_t & \in \mathbb{R}^q & \mbox{Sampled i.i.d. $\xi_t \sim \mathcal{D}$}\\
    x^{t+1} & = {\rm{proj}}_{\mathbb{B}(0,B)}(x^t - \gamma_t \sgrad_{\xi_t}(x_t)), & \text{ with } \sgrad_{\xi_t}(x_t) \in \partial f_{\xi_t}(x^t).
\end{align*}
\end{algorithm}

We now prove the following theorem,   which is a simplified version of Theorem 19 in~\cite{Defaziofact2020}.

\begin{theorem}\label{theo:stochgradprojB}
Let Assumptions~\tref{Ass:expectation of convex}, \tref{Ass:expectation of lipschitz} and~\tref{ass:boundsol}  hold.
Let $(x^t)_{t \in \mathbb{N}}$ be the  iterates generated by  \tref{Algo:Stochastic Subgradient Descent Projection}, with a decreasing sequence of stepsizes $\gamma_t \eqdef \frac{\gamma_0}{\sqrt{t+1}}$, with $\gamma_0 >0$.
Then we have for $T \geq 2$  and $\bar{x}_T \eqdef \frac{1}{T} \sum_{t=0}^{T-1}x_t$ that
\begin{eqnarray*}
\EE{}{f(\bar{x}_T) - \inf f}
 &\leq &     
 \left( \frac{ 3B^2}{\gamma_0}    + \gamma_0 G^2  \right)\frac{1}{\sqrt{T}}.
\end{eqnarray*}
\end{theorem}

\begin{proof}
We start by using Assumption \tref{Ass:SSD bounded solution} to write ${\rm{proj}}_{\mathbb{B}(0,B)}(x^*) = x^*$.
This together with the fact that the projection is nonexpansive (see Lemma \ref{L:prox nonexpansive} and Example \ref{Ex:projection is prox}) allows us to write, after expanding the squares
\begin{eqnarray*}
\norm{x^{t+1}-x^*}^2 
& \overset{\tref{Algo:Stochastic Subgradient Descent Projection}}{=} & 
\Vert {\rm{proj}}_{\mathbb{B}(0,B)}(x^t - \gamma_t \sgrad_{\xi_t}(x_t)) - {\rm{proj}}_{\mathbb{B}(0,B)}(x^*) \Vert^2 \\
& \leq & 
\norm{x^{t}-x^*- \gamma_t \sgrad_{\xi_t}(x_t)}^2 \\
& =& \norm{x^{t}-x^*}^2 - 2\gamma_t \dotprod{\sgrad_{\xi_t}(x_t),x^t-x^*} + \gamma_t^2 \norm{\sgrad_{\xi_t}(x_t)}^2,\nonumber
\end{eqnarray*}
We now want to take expectation conditioned on $x^t$.
We will use the fact that our subgradients are bounded from Assumption \tref{Ass:expectation of lipschitz}, and that  $\mathbb{E}\left[ \sgrad_{\xi_t}(x_t) \ | \ x^t \right] \in \partial f(x^t)$  (see Lemma \ref{L:expected subgradient formula}).
\begin{eqnarray*}
\EE{}{\norm{x^{t+1}-x^*}^2 \, | \, x^t} & = & \norm{x^{t}-x^*}^2 - 2\gamma_t \dotprod{\mathbb{E}\left[ \sgrad_{\xi_t}(x_t) \ | \ x^t \right],x^t-x^*} + \gamma_t^2 \EE{}{\norm{\sgrad_{\xi_t}(x_t)}^2 \ | \ x^t} \nonumber\\
& \leq &\norm{x^{t}-x^*}^2 - 2\gamma_t \dotprod{\mathbb{E}\left[ \sgrad_{\xi_t}(x_t) \ | \ x^t \right],x^t-x^*} + \gamma_t^2 G^2 \nonumber \\
& \overset{\eqref{eq:defsubgrad}}{ \leq} &\norm{x^{t}-x^*}^2  -2\gamma_t ( f(x^t) -\inf f) + \gamma_t^2 G^2.
\end{eqnarray*}
Taking expectation, dividing through by ${2\gamma_t}$ and
re-arranging gives
\begin{eqnarray*}
\EE{}{f(x^t) -\inf f} &\leq & \frac{1}{2\gamma_t}\EE{}{\norm{x^{t}-x^*}^2} -\frac{1}{2\gamma_t}\EE{}{\norm{x^{t+1}-x^*}^2 }   + \frac{\gamma_t G^2}{2}.
\end{eqnarray*}
Summing up from $t= 0 , \ldots, T-1$ and using telescopic cancellation gives 
\begin{align*}
\sum_{t=0}^{T-1}\EE{}{f(x^t) -\inf f} 
 &\leq 
\frac{1}{2\gamma_0}\norm{x^{0}-x^*}^2 + \frac{1}{2}\sum_{t=0}^{T-2} \left(\frac{1}{\gamma_{t+1}}-\frac{1}{\gamma_{t}}\right)\EE{\mathcal{D}}{\norm{x^{t+1}-x^*}^2}  + \frac{ G^2}{2} \sum_{t=0}^{T-1}\gamma_t.
 \label{eq:intermedproffsgd3}
\end{align*}
In the above inequality, we are going to bound the term
\begin{align*}
\left(\frac{1}{\gamma_{t+1}}-\frac{1}{\gamma_{t}}\right) = \frac{\sqrt{t+2}- \sqrt{t+1}}{\gamma_0} \leq \frac{1}{2\gamma_0\sqrt{t+1}},
\end{align*}
by using the fact that the the square root function is concave.
We are also going to bound the terms of the form $\EE{}{\norm{x^{t}-x^*}^2}$ by using the fact that $x^*$ and the sequence $x^t$ belong to $\mathbb{B}(0,B)$, due to the projection step in the algorithm:
\begin{equation*}
    \norm{x^{t}-x^*}^2 \leq \norm{x^{t}}^2 +2 \dotprod{x^{t},x^*} +\norm{x^*}^2 \leq 4B^2.
\end{equation*}
Finally we are also going to use integral bounds to write (see Lemma \ref{L:sum integral bounds})
\begin{equation*}
    \sum_{t=0}^{T-1} \frac{1}{\sqrt{t+1}} = \sum_{t=1}^T \frac{1}{\sqrt{t}} \leq 2\sqrt{T} - 1.
\end{equation*}
So we can now write (we use $\sqrt{2} - 1 \leq \frac{1}{2}$) :
\begin{align*}
\sum_{t=0}^{T-1}\EE{}{f(x^t) -\inf f} 
&\leq 
\frac{4B^2}{2\gamma_0} 
+ \frac{1}{2} \left( \sum_{t=0}^{T-2} \frac{1}{2 \gamma_0 \sqrt{t+1}} \right) 4B^2 
  + \frac{ \gamma_0 G^2}{2} \sum_{t=0}^{T-1}\frac{1}{\sqrt{t+1}} \nonumber \\
&\leq 
\frac{2B^2}{\gamma_0} 
+ \frac{B^2}{\gamma_0} (2 \sqrt{T-1} -1) 
  + \frac{ \gamma_0 G^2}{2} (2\sqrt{T} -1) \nonumber \\
&\leq 
\frac{B^2}{\gamma_0} 
+ \frac{2B^2}{\gamma_0}  \sqrt{T} 
  +  \gamma_0 G^2 \sqrt{T} \nonumber \\
&\leq 
\left( \frac{3B^2}{\gamma_0} 
 +  \gamma_0 G^2 \right)
\sqrt{T} \nonumber 
 \label{eq:intermedproffsgd6}
\end{align*}
where in the last inequality we used the fact that $1 \leq \sqrt{T}$. 
Finally let  $\bar{x}_T = \frac{1}{T}\sum_{t=0}^{T-1} x_t$, dividing through by $T$,  and using Jensen's inequality we have that 
\begin{eqnarray*}
\EE{}{f(\bar{x}_T) - \inf f} 
& \leq  &
\frac{1}{T}\sum_{t=0}^{T-1}\EE{}{f(x^t) - \inf f} \nonumber \\
 & \leq &  
 \left( \frac{ 3 B^2}{\gamma_0}    + {\gamma_0 G^2} \right)\frac{1}{\sqrt{T}}.
\end{eqnarray*}

\end{proof}

\subsection{Convergence for strongly convex functions with bounded solution}

Here we have to be careful, because there are no functions $f$ that are both strongly convex and Lipschitz continuous, as in Assumption~\tref{Ass:expectation of lipschitz}.

\begin{lemma}\label{lem:nostrongcvxbndgrad}
Consider Problem~\tref{Pb:stochastic functions}.
There exist no functions $f_\xi(x)$ such that $f(x) = \EE{\mathcal{D}}{f_{\xi}(x)}$ is $\mu$--strongly convex and that  Assumption~\tref{Ass:expectation of lipschitz} holds.
\end{lemma}

 \begin{proof} 
For ease of notation we will note $\sgrad(x) \eqdef \mathbb{E}_\mathcal{D}\left[ \sgrad_{\xi_t}(x_t) \right]$, for which we know that $\sgrad(x) \in \partial f(x)$ according to Lemma \ref{L:expected subgradient formula}.
Since $f$ is strongly convex, it admits a minimizer $x^* \in {\rm{argmin}}~f$ (recall Lemma \ref{L:strong convexity minimizers nonsmooth}).
Start by using the strong convexity of $f$ (see Lemma \ref{L:strong convexity hyperplans nonsmooth}) to write, for $x \neq x^*$
\[
\dotprod{\sgrad(x), x-x^*} \geq  f(x)  - f(x^*)+\frac{\mu}{2} \norm{x^*-x}^2 \geq \frac{\mu}{2} \norm{x^*-x}^2 
 \]
 where in the last inequality we used the fact that $f(x)  - f(x^*) \geq 0$.
 Using the above and the Cauchy-Schwarz inequality we have that
 \begin{align*}
     \mu \norm{x-x^*}^2 \leq 2 \dotprod{\sgrad(x), x-x^*} \leq 2\norm{\sgrad(x)} \norm{x-x^*}.
 \end{align*}
Dividing through by $\norm{x-x^*}$ gives
 \begin{eqnarray*}
     \mu \norm{x-x^*}  \leq 2  \norm{g(x)}.
 \end{eqnarray*}
Finally, observe that Assumption \tref{Ass:expectation of lipschitz} implies that $f$ is $G$-Lipschitz, which in turn implies that its subgradients are $G$-bounded (recall Lemma \ref{L:lipschitz equiv bouded subgradients}), that is $\Vert \sgrad(x) \Vert \leq G$.
Since the above holds for all $x\in\R^d$, we only need to take $x \notin \mathbb{B}\left(x^*, \frac{2G}{\mu}\right)$ to arrive at a contradiction.
\end{proof}

The problem in Lemma~\ref{lem:nostrongcvxbndgrad} is that we make two \emph{global} assumptions which are incompatible.
But we can consider a problem where those assumptions are only local.
In the next result, we will assume to know that the solution $x^*$ lives in a certain ball, that the subgradients are bounded on this ball, and we will consider the projected stochastic subgradient method~\tref{Algo:Stochastic Subgradient Descent Projection}.

\begin{theorem}\label{T:SSD CV strong convex}
Let Assumption~\tref{Ass:expectation of convex} hold, and assume further that $f$ is $\mu$-strongly convex.
Let \tref{ass:boundsol} hold, and assume that each $f_\xi$ has $G$-bounded subgradients on $\mathbb{B}(0,B)$, that is
\begin{equation*}
	\Vert x^*_\xi \Vert \leq G \ \text{ for all $x \in \mathbb{B}(0,B)$, for all $\xi \in \mathbb{R}^q$, for all $x_\xi^* \in \partial f_\xi(x)$.}
\end{equation*}
Consider $(x^t)_{t \in \mathbb{N}}$ a sequence generated by the~\tref{Algo:Stochastic Subgradient Descent Projection} algorithm, with a constant stepsize $\gamma_t \equiv \gamma \in ]0, \tfrac{1}{\mu}[$. 
Then for every $T \geq 1$
 \begin{equation*}
\EE{}{\norm{x^{T}-x^*}^2} \leq (1-\gamma \mu)^{T}\norm{x^{0}-x^*}^2+\frac{\gamma G^2}{\mu}.
\end{equation*}
\end{theorem}

\begin{proof}
Our assumption Assumption~\tref{ass:boundsol} guarantees that ${\rm{proj}}_{\mathbb{B}(0,B)}(x^*) = x^*$, so the definition of \tref{Algo:Stochastic Subgradient Descent Projection} together with   the nonexpasiveness of the projection gives
\begin{eqnarray*}
\norm{x^{t+1}-x^*}^2 
& = & 
\Vert {\rm{proj}}_{\mathbb{B}(0,B)}(x^t - \gamma \sgrad_{\xi_t}(x^t)) - {\rm{proj}}_{\mathbb{B}(0,B)}(x^*) \Vert^2  \\
& \leq &
\Vert (x^t - \gamma \sgrad_{\xi_t}(x^t)) - x^* \Vert^2 \\
& =& \norm{x^{t}-x^*}^2 - 2\gamma \dotprod{\sgrad_{\xi_t}(x^t),x^t-x^*} + \gamma^2 \norm{\sgrad_{\xi_t}(x^t)}^2. \nonumber
\end{eqnarray*}
We will now use our assumption that the subgradients are bounded: $\Vert \sgrad_{\xi_t}(x^t) \Vert \leq G$, using that the sequence $x^t$ belongs in $\mathbb{B}(0,B)$ because of the projection step in \tref{Algo:Stochastic Subgradient Descent Projection}.
Next we will use that  $\mathbb{E}\left[ \sgrad_{\xi_t}(x^t) \ | \ x^t \right] \in \partial f(x^t)$  (see Lemma \ref{L:expected subgradient formula}) and that  $f$ is strongly convex (recall Lemma \ref{L:strong convexity hyperplans nonsmooth}).
Taking expectation conditioned on $x^t$, we have that
\begin{eqnarray*}
\EE{}{\norm{x^{t+1}-x^*}^2 \, | \, x^t} & \leq & \norm{x^{t}-x^*}^2 - 2\gamma \dotprod{\mathbb{E}\left[ \sgrad_{\xi_t}(x^t) \ | \ x^t \right],x^t-x^*} + \gamma^2 \EE{}{\norm{\sgrad_{\xi_t}(x^t)}^2 \ | \ x^t} \nonumber\\
& \leq &\norm{x^{t}-x^*}^2 - 2\gamma \dotprod{\mathbb{E}\left[ \sgrad_{\xi_t}(x^t) \ | \ x^t \right],x^t-x^*} + \gamma^2 G^2 \nonumber \\
& \overset{\eqref{eq:strconv nonsmooth}}{ \leq} &
\norm{x^{t}-x^*}^2  -2\gamma ( f(x^t) - \inf f) - \gamma \mu \Vert x^t - x^* \Vert^2 + \gamma^2 G^2 \\
& \leq &
(1- \gamma \mu) \Vert x^t - x^* \Vert^2 + \gamma^2 G^2.
\end{eqnarray*}
Taking expectation on the above, and using a recurrence argument, we can deduce that
\begin{equation*}
    \EE{}{\norm{x^{t}-x^*}^2}
    \leq 
    (1-\gamma \mu)^{t}\norm{x^{0}-x^*}^2+\sum_{k=0}^{t-1} (1-\gamma \mu)^t \gamma^2 G^2.
\end{equation*}
Since
\begin{equation*}\label{eq:geoseriebnd}
 \sum_{k=0}^{t-1} (1-\gamma \mu)^t\gamma^2 G^2  
 = \gamma^2 G^2 \frac{1-(1-\gamma \mu)^{t}}{\gamma \mu} \leq \frac{\gamma^2 G^2}{\gamma \mu} = \frac{\gamma G^2}{\mu},
\end{equation*}
we conclude that
\begin{equation*}
    \EE{}{\norm{x^{t}-x^*}^2}
    \leq 
    (1-\gamma \mu)^{t}\norm{x^{0}-x^*}^2+\frac{\gamma G^2}{\mu}.
\end{equation*}
\end{proof}

\begin{corollary}[$\mathcal{O}\left( \frac{1}{\varepsilon} \right)$ complexity]\label{C:SSD complexity strong convex}
    Consider the setting of Theorem \ref{T:SSD CV strong convex}.
    For every $\varepsilon>0$, we can guarantee that $\EE{}{\norm{x^{T}-x^*}^2} \leq \varepsilon$ provided that
    \begin{equation*}
        \gamma = \min \left\{ \frac{\varepsilon \mu}{2G^2} ; \frac{1}{\mu} \right\}
        \quad \text{ and } \quad 
        T \geq \max \left\{ \frac{2G^2}{\varepsilon \mu^2} ; 1 \right\} \log \left( \frac{8B^2}{\varepsilon} \right).
    \end{equation*}
\end{corollary}

\begin{proof}
    Use Lemma \ref{lem:linear_pls_const} with $A= \frac{G^2}{\mu}$, $C=\mu$, and use the fact that $\Vert x^0 - x^* \Vert \leq 2B$.
\end{proof}

\subsection{Bibliographic notes}

The earlier non-asymptotic proof for the non-smooth case first appeared in the online learning literature, see for example~\cite{Zinkevich2003}. Outside of the online setting, convergence proofs for SGD  in the non-smooth setting with Lipschitz functions was given in~\cite{Shamir013}.
For the non-smooth strongly convex setting see~\cite{simonmarkbach-1-t} where the authors prove a simple $1/t$  convergence rate.

\section{Stochastic Polyak Stepsizes}
\label{sec:polyak}

In section \ref{sec:SGD} about 
\tref{Algo:Stochastic GD}, we saw that in order to set the step size we need to know (an upper bound of) the expected smoothness constant $L_{\max}$.
With that knowledge, we can guarantee a $O(\tfrac{\log(T+1)}{T})$ rate of convergence (see \cref{T:SGD convex smooth vanishing stepsize}).
To further obtain a $\tfrac{1}{\varepsilon^2}$ complexity rate, we saw in \cref{T:SGD complexity convex smooth} that we also need to know (an upper bound of) the variance at the solution $\sigma_f^*$ and $D:= \Vert x_0 - x^* \Vert$.
A similar story can be told for the Stochastic Subgradient Descent studied in section~\ref{sec:stochsubgrad}.
While no knowledge of the problem is needed to obtain a $O(\tfrac{\log(T+1)}{T})$ rate  (see \cref{T:SSD CV convex bounded special stepsizes}), we do need to know the Lipschitz constant $G$ and $D$ (resp. a bound on the solution) to improve it into a $\tfrac{1}{\varepsilon^2}$ complexity (resp. a $\tfrac{1}{\sqrt{T}}$ rate).

In this section we present an alternative approach for solving the \tref{Pb:stochastic functions} problem, where the algorithm and in particular the stepsizes require none of the above mentioned constants to achieve a $\tfrac{1}{\sqrt{T}}$ rate.
Instead, we only require the knowledge of the optimal values  $f_\xi(x^*)$ for some minimizer $x^*$ of $f$.
While this is definitively a lot to ask for in general, the values  $f_\xi(x^*)$ are very easy to access when \emph{interpolation holds} (see Definition \ref{D:interpolation holds}).
Indeed in this case we have $f_\xi(x^*) = \inf f_\xi$, which is often simply zero or at least can be computed.
Let us define this method, which is \tref{Algo:Stochastic Subgradient Descent} with a specific choice of stepsize $\gamma_t$, that we call the Stochastic Polyak Stepsize.

\begin{algorithm}[SPS]\label{Algo:SPS for SSD}
    Let Assumption \tref{Ass:expectation of convex} hold, and let $x^* \in {\rm{argmin}}~f$.
    At every iteration $t \in \mathbb{N}$, the \textbf{Stochastic Polyak Stepsize} (SPS) method is defined as\footnote{For every real $r \in \mathbb{R}$, we note $r_+ := \max\{0,r\}$ the positive part of $r$.} 
    \[
    \begin{array}{rlcl}
        \xi_t & \in \mathbb{R}^q & 
        \quad & 
        \mbox{ sampled i.i.d. }\xi_t \sim \mathcal{D} 
        \\
        \gamma_t & =   \dfrac{(f_{\xi_t}(x^t) -f_{\xi_t}(x^*))_+}{\norm{  \sgrad_{\xi_t}(x^t)}^2}&
        \quad &
        \text{ if } \sgrad_{\xi_t}(x^t) \neq 0, \gamma_t =0 \text{ otherwise,} \\
        x^{t+1} & = x^t - \gamma_t \sgrad_{\xi_t}(x^t), & \quad & \text{ with } \sgrad_{\xi_t}(x^t) \in \partial f_{\xi_t}(x^t).
    \end{array}
    \]
\end{algorithm}

\begin{remark}[The SPS is optimal in a certain sense]
    One motivation behind this definition is that the Stochastic Polyak Stepsize is a good choice for \tref{Algo:Stochastic Subgradient Descent}
    if we want to minimize the decrease of the Lyapunov energy $\Vert x^t -x^* \Vert^2$.
    To see this, write its (discrete) derivative, expand the squares and use convexity to write
    \begin{eqnarray}
    \norm{x^{t+1} -x^*}^2  - \norm{x^t -x^*}^2 
    & = & \nonumber
    - 2\gamma_t\dotprod{  \sgrad_{\xi_t}(x^t), x^t-x^* } + \gamma_t^2\norm{  \sgrad_{\xi_t}(x^t)}^2 \\
    & \leq & \label{eq:sps optimal upper bound}
    \gamma_t^2 \norm{ \sgrad_{\xi_t}(x^t)}^2 - 2\gamma_t  (f_{i_t}(x^t) - f_{i_t}(x^*)).
\end{eqnarray}
It is a simple exercise (see Lemma \ref{L:sps optimal upper bound}) to see that among all possible non-negative stepsizes $\gamma_t$, the Stochastic Polyak Stepsize is the one minimizing the right-hand side of this upper bound.
\end{remark}

Before proving any convergence result, we point out that \tref{Algo:SPS for SSD} enjoys a remarkable property for a stochastic method: its iterates are Fejér monotonic, almost surely.

\begin{lemma}[Monotonicity for SPS]\label{L:SPS monotone fejer}
    Let Assumption \tref{Ass:expectation of convex} hold and let $x^* \in {\rm{argmin}}~f$ be fixed. 
    If $(x^t)_{t \in \mathbb{N}}$ is generated  by \tref{Algo:SPS for SSD}, then the iterates are Fejér monotonic: for every $t \geq 0$, almost surely:
    \begin{equation}\label{eq:sps fejer monotonic plus}
        \norm{x^{t+1} -x^*}^2
        \leq
        \norm{x^t -x^*}^2
        - \gamma_t (f_{\xi_t}(x^t) - f_{\xi_t}(x^*))_+.
    \end{equation}
\end{lemma}

\begin{proof}
    Consider the inequality obtained in \eqref{eq:sps optimal upper bound}.
    If $\sgrad_{\xi_t}(x^t) = 0$, then by definition \tref{Algo:SPS for SSD} we have $\gamma_t = 0$.
    In this case \eqref{eq:sps optimal upper bound} becomes exactly $\norm{x^{t+1} -x^*}^2 \leq \norm{x^t -x^*}^2$.
    Consider now the case where  $\sgrad_{\xi_t}(x^t) \neq 0$, then we can inject the definition of $\gamma_t$ into \eqref{eq:sps optimal upper bound} to obtain
\begin{eqnarray}
	\norm{x^{t+1} -x^*}^2  - \norm{x^t -x^*}^2 
	&\leq & \notag
	\gamma_t^2 \norm{ \sgrad_{\xi_t}(x^t)}^2 - 2\gamma_t  (f_{\xi_t}(x^t) - f_{\xi_t}(x^*)) \\
	& = & \label{sps1}
	- \frac{(f_{\xi_t}(x^t) - f_{\xi_t}(x^*))_+^2}{\norm{\sgrad_{\xi_t}(x^t)}^2} \\
    & = & \nonumber
    - \gamma_t (f_{\xi_t}(x^t) - f_{\xi_t}(x^*))_+.
\end{eqnarray}
In both cases, we see that \eqref{eq:sps fejer monotonic plus} holds, and therefore that the iterates are Fejér monotonic.
\end{proof}

Because of this special property, we will always know a priori that the iterates of \tref{Algo:SPS for SSD} are bounded.
Therefore, any assumptions needed for the algorithm to work only need to be assumed on bounded sets.
For instance, instead of assuming the $f_\xi$ to have globally bounded subgradients as in Assumption \tref{Ass:expectation of lipschitz}, in this section we will only need a local boundedness assumption.

\begin{assumption}[$G$-bounded subgradients on $D$-solution ball]\label{Ass:bounded subgradients on solution ball}
    Let Assumption \tref{Ass:expectation of convex} hold, let $x^* \in \argmin~f$ and let $D \geq 0$.
    We assume that there exists $G \geq 0$ such that
    \begin{equation*}
        \Vert \eta_\xi \Vert \leq G, \quad 
        \text{ for every } x \in \mathbb{B}(x^*, D), \ \xi \sim \mathcal{D}, \ \eta_\xi \in \partial f_\xi(x).
    \end{equation*}
\end{assumption}

Let us take a moment to comment this assumption.
Fist, for a fixed function $f_\xi$, there always exists $G_\xi \geq 0$ such that its subgradients are $G_\xi$-bounded on $\mathbb{B}(x^*,D)$.
This is a consequence of the fact that our functions $f_\xi$ take finite values, and that finite convex functions in $\mathbb{R}^d$ have bounded subgradients on bounded sets, see \cite[Proposition 16.20]{BauCom}.
This being said, we understand that Assumption \tref{Ass:bounded subgradients on solution ball} is equivalent to require $\sup_\xi G_\xi <+ \infty$.
Therefore, we see that this assumption is automatically verified whenever the distribution $\mathcal{D}$ of the $\xi$'s has finite support, in other words if we are dealing with a \tref{Pb:sum of functions} problem!

\subsection{Convergence for convex functions with locally bounded subgradients}

\begin{theorem}\label{theo:SPS}
Let Assumption \tref{Ass:expectation of convex} hold and let $x^* \in {\rm{argmin}}~f$ be fixed. 
Let $(x^t)_{t \in \mathbb{N}}$ be generated by \tref{Algo:SPS for SSD}, and let $D := \Vert x^0 - x^* \Vert$.
Assume further that Assumption \tref{Ass:bounded subgradients on solution ball} holds.
Then, for every $T \geq 1$ and $\bar x^T := \frac{1}{T}\sum_{t=0}^{T-1} x^t$,
\begin{equation*}
    \E{f(\bar{x}^T) - \inf f}
    \leq
    \frac{GD}{\sqrt{T}}.
\end{equation*}
\end{theorem}

\begin{proof}
Start with the monotonicity lemma \ref{L:SPS monotone fejer}:
\begin{equation}\label{sps1.5}
        \gamma_t (f_{\xi_t}(x^t) - f_{\xi_t}(x^*))_+
        \leq
        \norm{x^t -x^*}^2 - \norm{x^{t+1} -x^*}^2,
\end{equation}
which in particular implies that $x^t \in \mathbb{B}(x^*,D)$.
Now we claim that 
\begin{equation}\label{sps1.6}
    \gamma_t (f_{\xi_t}(x^t) - f_{\xi_t}(x^*))_+ \geq \tfrac{1}{G^2}(f_{\xi_t}(x^t) - f_{\xi_t}(x^*))_+^2.
\end{equation}
In the case that $\sgrad_{\xi_t}(x^t) \neq 0$, this is a direct consequence of the definition of $\gamma_t$, the fact that $x^t \in \mathbb{B}(x^*,D)$ combined with Assumption \tref{Ass:bounded subgradients on solution ball}:
\begin{equation*}
    \gamma_t (f_{\xi_t}(x^t) - f_{\xi_t}(x^*))_+
    =
    \frac{(f_{\xi_t}(x^t) - f_{\xi_t}(x^*))_+^2}{\norm{\sgrad_{\xi_t}(x^t)}^2}
    \geq \frac{(f_{\xi_t}(x^t) - f_{\xi_t}(x^*))_+^2}{{G}^2}.
\end{equation*}
In the case that $\sgrad_{\xi_t}(x^t) = 0$, we have on the one hand that $\gamma_t = 0$ (by definition of the stepsize in \tref{Algo:SPS for SSD}).
On the other hand, this implies $0 \in \partial f_{\xi_t}(x^t)$ and this optimality condition in turns imply that $x^t$ is a minimizer for $f_{\xi_t}$ (recall \cref{P:Fermat convex nonsmooth}).
In particular we have that $f_{\xi_t}(x^t) - f_{\xi_t}(x^*) \leq 0$ and so that $(f_{\xi_t}(x^t) - f_{\xi_t}(x^*))_+=0$.
We can the combine \eqref{sps1.5} and \eqref{sps1.6} to write 
\begin{equation}\label{sps1.7}
        \frac{(f_{\xi_t}(x^t) - f_{\xi_t}(x^*))_+^2}{G^2}
        \leq
        \norm{x^t -x^*}^2 - \norm{x^{t+1} -x^*}^2.
\end{equation}
Now take expectation, and use Jensen's inequality on the function $t \mapsto t_+^2$ to obtain
\begin{equation*}
    \frac{1}{G^2} \left( \mathbb{E}\left[ f(x^t) - \inf f \right] \right)^2
    \leq
    \frac{1}{G^2} \mathbb{E}\left[ (f_{\xi_t}(x^t) - f_{\xi_t}(x^*))_+^2 \right]
    \leq
    \mathbb{E}\left[\norm{x^t -x^*}^2 \right] - \mathbb{E}\left[\norm{x^{t+1} -x^*}^2 \right].
\end{equation*}
Sum the above inequality for $t=0,\dots T-1$, see that there is a telescopic sum, multiply by $G^2/T$ and get rid of trivially negative terms to obtain
\begin{equation*}
    \frac{1}{T}\sum_{t=0}^{T-1} \left( \mathbb{E}\left[ f(x^t) - \inf f \right] \right)^2
    \leq
    \frac{G^2}{T} \left( \mathbb{E}\left[\norm{x^0 -x^*}^2 \right] - \mathbb{E}\left[\norm{x^{T} -x^*}^2 \right] \right)
    \leq
    \frac{D^2 G^2}{T}.
\end{equation*}
It remains to introduce $\bar x^T := \tfrac{1}{T}\sum_{t=0}^{T-1} x^t$ and to use Jensen' inequality twice, on both functions $f$ and $t \mapsto t^2$, together with the fact that $t\mapsto t^2$ is increasing, to conclude that
\begin{eqnarray*}
    \left( \mathbb{E}\left[ f(\bar x^T) - \inf f \right] \right)^2
    & \leq  &
    \left( \mathbb{E}\left[\frac{1}{T}\sum_{t=0}^{T-1} f( x^t) - \inf f \right] \right)^2
    =
    \left( \frac{1}{T}\sum_{t=0}^{T-1} \mathbb{E}\left[f( x^t) - \inf f \right] \right)^2 \\
    & \leq &
    \frac{1}{T}\sum_{t=0}^{T-1}\left( \mathbb{E}\left[f( x^t) - \inf f \right] \right)^2    \\
    & \leq &\frac{D^2 G^2}{T}.
\end{eqnarray*}
\end{proof}

\subsection{Convergence for strongly convex functions with locally bounded subgradients}

\begin{theorem}\label{T:SPS strongly convex lipschitz}
Let Assumption \tref{Ass:expectation of convex} hold, suppose that $f$ is $\mu$-strongly convex, and let $x^* \in {\rm{argmin}}~f$ be fixed. 
Let $(x^t)_{t \in \mathbb{N}}$ be generated by \tref{Algo:SPS for SSD}, and let $D := \Vert x^0 - x^* \Vert$.
Assume further that Assumption \tref{Ass:bounded subgradients on solution ball} holds.
Then, for every $T \geq 1$,
\begin{equation*}
    \mathbb{E}\left[ \norm{x^T -x^*}^2 \right]
    \leq
    \frac{4G^2}{\mu^2 T}.
\end{equation*}
\end{theorem}

\begin{proof}
    Start with the monotonicity lemma \ref{L:SPS monotone fejer}:
\begin{equation*}
        \gamma_t (f_{\xi_t}(x^t) - f_{\xi_t}(x^*))_+
        \leq
        \norm{x^t -x^*}^2 - \norm{x^{t+1} -x^*}^2,
\end{equation*}
which in particular implies that $x^t \in \mathbb{B}(x^*,D)$.
Using Assumption \tref{Ass:bounded subgradients on solution ball}, we can deduce that (see the proof of \eqref{sps1.6} for more details)
\begin{equation*}
        \frac{(f_{\xi_t}(x^t) - f_{\xi_t}(x^*))_+^2}{G^2}
        \leq
        \norm{x^t -x^*}^2 - \norm{x^{t+1} -x^*}^2.
\end{equation*}
Now take expectation, and use Jensen's inequality on the function $t \mapsto t_+^2$ to obtain
\begin{equation*}
    \frac{1}{G^2} \left( \mathbb{E}\left[ f(x^t) - \inf f \right] \right)^2
    \leq
    \frac{1}{G^2} \mathbb{E}\left[ (f_{\xi_t}(x^t) - f_{\xi_t}(x^*))_+^2 \right]
    \leq
    \mathbb{E}\left[\norm{x^t -x^*}^2 \right] - \mathbb{E}\left[\norm{x^{t+1} -x^*}^2 \right].
\end{equation*}
Now use the strong convexity of $f$ through \cref{L:strong convexity hyperplans nonsmooth} to write
\begin{equation*}
    \frac{\mu^2}{4 G^2} \left( \mathbb{E}\left[\Vert x^t - x^* \Vert^2 \right] \right)^2
    =
    \frac{1}{G^2} \left( \mathbb{E}\left[ \frac{\mu}{2}\Vert x^t - x^* \Vert^2 \right] \right)^2
    \leq
    \mathbb{E}\left[ \norm{x^t -x^*}^2 \right] - \mathbb{E}\left[\norm{x^{t+1} -x^*}^2 \right].
\end{equation*}
Setting $a_t := \mathbb{E}\left[ \norm{x^t -x^*}^2 \right]$ and $c=\tfrac{\mu^2}{4G^2}$, we have shown that $ca_t^2 \leq a_t - a_{t+1}$.
According to lemma \ref{L:sublinear sequence 1/t} this means that $a_t \leq \tfrac{1}{ct}$ for every $t \geq 1$, which is what we wanted to prove.
\end{proof}

\subsection{Bibliographic Notes}

The derivation of \tref{Algo:SPS for SSD} is taken from~\cite{garrigos2023function}, and finds its roots in Polyak's manuscript \cite{polyak1987introduction}.
It is almost identical to the derivation of other Stochastic Polyak stepsizes, such as the ones  proposed in~\cite{SPS,ALI-G}. 
See in particular \cite{SPS} where the authors replace $f_i(x^*)$ by $\inf f_i$, and obtain a complexity adaptive to the smoothness constant.

The proof of \cref{theo:SPS} for convex functions can be found in \cite{garrigos2023function}.
The proof of \cref{T:SPS strongly convex lipschitz} for strongly convex functions was given in the blogpost~\cite{pedregosa2023sps}.

\section{Proximal Gradient Descent}

\begin{problem}[Composite]\label{Pb:Nonsmooth}
We want to minimize a function $F : \mathbb{R}^d \to \mathbb{R} \cup \{+\infty\}$ which is a composite sum given by
\begin{equation*}
F(x) = f(x)+g(x),
\end{equation*}
where $f:\mathbb{R}^d \to \mathbb{R}$ is differentiable, and $g : \mathbb{R}^d \to \mathbb{R} \cup \{+\infty\}$ is proper l.s.c.
We require that the problem is well-posed, in the sense that ${\rm{argmin}}~F \neq \emptyset$.
\end{problem}

To exploit the structure of this composite sum, we will use the proximal gradient descent algorithm, which alternates gradient steps with respect to the differentiable term $f$, and proximal steps with respect to the nonsmooth term $g$.

\begin{algorithm}[PGD]\label{Algo:proximal gradient descent}
Let $x^0 \in \mathbb{R}^d$, and let $\gamma >0$ be a stepsize.
The \textbf{Proximal Gradient Descent (PGD)} algorithm defines a sequence $(x^t)_{t \in \mathbb{N}}$ which satisfies
\begin{equation*}
 x^{t+1}  =  \prox_{\gamma g} (x^t-\gamma \nabla f(x^t) ).
\end{equation*}
\end{algorithm}

\subsection{Convergence for convex functions}

\begin{theorem}\label{theo:convproxgrad}
Consider the Problem \tref{Pb:Nonsmooth}, and suppose that $g$ is convex, and that $f$ is convex and $L$-smooth, for some $L>0$. 
Let $(x^t)_{t \in \mathbb{N}}$ be the sequence of iterates generated by the algorithm \tref{Algo:proximal gradient descent}, with a stepsize $\gamma \in ]0,\frac{1}{L}]$. 
Then, for all $x^* \in {\rm{argmin}}~F$, for all $t \in \mathbb{N}$ we have that
\begin{equation*}
F(x^t)-\inf F \leq \frac{\norm{x^0-x^*} ^2}{2\gamma t}.
\end{equation*}
 \end{theorem}

\begin{proof}
Let $x^* \in {\rm{argmin}}~F$ be any minmizer of $F$.
We start by studying two (decreasing and nonnegative) quantities of interest : $F(x^t) - \inf F$ and $\Vert x^t - x^* \Vert^2$.

First, we show that $F(x^{t+1}) - \inf F$ decreases. 
For this, using the definition \ref{D:proximal operator} of $\prox_{\gamma g}$  together with the definition of \tref{Algo:proximal gradient descent}, we have that
\[x^{t+1} =  \underset{x' \in \mathbb{R}^d}{\rm{argmin}} \  \frac{1}{2}\norm{x' - (x^t - \gamma \nabla f(x^t))}^2+ \gamma g(x').\]
Consequently
\begin{equation*}
    g(x^{t+1}) + \frac{1}{2 \gamma}\Vert x^{t+1} - (x^t - \gamma \nabla f(x^t)) \Vert^2 
    \leq
    g(x^t) + \frac{1}{2 \gamma}\Vert x^{t} - (x^t - \gamma \nabla f(x^t)) \Vert^2.
\end{equation*}
After expanding the squares and rearranging the terms, we see that the above inequality is equivalent to
\begin{equation}\label{cpg2}
    g(x^{t+1}) - g(x^t) \leq \frac{-1}{2 \gamma} \Vert x^{t+1} - x^t \Vert^2 - \langle \nabla f(x^t), x^{t+1} - x^t \rangle.
\end{equation}
Now, we can use the fact that $f$ is $L$-smooth and \eqref{eq:smoothnessfunc} to write
\begin{equation}\label{cpg3}
    f(x^{t+1}) - f(x^t) \leq \langle \nabla f(x^t), x^{t+1} - x^t \rangle + \frac{L}{2}\Vert x^{t+1} - x^t \Vert^2.
\end{equation}
Summing \eqref{cpg2} and \eqref{cpg3}, and using the fact that $\gamma L \leq 1$, we  obtain that
\begin{equation}\label{PGD values decreasing}
    F(x^{t+1}) - F(x^t) \leq \frac{-1}{2 \gamma} \Vert x^{t+1} - x^t \Vert^2 + \frac{L}{2}\Vert x^{t+1} - x^t \Vert^2
    \leq 0.
\end{equation}
Consequently $F(x^{t}) - \inf F$ is decreasing.

Now we show that  $\Vert x^{t+1} - x^* \Vert^2$ is decreasing. 
For this we first expand the squares as follows
\begin{equation}
    \frac{1}{2 \gamma} \Vert x^{t+1} - x^* \Vert^2 - \frac{1}{2 \gamma} \Vert x^t - x^* \Vert^2
    =
    \frac{-1}{2 \gamma}\Vert x^{t+1} - x^t \Vert^2
    - \langle \frac{x^t - x^{t+1}}{\gamma}, x^{t+1} - x^* \rangle.\label{eq:tempmo8z4z}
\end{equation}
Since $x^{t+1} = \prox_{\gamma g}(x^t - \gamma \nabla f(x^t))$, we know from \Cref{L:prox characterization subdifferential} that 
$$\frac{x^t - x^{t+1}}{\gamma} \in \nabla f(x^t) + \partial g(x^{t+1}).$$
Using the above in~\eqref{eq:tempmo8z4z} we have that there exists some $\eta^{t+1} \in \partial g(x^{t+1})$ such that
\begin{eqnarray*}
    & &\frac{1}{2 \gamma} \Vert x^{t+1} - x^* \Vert^2 - \frac{1}{2 \gamma} \Vert x^t - x^* \Vert^2\\
    & =&
    \frac{-1}{2 \gamma}\Vert x^{t+1} - x^t \Vert^2
    - \langle  \nabla f(x^t) + \eta^{t+1}, x^{t+1} - x^* \rangle 
    \\
    &=&
    \frac{-1}{2 \gamma}\Vert x^{t+1} - x^t \Vert^2
    - \langle \eta^{t+1}, x^{t+1} - x^* \rangle 
    - \langle \nabla f(x^t), x^{t+1} - x^t \rangle
    +
    \langle \nabla f(x^t), x^* - x^t \rangle.
\end{eqnarray*}
On the first inner product term we can use that $\eta^{t+1} \in \partial g(x^{t+1})$ and definition of subgradient~\eqref{eq:defsubgrad} to write
\begin{equation}\label{cpg5}
    - \langle \eta^{t+1}, x^{t+1} - x^* \rangle
    =  \langle \eta^{t+1}, x^* - x^{t+1} \rangle
    \leq g(x^*) - g(x^{t+1}).
\end{equation}
On the second inner product term  we can use the smoothness of $L$ and \eqref{eq:smoothnessfunc} to write
\begin{equation}\label{cpg6}
    - \langle \nabla f(x^t), x^{t+1} - x^t \rangle
    \leq
    \frac{L}{2}\Vert x^{t+1} - x^t \Vert^2 + f(x^t) - f(x^{t+1}).
\end{equation}
On the last term we can use the convexity of $f$ and \eqref{eq:conv} to write
\begin{equation}\label{cpg7}
    \langle \nabla f(x^t), x^* - x^t \rangle \leq f(x^*) - f(x^t).
\end{equation}
By combining \eqref{cpg5}, \eqref{cpg6}, \eqref{cpg7}, and using the fact that $\gamma L \leq 1$, we obtain
\begin{eqnarray}
    \frac{1}{2 \gamma} \Vert x^{t+1} - x^* \Vert^2 - \frac{1}{2 \gamma} \Vert x^t - x^* \Vert^2
    &\leq & 
   \frac{-1}{2 \gamma}\Vert x^{t+1} - x^t \Vert^2
   +\frac{L}{2}\Vert x^{t+1} - x^t \Vert^2  \nonumber \\
   & & \quad 
      + g(x^*) - g(x^{t+1})
      + f(x^t) - f(x^{t+1})
    + f(x^*) - f(x^t)  \nonumber
    \\ \notag
   &= &   \frac{-1}{2 \gamma}\Vert x^{t+1} - x^t \Vert^2
    + \frac{L}{2}\Vert x^{t+1} - x^t \Vert^2
    - (F(x^{t+1}) -  F(x^*))\\
    & \leq & \label{PGD iterates decreasing}
    - (F(x^{t+1}) -  \inf F).
\end{eqnarray}
Now that we have established that the iterate gap and functions values are decreasing, we want to show that the  Lyapunov energy
\begin{equation*}
    E_t := 
    \frac{1}{2 \gamma} \Vert x^t - x^* \Vert^2 + t(F(x^t) - \inf F),
\end{equation*}
is decreasing. Indeed, re-arranging the terms and
using \eqref{PGD values decreasing} and \eqref{PGD iterates decreasing} we have that 
\begin{eqnarray}
E_{t+1} - E_t 
& =& \notag
(t+1)(F(x^{t+1}) - \inf F) - t(F(x^t) - \inf F) + \frac{1}{2 \gamma} \Vert x^{t+1} - x^* \Vert^2 - \frac{1}{2 \gamma} \Vert x^t - x^* \Vert^2 \\
& =& 
F(x^{t+1}) - \inf F + t(F(x^{t+1}) - F(x^t)) + \frac{1}{2 \gamma} \Vert x^{t+1} - x^* \Vert^2 - \frac{1}{2 \gamma} \Vert x^t - x^* \Vert^2 \nonumber \\
& \overset{\eqref{PGD values decreasing}}{\leq} & F(x^{t+1}) - \inf F + \frac{1}{2 \gamma} \Vert x^{t+1} - x^* \Vert^2 - \frac{1}{2 \gamma} \Vert x^t - x^* \Vert^2  \;
\overset{\eqref{PGD iterates decreasing}}{\leq} \;
0.\label{cpg1}
\end{eqnarray}
We have shown that $E_t$ is decreasing,  therefore we can write that
\begin{equation*}
    t(F(x^t) - \inf F) 
    \;\leq\; E_t \;\leq \;E_0\; =\; \frac{1}{2\gamma} \Vert x^0 - x^* \Vert^2,
\end{equation*}
and the conclusion follows after dividing by $t$.
\end{proof}

\begin{corollary}[$\mathcal{O}(1/\varepsilon)$ complexity]
Consider the setting of Theorem~\ref{theo:convproxgrad}, for a given $\varepsilon >0$ and $\gamma = L$ we have that 
\begin{equation*}
t \geq  \frac{L}{\varepsilon}\frac{\norm{x^0-x^*} ^2}{2 } \; \implies \; F(x^t) -\inf F \leq \varepsilon.
\end{equation*}
\end{corollary}

\subsection{Convergence for strongly convex functions}

\begin{theorem}\label{T:CV PGD strongly convex}
Consider the Problem \tref{Pb:Nonsmooth}, and suppose that $h$ is convex, and that $f$ is $\mu$-strongly convex and $L$-smooth, for some $L \geq \mu >0$. 
Let $(x^t)_{t \in \mathbb{N}}$ be the sequence of iterates generated by the algorithm \tref{Algo:proximal gradient descent}, with a stepsize $ 0 < \gamma \leq \frac{1}{L}$.
Then, for $x^* = {\rm{argmin}}~F$ and  $t \in \mathbb{N}$ we have that
\begin{equation*}
\norm{x^{t+1} -x^*}^2  \;\; \leq \;\;  \left(1-\gamma \mu\right) \norm{x^t -x^*}^2.
\end{equation*}
\end{theorem}

As for \tref{Algo:gradient descent constant stepsize} (see Theorem \ref{theo:gradstrconv}) we provide two different proofs here.

\begin{proof}[\rm\bf Proof of Theorem \ref{T:CV PGD strongly convex} with first-order properties]
Use the definition of \tref{Algo:proximal gradient descent} together with Lemma \ref{L:prox fixed point composite}, and the nonexpansiveness of the proximal operator (Lemma \ref{L:prox nonexpansive}), to write
\begin{eqnarray*}
    \Vert x^{t+1} - x^* \Vert^2
    & \leq &
    \Vert \prox_{\gamma g}(x^t - \gamma \nabla f(x^t)) - \prox_{\gamma g}(x^* - \gamma \nabla f(x^*)) \Vert ^2 \\
    & \leq &
    \Vert (x^t -x^*) - \gamma (\nabla f(x^t) - \nabla f(x^*)) \Vert^2 \\
    & = &
    \Vert x^t - x^* \Vert^2 
    + \gamma^2 \Vert \nabla f(x^t) - \nabla f(x^*) \Vert^2 
    - 2 \gamma \langle \nabla f(x^t) - \nabla f(x^*), x^t - x^* \rangle.
\end{eqnarray*}
The cocoercivity of $f$ (Lemma \ref{lem:convandsmooth}) gives us
\begin{equation*}
    \gamma^2 \Vert \nabla f(x^t) - \nabla f(x^*) \Vert^2
    \leq
    2\gamma^2 L \left( f(x^t) - f(x^*) - \langle \nabla f(x^*), x^t - x^* \rangle \right),
\end{equation*}
while the strong convexity of $f$ gives us (Lemma \ref{L:strong convexity differentiable hyperplans})
\begin{eqnarray*}
    - 2 \gamma \langle \nabla f(x^t) - \nabla f(x^*), x^t - x^* \rangle
    & = &
    2 \gamma \langle \nabla f(x^t), x^* - x^t \rangle
    +
    2 \gamma \langle \nabla f(x^*), x^t - x^* \rangle \\
    &\leq  &
    2 \gamma \left( f(x^*) - f(x^t) - \frac{\mu}{2}\Vert x^t - x^* \Vert^2 \right)
    +
    2 \gamma \langle \nabla f(x^*), x^t - x^* \rangle \\
    &=&
    - \gamma \mu \Vert x^t - x^* \Vert^2 - 2 \gamma \left( f(x^t) - f(x^*) - \langle \nabla f(x^*), x^t - x^* \rangle \right)
\end{eqnarray*}
Combining those three inequalities and rearranging the terms, we obtain
\begin{eqnarray*}
    \Vert x^{t+1} - x^* \Vert^2
    & \leq &
    (1- \gamma \mu) \Vert x^t - x^* \Vert^2 
     +
    (2\gamma^2 L - 2 \gamma) \left( f(x^t) - f(x^*) - \langle \nabla f(x^*), x^t - x^* \rangle \right).
\end{eqnarray*}
We conclude after observing that $f(x^t) - f(x^*) - \langle \nabla f(x^*), x^t - x^* \rangle \geq 0$ (because $f$ is convex, see Lemma \ref{L:convexity via hyperplanes}), and that $2\gamma^2 L - 2 \gamma \leq 0$ (because of our assumption on the stepsize).
\end{proof}

\begin{proof}[\rm\bf Proof of Theorem \ref{T:CV PGD strongly convex} with the Hessian]
Let $T(x) := x - \gamma \nabla f(x)$ 
so that the iterates of \tref{Algo:proximal gradient descent}  verify $x^{t+1} = \prox_{\gamma g}  (T(x^t))$.
From Lemma \ref{L:prox fixed point composite} we know that $\prox_{\gamma g}( T(x^*)) = x^*$, so we can write
\begin{equation*}
    \norm{x^{t+1} -x^*} =
\Vert \prox_{\gamma h}( T(x^t)) - \prox_{\gamma g}(T(x^*)) \Vert.
\end{equation*}
Moreover, we know from Lemma \ref{L:prox nonexpansive} that $\prox_{\gamma g}$ is $1$-Lipschitz, so
\begin{equation*}
    \norm{x^{t+1} -x^*} 
    \leq
\Vert  T(x^t) - T(x^*) \Vert.
\end{equation*}
Further, we already proved in the proof of Theorem \ref{theo:gradstrconv} that $T$ is $(1 - \gamma \mu)$-Lipschitz (assuming further that $f$ is twice differentiable).
Consequently,
\begin{equation*}
    \norm{x^{t+1} -x^*} 
    \leq
    (1 - \gamma \mu) \Vert x^t - x^* \Vert.
\end{equation*}
To conclude the proof, take the squares in the above inequality, and use the fact that $(1 - \gamma \mu)^2 \leq (1 - \gamma \mu)$.
\end{proof}

\begin{corollary}[$\log(1/\varepsilon)$ Complexity]
Consider the setting of Theorem~\ref{T:CV PGD strongly convex}, for a given $\varepsilon>0,$ we have that if $\gamma = 1/L$ then
\begin{equation*} 
T\geq \frac{L}{\mu} \log\left(\frac{\norm{x^{0}-x^*}^2}{\varepsilon}\right)  \quad \Rightarrow \quad \norm{x^{T}-x^*}^2 \leq \varepsilon.
\end{equation*}
\end{corollary}

\begin{proof}
    This is a direct consequence of lemma \ref{lem:itercomplex} in the appendix.
\end{proof}

\subsection{Bibliographic notes}

A proof of a $\mathcal{O}(\frac{1}{T})$ convergence rate for the \tref{Algo:proximal gradient descent} algorithm in the convex case can be found in \cite[Theorem 3.1]{BecTeb09}.
The linear convergence rate in the strongly convex case can be found in \cite[Proposition 4]{SchRouBac11a}.

\section{Proximal Stochastic  Gradient Descent}
\label{sec:sgdprox}

\begin{problem}[Composite Sum of Functions]\label{Pb:nonsmooth sum of functions}
We want to minimize a function $F : \mathbb{R}^d \to \mathbb{R}$ which writes as a composite sum
\begin{equation*}
F(x) \eqdef g(x) + f(x), 
\quad
f(x) \eqdef \frac{1}{n} \sum_{i=1}^n f_i(x),
\end{equation*}
where each $f_i : \mathbb{R}^d \to \mathbb{R}$ is differentiable, and $g : \mathbb{R}^d \to \mathbb{R} \cup \{+\infty\}$ is proper l.s.c.
We require that the problem is well-posed, in the sense that ${\rm{argmin}}~F \neq \emptyset$, and each $f_i$ is bounded from below.
\end{problem}

\begin{assumption}[Composite Sum of Convex]\label{Ass:SPGD nonsmooth sum of convex}
We consider the Problem \tref{Pb:nonsmooth sum of functions} and we suppose that $h$ and each $f_i$ are convex.
\end{assumption}

\begin{assumption}[Composite Sum of $L_{\max}$--Smooth]\label{Ass:SPGD nonsmooth sum of smooth}
We consider the Problem \tref{Pb:nonsmooth sum of functions}
and suppose that each $f_i$ is $L_i$-smooth.
We note $L_{\max} = \max\limits_{i=1,\dots,n} L_i$.
\end{assumption}

\begin{algorithm}[PSGD]\label{Algo:Proximal Stochastic GD}
Consider the Problem \tref{Pb:nonsmooth sum of functions}.
Let $x^0 \in \mathbb{R}^d$, and let $\gamma_t>0$ be a sequence of step sizes.
The \textbf{Proximal Stochastic Gradient Descent (PSGD)} algorithm defines a sequence $(x^t)_{t \in \mathbb{N}}$ satisfying
\begin{align*}
    i_t & \in \{1,\ldots n\} & \mbox{Sampled with probability }\frac{1}{n}, \nonumber\\
    x^{t+1} & = \prox_{\gamma g}\left(  x^t - \gamma_t\nabla f_{i_t}(x^t)\right). 
\end{align*}
\end{algorithm}

\subsection{Complexity for convex functions}

\begin{theorem}\label{theo:sgdproxconvex varying stepsizes}
Let Assumptions \tref{Ass:SPGD nonsmooth sum of convex} and \tref{Ass:SPGD nonsmooth sum of smooth} hold. Let $\sigma^*_F$ be the ~\tref{D:gradient solution variance composite}.
Let $(x^t)_{t\in \mathbb{N}}$ be a sequence generated by the \tref{Algo:Proximal Stochastic GD} algorithm with a nonincreasing sequence of stepsizes verifying $0< \gamma_0 \leq \tfrac{1}{8 L_{\max}}$.
Then, for all $T \geq 1$, all $x^* \in {\rm{argmin}}~F$, and $\bar x^T \eqdef \frac{1}{\sum_{t=0}^{T-1} \gamma_t}\sum_{t=0}^{T-1} \gamma_t x^{t+1}$: 
\begin{equation*}
    \mathbb{E}\left[  F(\bar x^T) - \inf F \right]
	\leq 
	\frac{C_0}{\sum_{t=0}^{T-1} \gamma_t}
	+
	\frac{4 \sigma^*_F \sum_{t=0}^{T-1} \gamma_t^2}{\sum_{t=0}^{T-1} \gamma_t}, 
	\quad C_0 = \Vert x^{0} - x^* \Vert^2
	+
	\tfrac{1}{8 L_{\max}} (F(x^{0})-  \inf F).
\end{equation*}
\end{theorem}

\begin{proof}
Let us start by looking at  $\Vert x^{t+1} - x^* \Vert^2 -  \Vert x^t - x^* \Vert^2$. Since we just compare $x^t$ to $x^{t+1}$, to lighten the notations we fix $j := i_t$. Expanding the squares, we have that
\begin{equation*}
    \frac{1}{2 \gamma_t} \Vert x^{t+1} - x^* \Vert^2 - \frac{1}{2 \gamma_t} \Vert x^t - x^* \Vert^2
    =
    \frac{-1}{2 \gamma_t}\Vert x^{t+1} - x^t \Vert^2
    - \langle \frac{x^t - x^{t+1}}{\gamma_t}, x^{t+1} - x^* \rangle.
\end{equation*}
Since $x^{t+1} = \prox_{\gamma_t g}(x^t - \gamma_t \nabla f_j(x^t))$, we know from Lemma \ref{L:prox characterization subdifferential}  that $\frac{x^t - x^{t+1}}{\gamma_t} \in \nabla f_j(x^t) + \partial g(x^{t+1})$.
So there exists some $\eta^{t+1} \in \partial g(x^{t+1})$ such that
\begin{align}
    \label{cspg4}
    \frac{1}{2 \gamma_t} \Vert x^{t+1} - x^* \Vert^2& - \frac{1}{2 \gamma_t} \Vert x^t - x^* \Vert^2 \\
    & = \notag
    \frac{-1}{2 \gamma_t}\Vert x^{t+1} - x^t \Vert^2
    - \langle  \nabla f_j(x^t) + \eta^{t+1}, x^{t+1} - x^* \rangle 
    \\
    & = \notag
    \frac{-1}{2 \gamma_t}\Vert x^{t+1} - x^t \Vert^2
    - \langle  \nabla f_j(x^t) - \nabla f(x^t), x^{t+1} - x^* \rangle 
    - \langle  \nabla f(x^t) + \eta^{t+1}, x^{t+1} - x^* \rangle .
\end{align}
We decompose the last term of \eqref{cspg4} as
\begin{equation*}
    - \langle  \nabla f(x^t) + \eta^{t+1}, x^{t+1} - x^* \rangle 
    =
    - \langle \eta^{t+1}, x^{t+1} - x^* \rangle 
    - \langle \nabla f(x^t), x^{t+1} - x^t \rangle
    +
    \langle \nabla f(x^t), x^* - x^t \rangle.
\end{equation*}
For the first term in the above we can use the fact that $\eta^{t+1} \in \partial g(x^{t+1})$ to write
\begin{equation}\label{cspg5}
    - \langle \eta^{t+1}, x^{t+1} - x^* \rangle
    =  \langle \eta^{t+1}, x^* - x^{t+1} \rangle
    \leq g(x^*) - g(x^{t+1}).
\end{equation}
On the second term we can use the fact that $f$ is $L$-smooth and \eqref{eq:smoothnessfunc} to write
\begin{equation}\label{cspg6}
    - \langle \nabla f(x^t), x^{t+1} - x^t \rangle
    \leq
    \frac{L}{2}\Vert x^{t+1} - x^t \Vert^2 + f(x^t) - f(x^{t+1}).
\end{equation}
On the last term we can use the convexity of $f$ and \eqref{eq:conv} to write
\begin{equation}\label{cspg7}
    \langle \nabla f(x^t), x^* - x^t \rangle \leq f(x^*) - f(x^t).
\end{equation}
By combining \eqref{cspg5}, \eqref{cspg6}, \eqref{cspg7}, and using the fact that $\gamma_t L \leq \gamma_0 L_{\max} \leq 1$, we obtain 
\begin{eqnarray}
& & \notag
    \frac{1}{2 \gamma_t} \Vert x^{t+1} - x^* \Vert^2 - \frac{1}{2 \gamma_t} \Vert x^t - x^* \Vert^2 \\
    &\leq & \notag
    \frac{-1}{2 \gamma_t}\Vert x^{t+1} - x^t \Vert^2
    + \frac{L}{2}\Vert x^{t+1} - x^t \Vert^2
    - (F(x^{t+1}) - \inf F)- \langle  \nabla f_j(x^t) - \nabla f(x^t), x^{t+1} - x^* \rangle\\
    & \leq & \label{SPGD iterates decreasing}
    - (F(x^{t+1}) - \inf F)- \langle  \nabla f_j(x^t) - \nabla f(x^t), x^{t+1} - x^* \rangle.
\end{eqnarray}

We now have to control the last term of \eqref{SPGD iterates decreasing}, in expectation.
To shorten the computation we temporarily introduce the operators
\begin{align*}
T & \eqdef I_d - \gamma_t \nabla f, \nonumber\\
\hat T & \eqdef I_d - \gamma_t \nabla f_j.\label{eq:GandGhatdef}
\end{align*}
Notice in particular that $x^{t+1} = \prox_{\gamma_t g}(\hat T(x^t))$.
We have that
\begin{eqnarray}
    - \langle  \nabla f_j(x^t) - \nabla f(x^t), x^{t+1} - x^* \rangle 
    & = & 
    - \langle  \nabla f_j(x^t) - \nabla f(x^t), \prox_{\gamma_t g}(\hat T(x^t)) - \prox_{\gamma_t g}(T(x^t)) \rangle  \nonumber
    \\
    & & - \langle  \nabla f_j(x^t) - \nabla f(x^t), \prox_{\gamma_t g}( T(x^t)) - x^* \rangle, \label{eq:telsinoinz4}
\end{eqnarray}
and observe that the last term is, in expectation, equal to zero.
This is due to the fact that $\prox_{\gamma_t g}( T(x^t)) - x^*$ is deterministic when conditioned on $x^t$.
Since we will later on take expectations, we drop this term and keep on going.  As for the first term,
 using the nonexpansiveness of the proximal operator (Lemma \ref{L:prox nonexpansive}),  we have that
\begin{eqnarray*}
    - \langle  \nabla f_j(x^t) - \nabla f(x^t), \prox_{\gamma_t g}(\hat T(x^t)) - \prox_{\gamma_t g}(T(x^t)) \rangle  
    &\leq &  \Vert \nabla f_j(x^t) - \nabla f(x^t) \Vert \Vert \hat T(x^t) - T(x^t) \Vert \\
    & = & \gamma_t \Vert \nabla f_j(x^t) - \nabla f(x^t) \Vert^2.
\end{eqnarray*}
Using the above two bounds in~\eqref{eq:telsinoinz4}  we have proved that (after taking expectation)
\begin{equation*}
    \mathbb{E}\left[ - \langle  \nabla f_j(x^t) - \nabla f(x^t), x^{t+1} - x^* \rangle  \right]
    \leq 
    \gamma_t \mathbb{E}\left[ \Vert \nabla f_j(x^t) - \nabla f(x^t) \Vert^2 \right]
    =
    \gamma_t \mathbb{V}\left[ \nabla f_j(x^t) \right].
\end{equation*}
Injecting the above inequality into \eqref{SPGD iterates decreasing}, we finally obtain 
\begin{equation}\label{spgdcg1}
    \frac{1}{2}\mathbb{E}\left[ \Vert x^{t+1} - x^* \Vert^2 \right] - \frac{1}{2}\mathbb{E}\left[ \Vert x^{t} - x^* \Vert^2\right]
    \leq
    - \gamma_t\mathbb{E}\left[  F(x^{t+1}) - \inf F \right]
    +
    \gamma_t^2 \mathbb{V}\left[ \nabla f_i(x^t) \right].
\end{equation}
To control the variance term $\mathbb{V}\left[ \nabla f_i(x^t) \right]$ we use the variance transfer Lemma \ref{L:PSGD variance transfer convex} with $x=x^t$ and $y=x^*$, which together with Definition \ref{D:gradient solution variance} and Lemma \ref{L:bregman divergence composite} gives
\begin{eqnarray*}
    \mathbb{V}\left[ \nabla f_i(x^t) \ | \ x^t \right] &\leq & 
    4 L_{\max}D_f(x^t;x^*) + 2 \sigma^*_F \\
    & \leq &
    4 L_{\max}\left(F(x^{t})-  \inf F \right) + 2 \sigma^*_F.
\end{eqnarray*}
Taking expectation in the above inequality and inserting it in \eqref{spgdcg1} gives
\begin{align}\label{spgdcg2}
    \frac{1}{2 }\mathbb{E}\left[ \Vert x^{t+1} - x^* \Vert^2 \right] &- \frac{1}{2 }\mathbb{E}\left[ \Vert x^{t} - x^* \Vert^2\right]
    \\
    & \leq  \notag
    - \gamma_t\mathbb{E}\left[  F(x^{t+1}) - \inf F \right]
    +
    4\gamma_t^2 L_{\max}\mathbb{E}\left[F(x^{t})-  \inf F \right] + 2\gamma_t^2 \sigma^*_F. \\
    & \leq
    - \gamma_t\mathbb{E}\left[  F(x^{t+1}) - \inf F \right]
    +
    \frac{\gamma_t}{2} \mathbb{E}\left[F(x^{t})-  \inf F \right] + 2\gamma_t^2 \sigma^*_F,
\end{align}
where in the last inequality we used the fact that $\gamma_t L_{\max} \leq \tfrac18$.
After reorganizing the terms, multiplying by $2$ and using the fact that $\gamma_t$ is decreasing, we obtain
\begin{eqnarray*}
	& & \gamma_t \mathbb{E}\left[  F(x^{t+1}) - \inf F \right] \\
	& \leq &
	\mathbb{E}\left[ \Vert x^{t} - x^* \Vert^2\right]
	- 
	\mathbb{E}\left[ \Vert x^{t+1} - x^* \Vert^2 \right]
	+
	{\gamma_t} \mathbb{E}\left[F(x^{t})-  \inf F \right]
	-
	{\gamma_{t+1}} \mathbb{E}\left[F(x^{t+1})-  \inf F \right]
	+ 4\gamma_t^2 \sigma^*_F.
\end{eqnarray*}
Sum this inequality over $t=0, \dots, T-1$ to obtain, after telescoping terms and canceling trivially negative terms
\begin{equation*}
	\sum_{t=0}^{T-1} \gamma_t \mathbb{E}\left[  F(x^{t+1}) - \inf F \right]
	\leq
	 \Vert x^{0} - x^* \Vert^2
	+
	{\gamma_0} (F(x^{0})-  \inf F)
	+ 4 \sigma^*_F \sum_{t=0}^{T-1} \gamma_t^2.
\end{equation*}
Now divide this inequality by $\sum_{t=0}^{T-1} \gamma_t$, and define $\bar x^T := \tfrac{1}{\sum_{t=0}^{T-1} \gamma_t} \sum_{t=0}^{T-1} \gamma_t x^{t+1}$ so that after using Jensen's inequality we conclude
\begin{eqnarray*}
	\mathbb{E}\left[  F(\bar x^T) - \inf F \right]
	& \leq &
	\tfrac{1}{\sum_{t=0}^{T-1} \gamma_t} \sum_{t=0}^{T-1} \gamma_t \mathbb{E}\left[  F(x^{t+1}) - \inf F \right] \\
	& \leq &
	\frac{\Vert x^{0} - x^* \Vert^2
	+
	{\gamma_0} (F(x^{0})-  \inf F)}{\sum_{t=0}^{T-1} \gamma_t}
	+
	\frac{4 \sigma^*_F \sum_{t=0}^{T-1} \gamma_t^2}{\sum_{t=0}^{T-1} \gamma_t}.
\end{eqnarray*}
\end{proof}
Analogously to Remark~\ref{rem:sgdstep},   different choices for the step sizes $\gamma_t$ allow us to trade off the convergence speed for the constant variance term.  In the next two corollaries we choose a constant and a $1/\sqrt{t}$ step size,  respectively,  followed by a $\mathcal{O}(\varepsilon^2)$ complexity result.

\begin{theorem}\label{T:PSGD convexe smooth constant stpesize}
Let Assumptions \tref{Ass:SPGD nonsmooth sum of convex} and \tref{Ass:SPGD nonsmooth sum of smooth} hold. Let $\sigma^*_F$ be the ~\tref{D:gradient solution variance composite}.
Let $(x^t)_{t\in \mathbb{N}}$ be a sequence generated by the \tref{Algo:Proximal Stochastic GD} algorithm with a constant sequence of stepsizes $\gamma_t \equiv \gamma \leq \frac{1}{8 L_{\max}}$.
Then, for all $T \geq 1$, all $x^* \in {\rm{argmin}}~F$, and $\bar x^T \eqdef \frac{1}{T}\sum_{t=1}^{T} x^{t}$: 
\begin{equation*}
    \mathbb{E}\left[  F(\bar x^T) - \inf F \right]
	\leq 
	\frac{C_0}{\gamma T}
	+
	{4 \gamma \sigma^*_F},
	\quad C_0 = \Vert x^{0} - x^* \Vert^2
	+
	\tfrac{1}{8 L_{\max}} (F(x^{0})-  \inf F).
\end{equation*}
In particular, if for a fixed horizon $T \geq 1$ we set $\gamma = \tfrac{\gamma_0}{\sqrt{T}}$ for some $\gamma_0 \leq \tfrac{1}{8 L_{\max}}$, then
\begin{equation*}
    \mathbb{E}\left[  F(\bar x^T) - \inf F \right]
	\leq 
	\frac{C_0}{\gamma_0 \sqrt{T}}
	+
	\frac{4 \gamma_0 \sigma^*_F }{\sqrt{T}}
	=
	\mathcal{O}\left( \frac{1}{\sqrt{T}} \right).
\end{equation*}
\end{theorem}

\begin{proof}
    This is a direct consequence of Theorem \ref{theo:sgdproxconvex varying stepsizes} since $\sum_{t=0}^{T-1} \gamma_t = \gamma T$ and $\sum_{t=0}^{T-1} \gamma_t^2 = \gamma^2 T$.
\end{proof}

\begin{corollary}[$\mathcal{O}(\varepsilon^{-2})$--Complexity]\label{theo:sgdproxconvex complexity fixed stepsize}
Let Assumptions \tref{Ass:SPGD nonsmooth sum of smooth} and \tref{Ass:SPGD nonsmooth sum of convex} hold. Let $\sigma^*_F$ be the ~\tref{D:gradient solution variance composite}.
Let $(x^t)_{t\in \mathbb{N}}$ be a sequence generated by the \tref{Algo:Proximal Stochastic GD} algorithm with a constant stepsize $\gamma$. 
For every $\varepsilon > 0$, we can guarantee that $\mathbb{E}\left[ F(\bar x^T) - \inf F \right] \leq \varepsilon$, provided that
\begin{equation*}
    \gamma = \frac{\gamma_0}{\sqrt{T}}, \quad  \gamma_0= \min \left\{ \frac{1}{8 L_{\max}}, \sqrt{\frac{C_0}{4 \sigma_F^*}} \right\}
    \quad \text{ and } \quad 
    T \geq \left( 4 \sqrt{C_0 \sigma_F^*} + 8 C_0 L_{\max} \right)^2 \frac{1}{\varepsilon^2},
\end{equation*}
where $\bar x^T = \frac{1}{T}\sum_{t=1}^{T}  x^t$, $C_0 = \Vert x^{0} - x^* \Vert^2
	+
	\tfrac{1}{8 L_{\max}} (F(x^{0})-  \inf F)$ and $x^* \in {\rm{argmin}}~F$. 
\end{corollary}

\begin{proof}
This is a direct consequence of Theorem \ref{T:PSGD convexe smooth constant stpesize} and Lemma \ref{L:complexity meta convex} with $A = C_0$, $B = 4 \sigma_F^*$ and $C = 8L_{\max}$.
\end{proof}

\begin{theorem}\label{T:PSGD convexe smooth vanishing stpesize}
Let Assumptions \tref{Ass:SPGD nonsmooth sum of convex} and \tref{Ass:SPGD nonsmooth sum of smooth} hold. Let $\sigma^*_F$ be the ~\tref{D:gradient solution variance composite}.
Let $(x^t)_{t\in \mathbb{N}}$ be a sequence generated by the \tref{Algo:Proximal Stochastic GD} algorithm with a vanishing sequence of stepsizes $\gamma_t = \tfrac{\gamma_0}{\sqrt{t+1}}$ where $\gamma_0 \leq \frac{1}{8 L_{\max}}$.
Then, for all $T \geq 1$, all $x^* \in {\rm{argmin}}~F$, and $\bar x^T \eqdef \frac{1}{\sum_{t=0}^{T-1} \gamma_t}\sum_{t=0}^{T-1} \gamma_t x^{t+1}$: 
\begin{equation*}
    \mathbb{E}\left[  F(\bar x^T) - \inf F \right]
	\leq 
	\frac{2C_0}{\gamma_0 \sqrt{T}}
	+
	\frac{16 \sigma^*_F \gamma_0 \log(T+1)}{\sqrt{T}}
	=
	\mathcal{O}\left( \frac{\log(T+1)}{\sqrt{T}} \right).
\end{equation*}
\end{theorem}

\begin{proof}
Start considering $\gamma_t = \frac{\gamma}{\sqrt{t+1}}$, and use integral bounds (see Lemma \ref{L:sum integral bounds}) to write 
\begin{equation*}
    \sum_{t=0}^{T-1} \gamma_t
    = \gamma_0 \sum_{t=1}^T \frac{1}{\sqrt{t}}
    \geq
    \frac{\gamma_0}{2} \sqrt{T}
    \quad \text{ and } \quad 
    \sum_{t=0}^{T-1} \gamma_t^2
    =
    \gamma_0^2 \sum_{t=1}^T \frac{1}{t}
    \leq 
    2\gamma_0^2 \log(T+1).
\end{equation*}
Injecting those bounds in the bound of \Cref{theo:sgdproxconvex varying stepsizes}, we obtain
\begin{equation*}
    \mathbb{E}\left[  F(\bar x^T) - \inf F \right]
	\leq 
	\frac{2C_0}{\gamma_0 \sqrt{T}}
	+
	\frac{16 \sigma^*_F \gamma_0 \log(T+1)}{\sqrt{T}}.
\end{equation*}
\end{proof}

\subsection{Complexity for strongly convex functions}

\begin{theorem}\label{theo:sgdprox strongconvex constant stepsizes}
Let Assumptions \tref{Ass:SPGD nonsmooth sum of smooth} and \tref{Ass:SPGD nonsmooth sum of convex} hold, and assume further that $f$ is $\mu$-strongly convex, for $\mu>0$. Let $\sigma^*_F$ be the ~\tref{D:gradient solution variance composite}.
Let $(x^t)_{t\in \mathbb{N}}$ be a sequence generated by the \tref{Algo:Proximal Stochastic GD} algorithm with a constant sequence of stepsizes verifying $0< \gamma \leq \frac{1}{2 L_{\max}}$.
Then, for $x^* = {\rm{argmin}}~F$, for all $t \in \mathbb{N}$,
\begin{equation*}
    \mathbb{E}\left[ \Vert x^t - x^* \Vert^2 \right]
    \leq 
    (1-\gamma \mu)^t \Vert x^0 - x^* \Vert^2 + \frac{2\gamma \sigma_F^*}{\mu}.
\end{equation*}
\end{theorem}

\begin{proof}
In this proof,  we fix $j:= i_t$ to lighten the notations.
Let us start by using the fixed-point property of the \tref{Algo:proximal gradient descent} algorithm (Lemma \ref{L:prox fixed point composite}), together with the nonexpansiveness of the proximal operator (Lemma \ref{L:prox nonexpansive}), to write
\begin{align}
    \Vert x^{t+1} - x^* \Vert^2
    & =  \notag
    \Vert \prox_{\gamma g}(x^t - \gamma \nabla f_j(x^t)) - \prox_{\gamma g}(x^* - \gamma \nabla f(x^*)) \Vert^2
    \\
    & \leq 
    \Vert (x^t - \gamma \nabla f_j(x^t)) - (x^* - \gamma \nabla f(x^*)) \Vert^2 \nonumber
    \\
    & = 
    \Vert x^t - x^* \Vert^2 + \gamma^2 \Vert \nabla f_j(x^t) - \nabla f(x^*) \Vert^2 - 2 \gamma \langle \nabla f_j(x^t) - \nabla f(x^*), x^t - x^* \rangle.  \label{spgdsc1}
\end{align}
Let us analyse the last two terms of the right-hand side of \eqref{spgdsc1}.
For the first term we use the Young's and the triangular inequality together with   Lemma \ref{L:PSGD variance transfer convex} we obtain
\begin{eqnarray}
    \gamma^2 \mathbb{E}_{x^t}\left[   \Vert \nabla f_j(x^t) - \nabla f(x^*) \Vert^2 \right]
    & \leq & \notag
    2 \gamma^2 \mathbb{E}_{x^t}\left[   \Vert \nabla f_j(x^t) - \nabla f_j(x^*) \Vert^2 \right] + 2 \gamma^2 \mathbb{E}\left[   \Vert \nabla f_j(x^*) - \nabla f(x^*) \Vert^2 \right]
    \\
    & = & \label{spgdsc2}
    2 \gamma^2 \mathbb{E}_{x^t}\left[   \Vert \nabla f_j(x^t) - \nabla f_j(x^*) \Vert^2 \right] + 2 \gamma^2 \sigma_F^*
    \\
    & \leq & \notag
    4 \gamma^2 L_{\max} D_f(x^t;x^*) + 2 \gamma^2 \sigma_F^*,
\end{eqnarray}
where  $D_f$ is the divergence of $f$ (see Definition \ref{D:divergence bregman}).
For the second term in~\eqref{spgdsc1} we use the strong convexity of $f$ (Lemma \ref{L:strong convexity differentiable hyperplans}) to write
\begin{eqnarray}
    - 2 \gamma \mathbb{E}_{x^t} \left[  \langle \nabla f_j(x^t) - \nabla f(x^*), x^t - x^* \rangle\right]
    & = & \notag
    - 2 \gamma  \langle \nabla f(x^t) - \nabla f(x^*), x^t - x^* \rangle
    \\
    & = & \label{spgdsc3}
    2 \gamma \langle \nabla f(x^t), x^* - x^t \rangle
    + 2 \gamma \langle \nabla f(x^*), x^t - x^* \rangle
    \\
    & \leq & \notag
    -\gamma \mu \Vert x^t - x^* \Vert^2 - 2 \gamma D_f(x^t;x^*).
\end{eqnarray}
Combining \eqref{spgdsc1}, \eqref{spgdsc2} and \eqref{spgdsc3} and taking expectations gives
\begin{eqnarray*}
    \mathbb{E}\left[ \Vert x^{t+1} - x^* \Vert^2 \right]
    & \leq &
    (1-\gamma \mu) \mathbb{E}\left[ \Vert x^{t} - x^* \Vert^2 \right] + 2 \gamma(2\gamma L_{\max} -1) \mathbb{E}\left[  D_f(x^t;x^*)\right] + 2 \gamma^2 \sigma_F^* \\
    & \leq &
    (1-\gamma \mu) \mathbb{E}\left[ \Vert x^{t} - x^* \Vert^2 \right] + 2 \gamma^2 \sigma_F^*,
\end{eqnarray*}
where in the last inequality we used our assumption that $2\gamma L_{\max} \leq 1$.
Now, recursively apply the above to write
\begin{equation*}
    \mathbb{E}\left[ \Vert x^{t} - x^* \Vert^2 \right]
    \leq 
    (1-\gamma \mu)^t \Vert x^0 - x^* \Vert^2 + 2 \gamma^2 \sigma_F^* \sum_{s=0}^{t-1} (1-\gamma \mu)^s,
\end{equation*}
and conclude by upper bounding this geometric sum using
\begin{equation*}
    \sum_{s=0}^{t-1} (1-\gamma \mu)^s
    =
    \frac{1 - (1-\gamma \mu)^t}{\gamma \mu}
    \leq
    \frac{1}{\gamma \mu}.
\end{equation*}
\end{proof}

\begin{corollary}[$\mathcal{\tilde{O}}(1/\varepsilon)$ Complexity]
Consider the setting of Theorem~\ref{theo:sgdprox strongconvex constant stepsizes}. Let $\varepsilon >0.$ If we set 
\begin{eqnarray*}
    \gamma &= & \min\left\{\frac{\varepsilon}{2} \frac{\mu}{2 \sigma_F^*}, \frac{1}{2L_{\max}} \right\}
\end{eqnarray*}
then
\begin{equation*}
    t\geq \max\left\{ \frac{1}{\varepsilon}\frac{4 \sigma_F^*}{\mu^2}, \; \frac{2L_{\max}}{\mu} \right\}\log\left(\frac{2\norm{x^0-x^*}^2}{\varepsilon}\right) \quad \implies \quad  \norm{x^t-x^*}^2 \leq \varepsilon
\end{equation*} 
\end{corollary}

\begin{proof} 
Applying Lemma~\ref{lem:linear_pls_const} with $A = \frac{2 \sigma_F^*}{\mu}$, $C = 2L_{\max}$ and $\alpha_0 = \norm{x^0-x^*}^2$ gives the result. 
\end{proof}

\subsection{Bibliographic notes}
In proving Theorem~\ref{theo:sgdproxconvex varying stepsizes}  we simplified the more general Theorem 3.3 in~\cite{Khaled-unified}, which gives a convergence rate for several stochastic proximal algorithms for which proximal SGD is one special case. 
In the smooth and strongly convex case the paper~\cite{gorbunov2019unified}  gives a general theorem for the convergence of stochastic proximal algorithms which includes proximal SGD. For adaptive stochastic proximal methods see~\cite{AsiD19}.

\section{Stochastic Proximal Point}

Here we consider the problem \tref{Pb:stochastic functions} of minimizing a stochastic function $f(x) = \mathbb{E}_\mathcal{D}\left[ f_{\xi}(x) \right]$ by the means of the stochastic proximal point method. This algorithm simply computes, at every iteration, a proximal step with respect to a sampled function $f_\xi$.
For a recap on what is a proximal operator and its properties, see \Cref{S:nonsmooth theory:proximal operator}.

\begin{algorithm}[SPP]\label{Algo:Stoch-Proximal-Point}
Let Assumption \tref{Ass:expectation of convex} hold.
Let $x^0 \in \mathbb{R}^d$, and let $\gamma_t>0$ be a sequence of step sizes.
The \textbf{Stochastic Proximal Point (SPP)} algorithm defines a sequence $(x^t)_{t \in \mathbb{N}}$ satisfying
\begin{align*}
    \xi_t & \in \mathbb{R}^q \qquad \qquad  \mbox{Sampled i.i.d. $\xi_t \sim \mathcal{D}$} \nonumber\\
    x^{t+1} & = \prox_{\gamma_t f_{\xi_t}}(x^t).
\end{align*}
\end{algorithm}

\subsection{Convergence for convex Lipschitz functions}

In Theorem \ref{T:SPP CV convex lipschitz general stepsize} we give a bound for the \tref{Algo:Stoch-Proximal-Point} algorithm, for general stepsizes.
Note that obtained bound is the same,  down to the constant,  as the one for \tref{Algo:Stochastic Subgradient Descent} in~\Cref{theo:stochgradB}.
We then specialize our bound for constant stepsizes in \Cref{T:SPP CV convex lipschitz constant stepsize} and deduce a $\mathcal{O}(\tfrac{1}{\varepsilon})$ complexity rate in \Cref{T:SPP complexity convex lipschitz}.
We then get a $\mathcal{O}\left( \frac{\log(T+1)}{\sqrt{T}} \right)$ convergence rate with  vanishing stepsizes in \Cref{T:SPP CV convex lipschitz vanishing stepsize}.

\begin{theorem}\label{T:SPP CV convex lipschitz general stepsize}
Let Assumptions \tref{Ass:expectation of convex}  and \tref{Ass:expectation of lipschitz} hold. 
Consider $(x^t)_{t \in \mathbb{N}}$ a sequence generated by the \tref{Algo:Stoch-Proximal-Point} algorithm,  with  stepsizes $\gamma_t  >0$.
Then for every $T \geq 1$ and $\bar{x}^T \eqdef \frac{1}{\sum_{t=0}^{T-1}\gamma_t}\sum_{t=0}^{T-1} \gamma_t x_t$ we have
\begin{eqnarray*}\label{eq:SSP CV convex bounded general stepsize}
 \E{f(\bar{x}^{T}) - \inf f } & \leq &   \frac{\Vert x^0 - x^* \Vert^2}{2\sum_{t=0}^{T-1} \gamma_t}  + \frac{G^2 \sum_{t=0}^{T-1} \gamma_t^2}{2 \sum_{t=0}^{T-1} \gamma_t} .
\end{eqnarray*}
\end{theorem}

\begin{proof}
Let us start by looking at  $\Vert x^{t+1} - x^* \Vert^2 -  \Vert x^t - x^* \Vert^2$. 
Since we just compare $x^t$ to $x^{t+1}$, to lighten the notations we will note $f_t$ instead of $f_{\xi_t}$. 
Expanding the squares, we have that
\begin{equation*}
    \frac{1}{2 \gamma_t} \Vert x^{t+1} - x^* \Vert^2 - \frac{1}{2 \gamma_t} \Vert x^t - x^* \Vert^2
    =
    \frac{-1}{2 \gamma_t}\Vert x^{t+1} - x^t \Vert^2
    - \langle \frac{x^t - x^{t+1}}{\gamma_t}, x^{t+1} - x^* \rangle.
\end{equation*}
Since $x^{t+1} = \prox_{\gamma_t f_t}(x^t)$, we know after applying optimality conditions in the definition of the prox (see Lemma \ref{L:prox characterization subdifferential})  that $\eta^{t+1} := \frac{x^t - x^{t+1}}{\gamma_t} \in \partial f_t(x^{t+1})$.
This fact, together with using the convexity of $f_t$ and the definition of subgradients (recall Definition \ref{D:subdifferential convex}), allows us to write
\begin{eqnarray}
	\frac{1}{2 \gamma_t} \Vert x^{t+1} - x^* \Vert^2
	 - \frac{1}{2 \gamma_t} \Vert x^t - x^* \Vert^2 
    & = & \label{sppcl1}
    \frac{-1}{2 \gamma_t}\Vert x^{t+1} - x^t \Vert^2
    + \langle   \eta^{t+1},  x^* - x^{t+1} \rangle \\
    & \leq & \notag
    \frac{-1}{2 \gamma_t}\Vert x^{t+1} - x^t \Vert^2
    + f_t(x^*) - f_t(x^{t+1}). 
\end{eqnarray}
Use now the Lipschitzness of $f_t$ to bound (we note $\mathbb{E}_t$ the expectation conditioned to $x^t$):
\begin{eqnarray*}
	\mathbb{E}_t \left[ f_t(x^*) - f_t(x^{t+1}) \right]
	&=&
	\inf f -  f(x^t) + \mathbb{E}_t \left[f_t(x^t) - f_t(x^{t+1}) \right] \\
	& \leq &
	\inf f -  f(x^{t}) + G \mathbb{E}_t \left[ \Vert x^t - x^{t+1} \Vert \right] \\
	& = & \inf f - f(x^{t}) + G \delta_t,
\end{eqnarray*}
where in the last equality we noted $\delta_t := \mathbb{E}_t \left[ \Vert x^{t+1} - x^t \Vert \right]$.
Take now the expectation conditioned to $x^t$ in \eqref{sppcl1}, multiply by $\gamma_t$, and use the above bound to write
\begin{eqnarray*}
	\frac{1}{2 \gamma_t} \mathbb{E}_t \left[  \Vert x^{t+1} - x^* \Vert^2 \right]
	 - \frac{1}{2} \mathbb{E}_t \left[ \Vert x^t - x^* \Vert^2 \right]
	& \leq &
	\frac{-1}{2\gamma_t} \delta_t^2
	+\inf f - f(x^{t}) + G \delta_t \\
	& \leq &
	\inf f - f(x^{t}) + \frac{\gamma_t G^2}{2},
\end{eqnarray*}
where in the last inequality we simply used the fact that the second order polynomial $aX^2 +bX$ with $a<0$ is upper bounded by $\tfrac{-b^2}{4a}$.
We now take expectation on this inequality and multiply it by $\gamma_t$ to obtain, after reorganizing terms:
\begin{equation*}
	\gamma_t \mathbb{E} \left[ f(x^t) - \inf f \right]
	\leq
	\frac{1}{2} \mathbb{E}_t \left[ \Vert x^t - x^* \Vert^2 \right]
	- \frac{1}{2} \mathbb{E}_t \left[  \Vert x^{t+1} - x^* \Vert^2 \right] + \frac{\gamma_t^2 G^2}{2}.
\end{equation*}
Sum this over $t=0, \dots, T-1$, simplify the telescopic sum and cancel trivially positive terms to write
\begin{equation*}
	\sum_{t=0}^{T-1} \gamma_t \mathbb{E} \left[ f(x^t) - \inf f \right]
	\leq
	\frac{1}{2} \Vert x^0 - x^* \Vert^2 + \frac{G^2}{2} \sum_{t=0}^{T-1} \gamma_t^2.
\end{equation*}
Finally, divide by $\sum_{t=0}^{T-1} \gamma_t$ and use Jensen's inequality (thanks to the convexity of $f$) to conclude
\begin{equation*}
\mathbb{E} \left[ f(\bar x^T) - \inf f \right]
\leq
\mathbb{E} \left[\frac{1}{\sum_{t=0}^{T-1} \gamma_t} \sum_{t=0}^{T-1} \gamma_t (f(x^t) - \inf f ) \right]
\leq 
\frac{\Vert x^0 - x^* \Vert^2}{2\sum_{t=0}^{T-1} \gamma_t}  + \frac{G^2 \sum_{t=0}^{T-1} \gamma_t^2}{2 \sum_{t=0}^{T-1} \gamma_t} .
\end{equation*} 
\end{proof}

\begin{theorem}\label{T:SPP CV convex lipschitz constant stepsize}
Let Assumptions~\tref{Ass:expectation of convex} and \tref{Ass:expectation of lipschitz} hold. 
Consider $(x^t)_{t \in \mathbb{N}}$ a sequence generated by the \tref{Algo:Stoch-Proximal-Point} algorithm, with a constant  stepsize $\gamma_t \equiv \gamma >0$.
Then for every $T \geq 1$ and  $\bar{x}^T \eqdef \frac{1}{T}\sum_{t=0}^{T-1}  x_t$ we have
\begin{eqnarray*}
\EE{}{f(\bar{x}^{T})- \inf f}
    \leq 
    \frac{\norm{x^{0}-x^*}^2 }{2 \gamma T} +\frac{\gamma  G^2}{2}.
\end{eqnarray*}
In particular, for a fixed horizon $T \geq 1$ and $\gamma = \frac{1}{\sqrt{T}}$, we see that
\begin{eqnarray*}
\EE{}{f(\bar{x}^{T})- \inf f}
    \leq 
    \frac{\norm{x^{0}-x^*}^2 + G^2 }{2 \sqrt{T}} = \mathcal{O}\left( \frac{1}{\sqrt{T}} \right).
\end{eqnarray*}
\end{theorem}

\begin{proof}
    Apply Theorem \ref{T:SPP CV convex lipschitz general stepsize} with $\sum_{t=0}^{T-1} \gamma_t = \gamma T$ and   $\sum_{t=0}^{T-1} \gamma_t^2 = \gamma^2 T$.
\end{proof}

\begin{corollary}\label{T:SPP complexity convex lipschitz}
Consider the setting of Theorem \ref{T:SPP CV convex lipschitz constant stepsize}.
For every $\varepsilon >0$ we can guarantee that $\EE{}{f(\bar{x}^{T})- \inf f} \leq \varepsilon$ provided that
\begin{equation*}
	\gamma = \frac{\norm{x^{0}-x^*}}{G\sqrt{T}}
	\quad \text{ and } \quad 
    T \geq \frac{\norm{x^{0}-x^*}^2 G^2}{\varepsilon^2}.
\end{equation*}
\end{corollary}

\begin{proof}
This is a direct consequence of Theorem \ref{T:SPP CV convex lipschitz constant stepsize} and Lemma \ref{L:complexity meta convex}, with $A = \tfrac12 \Vert x^0 - x^* \Vert^2$, $B = \tfrac12 G^2$ and $C=0$.
\end{proof}

\begin{theorem}\label{T:SPP CV convex lipschitz vanishing stepsize}
    Let Assumptions~\tref{Ass:expectation of convex} and \tref{Ass:expectation of lipschitz} hold. 
Consider $(x^t)_{t \in \mathbb{N}}$ a sequence generated by the \tref{Algo:Stoch-Proximal-Point} algorithm, with a sequence of stepsizes $\gamma_t \eqdef \frac{\gamma_0}{\sqrt{t+1}}$ for some $\gamma_0 >0$. 
    We have for every $T \geq 1$ and for $\bar{x}^T \eqdef \frac{1}{\sum_{t=0}^{T-1}\gamma_t}\sum_{t=0}^{T-1} \gamma_t x_t$ that
 \begin{equation*}
\EE{}{f(\bar{x}^{T}) -\inf f }
\leq 
\frac{\norm{x^{0}-x^*}^2}{\gamma_0 \sqrt{T}}
+
\frac{2\gamma_0 G^2 \log(T+1)}{\sqrt{T}}
=
\mathcal{O}\left( \frac{\log(T+1)}{\sqrt{T}} \right).
\end{equation*}
\end{theorem}

\begin{proof}
Start considering $\gamma_t = \frac{\gamma}{\sqrt{t+1}}$, and use integral bounds (see Lemma \ref{L:sum integral bounds}) to write 
\begin{equation*}
    \sum_{t=0}^{T-1} \gamma_t
    = \gamma_0 \sum_{t=1}^T \frac{1}{\sqrt{t}}
    \geq
    \frac{\gamma_0}{2} \sqrt{T}
    \quad \text{ and } \quad 
    \sum_{t=0}^{T-1} \gamma_t^2
    =
    \gamma_0^2 \sum_{t=1}^T \frac{1}{t}
    \leq 
    2\gamma_0^2 \log(T+1).
\end{equation*}
Injecting those bounds in the bound of \Cref{T:SPP CV convex lipschitz general stepsize}, we obtain
\begin{equation*}
    \EE{}{f(\bar{x}^{T})- \inf f}
    \leq 
    \frac{\norm{x^{0}-x^*}^2 + \sum_{t=0}^{T-1}\gamma_t^2 G^2 }{2\sum_{t=0}^{T-1}\gamma_t}
    \leq
    \frac{\norm{x^{0}-x^*}^2 + 2\gamma_0^2 G^2 \log(T+1)}{\gamma_0 \sqrt{T}}.
\end{equation*}
\end{proof}

\subsection{Bibliographic notes}

The proof of \Cref{T:SPP CV convex lipschitz general stepsize} is adapted from the more general Theorem~4.4 in~\cite{Davis2019}.

\printbibliography

\newpage
\appendix

\section{Appendix}

\subsection{Converting Rates into Complexity}

\begin{lemma}\label{L:complexity meta convex}
Let $(\alpha_t)_{t \in \mathbb{N}} \subset [0,+\infty)$ be a sequence satisfying
\begin{equation*}
	\alpha_t \leq \frac{A}{\gamma t} + \gamma B,
\end{equation*}
where $A, B\geq 0$ and $\gamma \in (0, \tfrac{1}{C}]$ is a free parameter, with $C \geq 0$ (when $C=0$ we replace $\tfrac{1}{C}$ with $+\infty$).
For every $\varepsilon > 0$, $\alpha_t \leq \varepsilon$ can be guaranteed if
\begin{equation*}
	\gamma = \frac{\gamma_0}{\sqrt{t}},
	\quad
	\gamma_0 = \min \left\{ \frac{1}{C}, \sqrt{\frac{A}{B}} \right\}
	\quad \text{ and } \quad 	
	t \geq \left( 2 \sqrt{AB} + AC \right)^2 \frac{1}{\varepsilon^2}.
\end{equation*}
\end{lemma}

\begin{proof}
	Take $\gamma = \tfrac{\gamma_0}{\sqrt{t}}$ with $\gamma_0 \leq \tfrac{1}{C}$, then
	\begin{equation*}
	\alpha_t \leq \frac{A}{ \gamma_0 \sqrt{t}} + \frac{B \gamma_0}{ \sqrt{t}}
	=
	\left( \frac{A}{\gamma_0} + B \gamma_0 \right) \frac{1}{\sqrt{t}}.
	\end{equation*}
	We can see that $\gamma_0 \mapsto \tfrac{A}{\gamma_0} + B \gamma_0 $ is minimal over $(0, C^{-1}]$ when $\gamma_0 = \min \{ \tfrac{1}{C} , \sqrt{\tfrac{A}{B}} \}$.
	In this case we see that
	\begin{equation*}
	\alpha_t
	\leq
	\left( {A} \max \{ C, \sqrt{\frac{B}{A}} \} + B \sqrt{\frac{A}{B}} \right) \frac{1}{\sqrt{t}}
	\leq
	\left( 2 \sqrt{AB} + AC \right)\frac{1}{\sqrt{t}},
	\end{equation*}
	where in the last inequality we used the fact that $\max \{ a,b \} \leq a+b$.
	Therefore $\alpha_t \leq \varepsilon$ is guaranteed provided that $\sqrt{t} \geq \left( 2 \sqrt{AB} + AC \right) \frac{1}{\varepsilon}$.
\end{proof}

The following lemma was copied from Lemma 11 in~\cite{Gowerthesis}.

\begin{lemma} \label{lem:itercomplex}
Consider the sequence $(\alpha_k)_k \in \R_+$ of positive scalars that converges to zero according to
\begin{equation}\label{eq:alphaconv} \alpha_k \leq \rho^k \, \alpha_0,\end{equation}
where $\rho \in [0, 1[.$
For a given $1>\varepsilon >0$ we have that
\begin{equation} \label{C1eq:itercomplex}k\geq \frac{1}{1-\rho} \log\left(\frac{\alpha_0}{\varepsilon}\right)  \quad \Rightarrow \quad \alpha_k \leq \varepsilon.
\end{equation}
\end{lemma}
\begin{proof}
First note that if $\rho=0$ the result follows trivially. Assuming $\rho \in ]0,\,1[$,
rearranging~\eqref{eq:alphaconv} and applying the logarithm to both sides gives
\begin{equation} \label{eq:logalphaconv}
 \log\left(\frac{\alpha_0}{\alpha_k}\right) \geq   k  \log\left(\frac{1}{\rho}\right).
 \end{equation}
Now using that
\begin{equation}\label{eq:logineq}
\frac{1}{1-\rho} \log\left(\frac{1}{\rho}\right) \geq 1,
\end{equation}
for all $\rho \in ]0,1[$ and assuming that
\begin{equation}\label{eq:kiterassump}
k\geq \frac{1}{1-\rho} \log\left(\frac{\alpha_0}{\varepsilon}\right) ,
\end{equation}
 we have that
 \begin{eqnarray*}
  \log\left(\frac{\alpha_0}{\alpha_k}\right) & \overset{\eqref{eq:logalphaconv}}{\geq}&
  k \log\left(\frac{1}{\rho}\right)  \\
  & \overset{\eqref{eq:kiterassump}}{\geq} &  \frac{1}{1-\rho} \log\left(\frac{1}{\rho}\right) \log\left(\frac{\alpha_0}{\varepsilon}\right)\\
  & \overset{\eqref{eq:logineq}}{\geq} &
 \log\left(\frac{\alpha_0}{\varepsilon}\right) 
 \end{eqnarray*}

Applying exponentials to the above inequality gives~\eqref{C1eq:itercomplex}.
\end{proof}

As an example of the use this lemma, consider the sequence of random vectors $(Y^k)_k$ for which the expected norm converges to zero according to
\begin{equation*}\label{ch:one:eq:expnormconv}
\E{\norm{Y^k}^2} \leq \rho^k \norm{Y^0}^2.
\end{equation*}
 Then applying Lemma~\ref{lem:itercomplex} with $\alpha_k = \E{\norm{Y^k}^2}$ for a given $1>\varepsilon>0$ states that
 \[k \geq \frac{1}{1-\rho} \log\left(\frac{1}{\varepsilon}\right) \quad \Rightarrow \quad \E{\norm{Y^k}^2} \leq \varepsilon\, \norm{Y^0}^2.\]

 \begin{lemma}\label{lem:linear_pls_const}
Consider the recurrence given by
\begin{eqnarray}\label{eq:linear_plus_const}
    \alpha_t \leq (1-\gamma \mu)^t \alpha_0 + A \gamma,
\end{eqnarray}
where $\mu>0$ and $A,C\geq 0$ are given constants and $\gamma \in ]0,\frac{1}{C}]$ is a free parameter.
If
\begin{eqnarray}\label{eq:gammalinear_pl_const}
    \gamma =  \min\left\{\frac{\varepsilon}{2A}, \; \frac{1}{C} \right\}
\end{eqnarray}
then 
\[t\geq \max\left\{\frac{1}{\varepsilon}  \frac{2A}{\mu }, \frac{C}{\mu}\;  \right\}\log\left(\frac{2\alpha_0}{\varepsilon}\right) \quad \implies \quad    \alpha_t \leq \varepsilon.\]
 \end{lemma}
 
\begin{proof}  First we restrict $\gamma$ so that the second term in~\eqref{eq:linear_plus_const} is less than $\varepsilon/2$, that is
\[ A\gamma \leq \frac{\varepsilon}{2} \quad \implies \quad \gamma \leq \frac{\varepsilon}{2A} . \]
Thus we set $\gamma$ according to~\eqref{eq:gammalinear_pl_const} to also satisfy the constraint that $\gamma \leq \frac{1}{C}.$

Furthermore we want
\[(1-\mu \gamma )^t \alpha_0 \leq  \frac{\varepsilon}{2}.\]
Taking logarithms and re-arranging the above means that we want
\begin{equation}
 \log\left(\frac{2\alpha_0}{  \varepsilon }\right) \leq t\log\left(\frac{1}{1 - \gamma \mu }\right).\label{eq:acb378babuhi3}
\end{equation}
Now using that $\log\left(\frac{1}{\rho}\right) \geq 1-\rho,$ with $\rho = 1-\gamma \mu \in ]0,1]$, we see that for \eqref{eq:acb378babuhi3} to be true, it is enough to have
 \[t\geq \frac{1}{\mu \gamma }\log\left(\frac{2\alpha_0}{\varepsilon}\right).\]
Substituting in $\gamma$ from~\eqref{eq:gammalinear_pl_const} gives
\[t\geq \max\left\{\frac{1}{\varepsilon}  \frac{2A}{\mu }, \frac{C}{\mu}\;  \right\}\log\left(\frac{2\alpha_0}{\varepsilon}\right).\]
\end{proof}

\subsection{A collection of simple but technical facts}

\begin{lemma}[A nonconvex PŁ function]\label{L:PL nonconvex example}
Let $f(t) = t^2 + 3\sin(t)^2$.
Then $f$ is $\mu$-Polyak-Łojasiewicz with $\mu = \frac{1}{40}$, while not being convex.
\end{lemma}

\begin{proof}
The fact that $f$ is not convex follows directly from the fact that $f''(t) = 2 + 6\cos(2t)$  can be nonpositive, for instance $f''(\frac{\pi}{2}) = -4$.
To prove that $f$ is PŁ, start by computing $f'(t) = 2t + 3 \sin(2t)$, and $\inf f =0$.
Therefore we are looking for a constant $\alpha = 2 \mu >0$ such that
\begin{equation*}
    \text{ for all $t \in \mathbb{R}$, } \quad
    (2t + 3 \sin(2t))^2 \geq \alpha \left( t^2 + 3 \sin(t)^2 \right).
\end{equation*}
Using the fact that $\sin(t)^2 \leq t^2$, we see that it is sufficient  to find $\alpha>0$ such that
\begin{equation*}
    \text{ for all $t \in \mathbb{R}$, } \quad
    (2t + 3 \sin(2t))^2 \geq 4\alpha  t^2.
\end{equation*}
Now let us introduce $X=2t$, $Y=3\sin(2t)$, so that the above property is equivalent to
\begin{equation*}
    \text{ for all $(X,Y) \in \mathbb{R}^2$ such that $Y = 3 \sin(X)$, } \quad
    (X + Y)^2 \geq \alpha  X^2.
\end{equation*}
It is easy to check whenever the inequality $(X + Y)^2 \geq \alpha  X^2$ is verified or not:
\begin{equation}\label{plnc1}
    (X + Y)^2 < \alpha  X^2
    \Leftrightarrow
    \begin{cases}
        X > 0 \text{ and } -(1+\sqrt{\alpha})X < Y < -(1-\sqrt{\alpha})X \\
        \text{ or} \\
        X < 0 \text{ and }  -(1-\sqrt{\alpha})X < Y < -(1+\sqrt{\alpha})X
    \end{cases}
\end{equation}
Now we just need to make sure  that the curve $Y = 3 \sin(X)$ violates those conditions for $\alpha$ small enough.
We will consider different cases depending on the value of $X$:
\begin{itemize}
    \item If $X \in [0,\pi]$, we have $Y = 3 \sin(X) \geq 0 > -(1 - \sqrt{\alpha}) X$, provided that $\alpha <1$.
    \item On $[\pi, \frac{5}{4}\pi]$ we can use the inequality $\sin(t) \geq \pi - t$.
    One way to prove this inequality is to use the fact that $\sin(t)$ is convex on $[\pi,2\pi]$ (its second derivative is $-\sin(t) \geq 0$), which implies that $\sin(t)$ is greater than its tangent at $t_0=\pi$, whose equation is $\pi - t$. This being said, we can write (remember that $X \in [\pi,\frac{5}{4}\pi]$ here): 
    \begin{equation*}
        Y = 3 \sin(X) \geq 3(\pi - X) \geq 3( \pi - \frac{5}{4}\pi) = -\frac{3}{4}\pi
        >
        -(1-\sqrt{\alpha}) \pi \geq -(1-\sqrt{\alpha}) X,
    \end{equation*}
    where the strict inequality is true whenever $\frac{3}{4} < 1-\sqrt{\alpha} \Leftrightarrow \alpha < \frac{1}{16} \simeq 0.06$.
    \item If $X \in [\frac{5}{4}\pi, +\infty[$, we simply use the fact that
    \begin{equation*}
        Y = 3 \sin(X) \geq -3 > -(1-\sqrt{\alpha}) \frac{5}{4}\pi \geq -(1-\sqrt{\alpha}) X,
    \end{equation*}
    where the strict inequality is true whenever $3 < (1-\sqrt{\alpha}) \frac{5}{4}\pi \Leftrightarrow \alpha < \left( 1 - \frac{12}{5 \pi} \right)^2 \simeq 0.055$.
    \item If $X \in ]-\infty,0]$, we can use the exact same arguments (use the fact that sine is a odd function) to obtain that $Y < -(1-\sqrt{\alpha}) X$.
\end{itemize}
In every cases, we see that \eqref{plnc1} is violated when $Y = 3 \sin(X)$ and $\alpha =0.05$, which allows us to conclude that $f$ is $\mu$-PŁ with $\mu = \alpha / 2 = 0.025 = 1/40$.
\end{proof}

\begin{lemma}\label{L:sps optimal upper bound}
    The Stochastic Polyak Stepsize (see Algorithm \ref{Algo:SPS for SSD}) minimizes the right-hand side of \eqref{eq:sps optimal upper bound}.
\end{lemma}

\begin{proof}
We want to minimize the right hand side of \eqref{eq:sps optimal upper bound}, which means solving
 \begin{equation*}
\arg\min_{\gamma \geq 0}\; q(\gamma) \;:=\; -2\gamma(f_{\xi_t}(x^t) -f_{\xi_t}(x^*)) + \gamma^2 \norm{  \sgrad_{\xi_t}(x^t)}^2.
 \end{equation*}
If $\sgrad_{\xi_t}(x^t)=0$, then $0 \in \partial f_{\xi_t}(x^t)$, so from this optimality condition (recall \cref{P:Fermat convex nonsmooth}) we deduce  that $x^t$ is a minimizer of $f_{\xi_t}$.
In particular,  $(f_{\xi_t}(x^t) -f_{\xi_t}(x^*)) \leq 0$. 
This means that the solution to our problem is $\gamma=0$,  which coincides with SPS in this case.
If $\sgrad_{\xi_t}(x^t)\neq 0$, then $q$ is a positive definite quadratic function whose unconstrained minimizer is clearly 
\begin{eqnarray*}\label{eq:lohoz9hz4s}
\hat \gamma = \frac{f_{\xi_t}(x^t) -f_{\xi_t}(x^*)}{\norm{  \sgrad_{\xi_t}(x^t)}^2}.
\end{eqnarray*}
Therefore the solution is $\hat \gamma$ if it is positive, and is zero whenever $\hat \gamma \leq 0$.
In other words, the optimal step size is $(\hat \gamma)_+$, which again is exactly the SPS.
\end{proof}

\subsection{Useful inequalities}

\begin{lemma}\label{ex:xpossqovery}
    Let $C = \mathbb{R} \times (0,+\infty) \cup (-\infty, 0] \times \{0\}$ be a convex subset of $\mathbb{R}^2$.
    Let $x_+ := \max\{0,\; x\}$, and let $f : C \to \mathbb{R}$ be defined by
    \begin{equation*}
    f(x,y) =
    \begin{cases}
    	\tfrac{(x_+)^2}{y} & \text{ if } y>0, \\
    	0 & \text{ if } y=0.
    \end{cases}
    \end{equation*}
Then $f$ is convex over $C$.
\end{lemma}

\begin{proof}
    Define $U = \mathbb{R} \times (0,+\infty)\subset C$, which is an open convex set, and let us start by proving that $f$ is convex over $U$.
    To do this, we want to compute its hessian and verify that it is positive semi-definite on $U$.
    We first compute its gradient for every $(x,y) \in U$:
    \begin{equation*}
        \nabla f(x,y) = \left( \frac{2x_+}{y}, \frac{-(x_+)^2}{y^2} \right).
    \end{equation*}
    We can now turn on differentiating $\nabla f$.
    To do so, we take $(x,y) \in U$ and we consider three cases:
    \begin{itemize}
        \item If $x > 0$ : then $\nabla f$ is differentiable at $(x,y)$ with $\nabla f(x,y) = \left( \frac{2x}{y}, \frac{-x^2}{y^2} \right)$, from which we deduce that 
        \begin{equation*}
            \nabla^2 f(x,y) =
            \frac{2}{y^3}
            \begin{pmatrix}
            y^2 & -xy \\ -xy & x^2
            \end{pmatrix}.
        \end{equation*}
        We see that the trace and determinant of this matrix are nonnegative, so we conclude that $\nabla^2 f(x,y)$ is positive semi-definite.
        \item If $x <0$ : then $\nabla f$ locally constant to $(0,0)$, from which we deduce that $\nabla^2 f(x,y) = 0$, which is also positive semi-definite.
        \item If $x=0$ : then $\nabla f$ is not differentiable at $(0,y)$.
        Nevertheless, $\nabla f$ is locally Lipschitz (as a composition and product of elementary locally Lipschitz functions).
        Therefore we can compute its \emph{generalized hessian} $\partial^2 f(x,y)$ (see \cite[Definition 2.1]{HirStrNgu84}), which is the convex hull of  the possible limits of hessians around $(x,y)$:
        \begin{equation*}
            \partial^2 f(x,y) = 
            {\rm{co}}~\left\{ \lim\limits_{n \to +\infty} \nabla^2 f(x_n,y_n) \ : \ (x_n,y_n) \to (x,y) \text{ and $\nabla^2 f(x_n,y_n)$ exists} \right\}. 
        \end{equation*}
        We see that the two possible limits for hessians in a neighbourhood of $(0,y)$ are the ones for which $x_n$ converges to zero with positive (resp. negative) values, that is
        \begin{equation*}
            \lim\limits_{x_n \to 0^+, y_n \to y} \nabla^2 f(x_n,y_n) = 
            \begin{pmatrix}
            \frac{2}{y} & 0 \\ 0 & 0
            \end{pmatrix}
            \quad \text{ and } \quad 
            \lim\limits_{x_n \to 0^-, y_n \to y} \nabla^2 f(x_n,y_n) =
            \begin{pmatrix}
            0 & 0 \\ 0 & 0
            \end{pmatrix}.
        \end{equation*}
        In other words, the generalized hessian at $(0,x)$ is a set of positive semi-definite matrices
        \begin{equation*}
            \partial^2 f(0,x) = \left\{  \begin{pmatrix}
            h & 0 \\ 0 & 0
            \end{pmatrix} \ : \ 0 \leq h \leq  \frac{2}{y} \right\}.
        \end{equation*}
    \end{itemize}
    We have proven that the (generalized) hessian of $f$ is positive semi-definite at every $(x,y) \in U$, so we can deduce that $f$ is convex on $U$ \cite[Example 2.2]{HirStrNgu84}.
    Now we can proceed with the last part of the proof, which is proving the convexity of $f$ over $C$, not only $U$.
    To do so we will simply rely on the definition of convexity: let $z_1,z_2 \in C$, let $\alpha \in (0,1)$, and let us show that
    \begin{equation*}
    f((1-\alpha)z_1 + \alpha z_2) \leq (1-\alpha)f(z_1) + \alpha f(z_2).
    \end{equation*}
    Let us distinguish a few cases.
    \begin{itemize}
    		\item If $z_1 \in U$ and $z_2 \in U$, then we know that the inequality holds, since we already proved the convexity of $f$ over $U$.
    		\item If $z_1 \notin U$ and $z_2 \in U$, then $z_1 = (x_1,0)$ with $x_1 \leq 0$ and $z_2 = (x_2,y_2)$ with $y_2>0$.
    		In that case,    		
    		\begin{eqnarray*}
    		f((1-\alpha)z_1 + \alpha z_2)
    		&=&
    		f((1-\alpha)x_1 + \alpha x_2, \alpha y_2)
    		=
    		\frac{((1-\alpha)x_1 + \alpha x_2)_+^2}{\alpha y_2} \\
    		&\leq & 
    		\frac{(\alpha x_2)_+^2}{\alpha y_2}
    		=
    		\alpha f(z_2) = (1-\alpha)f(z_1) + \alpha f(z_2),
    		\end{eqnarray*}
    		where in the inequality we used the fact that $t \mapsto t_+^2$ is nondecreasing together with the fact that $x_1 \leq 0$.
    		By symmetry, we get the same conclusion if $z_1 \in U$ and $z_2 \notin U$.
    		\item If $z_1 \notin U$ and $z_2 \notin U$, then $z_1 = (x_1,0)$ with $x_1 \leq 0$ and $z_2 = (x_2,0)$ with $x_2 \leq 0$.
    		In that case, we immediately see that
    		\begin{equation*}
    		f((1-\alpha)z_1 + \alpha z_2)
    		=
    		f((1-\alpha)x_1 + \alpha x_2, 0)
    		=
    		0 
    		=  
    		(1-\alpha)f(z_1) + \alpha f(z_2).
    		\end{equation*}
    \end{itemize}
\end{proof}

\begin{lemma}\label{L:log inequality 1/2}
    For every $x > 0$, $\frac{1}{\log(1+x)} \leq \frac{1}{2} + \frac{1}{x}$.
\end{lemma}

\begin{proof}
    This inequality is equivalent to $\log(1+x) \geq \frac{2x}{x+2}$, or again $(x+2)\log(1+x) \geq 2x$.
    Define $\phi : (-1,+\infty) \to \mathbb{R}$ as $\phi(x) = (x+2)\log(1+x)$ and compute its derivatives:
    \begin{equation*}
        \phi'(x) = \log(1+x) +1+\frac{1}{1+x} 
        \quad \text{ and } \quad 
        \phi''(x) = \frac{x}{(1+x)^2}.
    \end{equation*}
    We see that $\phi''(x) \geq 0$ for all $x \geq 0$, so $\phi$ is convex on $[0,+\infty)$.
    So we can use the tangent inequality:
    \begin{equation*}
        \phi(x) \geq \phi(0) + \phi'(0)(x-0) = 0 + 2x = 2x.
    \end{equation*}
    Therefore $\phi(x) \geq 2x$ for all $x \geq 0$, which is what we wanted to prove.
\end{proof}

\begin{lemma}[Sum-Integral bounds]\label{L:sum integral bounds}
    The following bounds hold for every integer $T \geq 1$:
    \begin{equation*}
        \frac{4}{5}\sqrt{T}
        \leq \sum_{t=1}^T \frac{1}{\sqrt{t}}
        \leq 2\sqrt{T} -1,
    \end{equation*}
    \begin{equation*}
        \log(T+1) 
        \leq \sum_{t=1}^T \frac{1}{t}
        \leq 2\log(T+1).
    \end{equation*}
\end{lemma}

\begin{proof}
    We make use of standard arguments. 
    If $\phi : (0,+\infty) \to (0,+\infty)$ is decreasing, then 
    \begin{equation*}
        \int_1^{T+1} \phi(t) dt \leq \sum_{t=1}^T \phi(t) \leq \phi(1) + \int_1^T \phi(t) dt.
    \end{equation*}
    When $\phi(t) = \frac{1}{\sqrt{t}}$, the lower bound becomes
    \begin{equation*}
        \int_1^{T+1} \frac{1}{\sqrt{t}} dt
        =
        \left[ 2\sqrt{t} \right]_1^{T+1}
        =
        2 (\sqrt{T+1} - 1)
        \geq \frac{4}{5} \sqrt{T},
    \end{equation*}
    where in the last inequality we used the fact that $\inf\limits_{t \geq 1} (\sqrt{t+1}-1)/\sqrt{t} = (\sqrt{2}-1)/\sqrt{1} \simeq 0.41 > 2/5$.
    Still with $\phi(t) = \frac{1}{\sqrt{t}}$, the upper bound becomes
    \begin{equation*}
        1 + \int_1^T \frac{1}{\sqrt{t}} dt
        =
        1+ \left[ 2 \sqrt{t} \right]_1^T
        =
        1 + 2\sqrt{T} - 2 = 2 \sqrt{T} -1.
    \end{equation*}
    When $\phi(t) = \frac{1}{t}$, the lower bound becomes
    \begin{equation*}
        \int_1^{T+1} \frac{1}{t} dt
        =
        \left[ \log(t) \right]_1^{T+1}
        =
        \log(T+1) - \log(1) = \log(T+1).
    \end{equation*}
    Still with $\phi(t) = \frac{1}{t}$, the upper bound becomes
    \begin{equation*}
        1 + \int_1^T \frac{1}{t} dt
        =
        1 + \left[ \log(t) \right]_1^T
        =
        1 + \log(T)
        \leq 
        2 \log(T+1)
    \end{equation*}
    where in the last inequality we used the fact that $\sup\limits_{t \geq 1} (1+\log(t))/\log(t+1)  \simeq 1.54 < 2$.
\end{proof}

\begin{lemma}\label{L:sublinear sequence 1/t}
    Let $(a_t)_{t \in \mathbb{N}} \subset [0,+\infty[$ be a sequence such that  $ca_t^2 \leq a_t - a_{t+1}$, for some $c >0$.
    Then for every $t \geq 1$, $a_t \leq \tfrac{1}{ct}$.
\end{lemma}

\begin{proof}
    Let us rewrite the assumption as $a_{t+1} \leq a_t - ca_t^2 = \phi(a_t)$, where $\phi(t) = t - ct^2$.
    A~quick analysis shows that $\phi$ is increasing on $[0,\tfrac{1}{2c}]$ and decreasing on $[\tfrac{1}{2c}, + \infty[$.
    In particular, $\max \phi = \phi(\tfrac{1}{2c}) = \tfrac{1}{4c}$.
    From all this, we see that for every $t \geq 1$ we have $a_t \leq \phi(a_{t-1}) \leq \max \phi = \tfrac{1}{4c}$.
    Now let us prove the claim.
    \begin{itemize}
        \item For $t=1$, this is immediate because $a_1 \leq \tfrac{1}{4c} \leq \tfrac{1}{c}$.
        \item For $t\geq 2$, we use a recursive argument.
        We have $u_t \leq \phi(u_{t-1})$ where $u_{t-1} \leq \tfrac{1}{4c}$, so using the fact that $\phi$ is increasing on $[0,\tfrac{1}{2c}]$, together with $u_{t-1} \leq \tfrac{1}{c(t-1)}$, we obtain
        \begin{eqnarray*}
            u_t 
            & \leq  &
            \phi\left( \frac{1}{c(t-1)} \right)
             =
            \frac{1}{c(t-1)} - c \frac{1}{c^2(t-1)^2}
            =
            \frac{1}{c} \left( \frac{t-2}{(t-1)^2} \right)
            =
            \frac{1}{ct} \left( \frac{t^2 - 2t}{t^2-2t + 1 } \right) \\
            & \leq &
            \frac{1}{ct}.
        \end{eqnarray*}
    \end{itemize}
\end{proof}

\end{document}